\documentclass[11pt]{amsart}
\usepackage[english]{babel}
\usepackage{tikz,tikz-3dplot}
\usepackage{pgf}
\usepackage{pgfplots}
\usepackage{cancel,centernot,slashed}

\definecolor{cof}{RGB}{219,144,71}
\definecolor{pur}{RGB}{186,146,162}
\definecolor{greeo}{RGB}{91,173,69}
\definecolor{greet}{RGB}{52,111,72}

\tdplotsetmaincoords{70}{165}
\usetikzlibrary{calc,arrows,patterns}
\usepackage[utf8]{inputenc}

     \newcommand{\PARENS}[1]{\left(#1\right)}
          \newcommand{\ccases}[1]{\begin{cases}#1\end{cases}}
         \usepackage{amsmath}
        \usepackage{mathrsfs}
        \usepackage{amssymb,amsthm,amscd}
       \usepackage{upgreek}
       \newcommand{\coloneq}{:=}
       \renewcommand{\mathscr}{}
       
       \usepackage[scaled]{helvet} 
       \usepackage{courier} 
       \usepackage{euler} 
       \normalfont
       \usepackage[T1]{fontenc}
\usepackage{bigstrut}
\usepackage{array}
\usepackage{mathtools,arydshln}\mathtoolsset{centercolon}
\usepackage{marvosym}
\usepackage[
pdftex,hyperfootnotes]{hyperref}
\hypersetup{pdfauthor={Gerardo Araznibarreta and Manuel Ma\~{n}as}, pdftitle={Multivariate Laurent orthogonal polynomials and integrable systems}, pdfsubject={Mathematics, Mathematical Physics},pdfkeywords={Multivariate orthogonal Laurent polynomials, Borel--Gauss factorization, quasi-determinants, Christoffel--Darboux kernels, Darboux transformations,  Christoffel formula, integrable hierarchies, Toda equations, KP equations}, pdfcreator={PDF.creator@UCM},pdfproducer={PDF.producer@UCM}}

\usepackage[total={18cm,24.5cm},top=2cm, left=1.5cm]{geometry}
\usepackage[toc,page]{appendix}
\usepackage{tocvsec2}
\usepackage{cancel}
\usepackage{bbm}

\DeclareFontFamily{U}{MnSymbolC}{}
\DeclareSymbolFont{MnSyC}{U}{MnSymbolC}{m}{n}
\DeclareFontShape{U}{MnSymbolC}{m}{n}{
    <-6>  MnSymbolC5
   <6-7>  MnSymbolC6
   <7-8>  MnSymbolC7
   <8-9>  MnSymbolC8
   <9-10> MnSymbolC9
  <10-12> MnSymbolC10
  <12->   MnSymbolC12}{}
\DeclareMathSymbol{\intprod}{\mathbin}{MnSyC}{'267}

\newcommand{\ee}{\boldsymbol{e}}
\newcommand{\n}{\boldsymbol{n}}
\newcommand{\m}{\boldsymbol{m}}
\newcommand{\x}{\boldsymbol{x}}
\newcommand{\y}{\boldsymbol{y}}
\newcommand{\z}{\boldsymbol{z}}
\newcommand{\q}{{\boldsymbol{\alpha}}}

\newcommand{\bUpsilon}{\boldsymbol{\Upsilon}}

\newcommand{\ii}{\operatorname{i}}

\renewcommand{\d}{\operatorname{d}}
\newcommand{\Exp}[1]{\operatorname{e}^{#1}}

\newcommand{\diag}{\operatorname{diag}}

\newcommand{\Ds}{\mathscr D}

\newcommand{\Z}{\mathbb{Z}}
\newcommand{\R}{\mathbb{R}}
\newcommand{\C}{\mathbb{C}}
\newcommand{\I}{\mathbb{I}}
\newcommand{\T}{\mathbb{T}}
\newcommand{\D}{\mathbb{D}}

\newtheorem{pro}{Proposition}[section]
\newtheorem{lemma}{Lemma}[section]
\newtheorem{definition}{Definition}[section]
\newtheorem{theorem}{Theorem}[section]
\newtheorem{cor}{Corollary}[section]
\numberwithin{equation}{section}

\begin{document}

\title[Multivariate orthogonal Laurent polynomials and integrable systems]{Multivariate orthogonal Laurent polynomials\\ and integrable systems}
\author{Gerardo Ariznabarreta}\address{Departamento de F\'{\i}sica Te\'{o}rica II (M\'{e}todos Matem\'{a}ticos de la F\'{\i}sica), Universidad Complutense de Madrid, 28040-Madrid, Spain}
\email{gariznab@ucm.es}
\thanks{GA thanks economical support from the Universidad Complutense de Madrid  Program ``Ayudas para Becas y Contratos Complutenses Predoctorales en Espa\~{n}a 2011"}
\author{Manuel Ma\~{n}as}
\email{manuel.manas@ucm.es}
\thanks{MM thanks economical support from the Spanish ``Ministerio de Econom\'{\i}a y Competitividad" research project MTM2012-36732-C03-01,  \emph{Ortogonalidad y aproximaci\'{o}n; teor\'{\i}a y aplicaciones}. He also want to reckon illuminating discussions with Francisco Marcellan, Ignacio Sols and Jesko Hüttenhain (through MathExchange)}
\maketitle
\begin{abstract}
An ordering  for Laurent polynomials in the algebraic torus $(\C^*)^D$, inspired  by the Cantero--Moral--Velázquez  approach to orthogonal Laurent polynomials in the unit circle,  leads to the construction of a moment matrix for a given Borel measure in the unit torus $\T^D$. The Gauss--Borel factorization of this moment matrix
allows for the construction of multivariate biorthogonal Laurent polynomials in the unit torus which can be expressed as  last  quasi-determinants  of bordered truncations of the moment matrix. The  associated second kind functions are expressed in terms of the Fourier series of the given measure.  Persymmetries and partial persymmetries of the moment matrix are studied and Cauchy integral representations of the second kind functions are found as well as Plemej type formul{\ae}.  Spectral matrices give string equations for the moment matrix which model the three-term relations as well as the Christoffel--Darboux formul{\ae}.

Christoffel type perturbations of the measure given by the multiplication by  Laurent polynomials are studied. Sample matrices on poised sets of nodes, which belong to the algebraic hypersurface of the perturbing Laurent polynomial, are used for the finding of a   Christoffel formula that expresses  the perturbed orthogonal Laurent polynomials in terms  of a last quasi-determinant of a bordered sample matrix constructed in terms of the original orthogonal Laurent polynomials. Poised sets exist only for nice Laurent polynomials which are analyzed from the perspective of Newton polytopes and tropical geometry. Then, an algebraic geometrical characterization of nice Laurent  polynomial perturbation and poised sets is given; full column rankness of the corresponding multivariate Laurent--Vandermonde  matrices and a product of different prime nice Laurent polynomials leads to such sets. Some examples are constructed in terms of perturbations of the Lebesgue--Haar measure.

Discrete and continuous deformations of the measure lead to a Toda type integrable hierarchy, being the corresponding flows  described through Lax and Zakharov--Shabat equations; bilinear equations and vertex operators are found. Varying size matrix nonlinear partial difference and differential equations of the 2D Toda lattice type are shown to be solved by matrix coefficients of the multivariate orthogonal polynomials. The discrete flows are connected with a Gauss--Borel factorization of the Jacobi type matrices and its quasi-determinants alow  for expressions for the multivariate orthogonal polynomials in terms of shifted quasi-tau matrices, which generalize those that relate the Baker functions with  ratios of Miwa shifted $\tau$-functions in the 1D scenario.
It is shown that the discrete and continuous flows are deeply connected and determine nonlinear partial difference-differential equations that involve only one site in the integrable lattice behaving as a Kadomstev--Petviashvili type system. 

\end{abstract}
\newpage
\newpage
\tableofcontents

\section{Introduction}

In this paper we give a theory for multivariate orthogonal Laurent polynomials in the unit torus (MVOLPUT). Our motivation comes from  our previous work in orthogonal polynomials in the unit circle (OLPUC) \cite{carlos}, its matrix extension \cite{MOLPUC} and our recent developments on multivariate orthogonal polynomials in real spaces (MVOPR) \cite{MVOPR,MVOPR-darboux}.
The natural framework for the study of the ring of Laurent polynomials in several variables is the algebraic torus $(\C^*)^D$, $\C\coloneq \C\setminus\{0\}$.
The orthogonal elements appear once  a Borel measure $\d\mu$ with support in the $D$-dimensional  unit torus $\T^D$, $\T\coloneq\{z\in\C: |z|=1\}$, is considered.
Inspired by the Cantero--Moral--Veláquez (CMV) approach \cite{CMV,watkins}, se also \cite{carlos,MOLPUC}, 
we propose an order for the multivariate Laurent polynomials in the algebraic torus $(\C^*)^D$. 
For the  MVOLPUT we find in this paper not only three-term relations and Christoffel--Darboux formul\ae, but  also construct associated second kind functions and partial second kind functions given in terms of the Fourier series of the measure and Cauchy integral transforms, all connected through Plemej relations. Remarkably, we are able to find the extension of the Christoffel formula for  Laurent polynomial perturbations of the measure in terms of a last quasideterminant of a bordered sample matrix constructed from the original MVOLPUT evaluated at a poised set in the algebraic hypersurface of the zeroes of the Laurent polynomial generating the perturbation. 
As the  chosen order does not generate any gradation, for the finding of poised sets we need of what we have called nice Laurent polynomials. The discussion  requires of mathematical ideas connected with tropical geometry as Newton polytopes and also of elementary algebraic geometry. Finally, we are able to construct Toda integrable systems of discrete and continuos type and also find how the Kadomtsev--Petviashvilli flows live in each site of the Toda lattice. 

For a better understanding of the results achieved and its context let us perform  some preliminary comments on the state of the art of the subjects involved.

\subsection{Historical background}

\subsubsection{Orthogonal polynomials in the unit circle} The unit circle   $\T$ is the border of the unit disk  $\mathbb{D}\coloneq \{z \in \C : |z|<1\}$. A complex Borel measure $\mu$ supported in $\T$   is said to be positive definite if it maps measurable sets into non-negative numbers, that in the  absolutely continuous situation (with respect to the Lebesgue--Haar measure  $\frac{\d\theta}{2\pi}$) has the form $w({\theta}) \d \theta$, $\theta\in[0,2\pi)$]. For the positive definite situation the orthogonal polynomials in the unit circle (OPUC) or Szeg\H{o} polynomials are defined as the  monic polynomials  $P_n$ of degree  $n$ that satisfy the  orthogonality relations, $\int_{\T}P_n(z) z^{-k} \d \mu(z)=0$, for $ k=0,1,\dots,n-1$, \cite{Szego}.
Orthogonal polynomials on the real line (OPRL) with support on  $[-1,1]$ are connected with OPUC   \cite{Freud,Berriochoa}.  In the real case the three-term relations provide a tridiagonal matrix, the so called Jacobi operator, while  in the unit circle support case the problem leads to a Hessenberg matrix. OPUC's recursion relation is given in terms of  reverse Szeg\H{o} polynomials $P^*_l(z):=z^l \overline{P_l(\bar z^{-1})}$ and 
 reflection or Verblunsky  coefficients $\alpha_l:=P_l(0)$. The recursion relations for
the Szeg\H{	o} polynomials can be written as $\left(\begin{smallmatrix} P_l \\ P_{l}^* \end{smallmatrix}\right)=
\left(\begin{smallmatrix} z & \alpha_l \\ z \bar \alpha_l & 1 \end{smallmatrix}\right)
\left(\begin{smallmatrix} P_{l-1} \\ P_{l-1}^* \end{smallmatrix}\right)$. 
   Szeg\H{o}'s theorem implies for a nontrivial probability measure $\d\mu$  on $\T$ with
Verblunsky coefficients $\{\alpha_n\}_{n=0}^\infty$ that the corresponding Szeg\H{o}'s polynomials are dense in $L^2(\T,\mu)$ if and
only if $\prod_{n=0}^\infty (1-|\alpha_n)|^2)=0$. For an absolutely continuous probability measure  Kolmogorov's density theorem ensures that density in $L^2(\T,\mu)$ of the  OPUC holds if and only if   the so called Szeg\H{o}'s condition $\int_{\T}\log (w(\theta)\d\theta=-\infty$ is fulfilled, \cite{Simon-S}. 
We refer the reader to Barry Simon's books \cite{Simon-1} and \cite{Simon-2} for a very detailed studied of OPUC.

\subsubsection{Orthogonal Laurent polynomials and CMV} In the context of the strong Stieltjes moment problem the papers \cite{Jones-3,Jones-4}  could be considered as the seed for the study  of orthogonal Laurent polynomials on the real line (OLPRL).
Whenever we have solution of the moment problem we have  Laurent polynomials $\{Q_n\}_{n=0}^\infty$ that fulfill the orthogonality conditions  $\int_{\R}x^{-n+j}Q_n(x)\d \mu(x)=0$ for $j=0,\dots,n-1$.
For the early development of the theory see \cite{Cochran,Diaz,Jones-5} and \cite{Njastad}.  The circle $\T$ context was first considered in \cite{Thron} , see also \cite{Barroso-Vera,CMV,Barroso-Daruis,Barroso-Snake} where  matters  like recursion relations, Favard's theorem, quadrature problems, and Christoffel--Darboux formul{\ae} were treated. 
The CMV \cite{CMV} representation is a hallmark in the study of certain aspects of Szeg\H{o} polynomials. Indeed, despite the set of OLPUC being dense in $L^2(\T,\mu)$ in general this is not true for the OPUC,   \cite{Bul} and \cite{Barroso-Vera}. Now,  the recursion relations for ordinary Szeg\H{o} polynomials and its reverse polynomials is replaced by a  five-term relation similar to the OPRL situation. Alternative or generic orders in the base used to span the space of OLPUC can be found in \cite{Barroso-Snake}.
 Other papers have reviewed and
broadened the study of CMV matrices, see for example \cite{CMV-Simon,Killip}. As was pointed out in \cite{CMV-Simon} the discovery of the advantages of the CMV ordering goes back to previous work \cite{watkins}.

\subsubsection{Multivariate orthogonal polynomials} The monographs \cite{Dunkl} and \cite{xu4} are  highly recommend  references for understanding the state of the art regarding multivariate orthogonal polynomials.  The  recurrence relation for orthogonal polynomials in several variables was studied by Xu in \cite{xu0}, while in \cite{xu1} he linked multivariate orthogonal polynomials with a commutative family of self-adjoint operators and the spectral theorem was used to show the existence of a three-term relation for the orthogonal polynomials. He discusses in \cite{xu2}  how the three-term relation leads to the construction of multivariate orthogonal polynomials and cubature formul{\ae}.  
The analysis of orthogonal polynomials and cubature formul{\ae} on the unit ball, the standard simplex, and the unit sphere  \cite{xu6 } lead to conclude  the strong connection of orthogonal structures and cubature formul{\ae} for these three regions. The paper \cite{xu3} presents a systematic study of the common zeros of polynomials in several variables which are related to higher dimensional quadrature.
 Karlin and McGregor  \cite{karlin}  and Milch \cite{milch} discussed interesting examples of multivariate Hahn and Krawt\-chouk polynomials related to growth birth and death processes. There have been substantial developments since 1975, for instance, the spectral properties of these multivariate Hahn and Krawtchouk polynomials have been studied in \cite{geronimo-iliev}. 
In \cite{geronimo 1} a two-variable positive extension problem for trigonometric polynomials was discussed  ---here the extension is required to be the reciprocal of the absolute value squared of a stable polynomial. This could be understood  as an autoregressive filter design problem for bivariate stochastic processes. 
The authors show that the existence of a solution is equivalent to solving a finite positive definite matrix completion problem where the completion is required to satisfy an additional low rank condition. As a corollary of the main result a necessary and sufficient condition for the existence of a spectral Fejér--Riesz factorization of a strictly positive two-variable trigonometric polynomial is given in terms of the Fourier coefficients of its reciprocal. A spectral matching result is obtained, as well as  inverse formulas for doubly-indexed Toeplitz matrices. 
Geronimo and Woerdeman  used tools including a specific two-variable Kronecker theorem based on certain elements from algebraic geometry, as well as a two-variable Christoffel--Darboux like formula. In \cite{knese} reproducing kernels are used to give and alternative proofs for some of the mentioned results on orthogonal polynomials on the two dimensional torus (and related subjects).  Regarding extension problems on the torus and moment matrices see \cite{bakonyi}. In \cite{diejen} a formula describing the asymptotics of a class of multivariate orthogonal polynomials with hyperoctahedral symmetry as the degree tends to infinity is given. The polynomials under consideration are characterized by a factorized weight function satisfying certain analyticity assumptions. 

\subsubsection{Darboux transformations} In the context of the Sturm--Liouville theory, Gaston Darboux   \cite{darboux}   introduced  these transformations for the first time.  Much later \cite{matveev}  this  transformation was named after Darboux. It has been extensively developed for orthogonal polynomials \cite{grunbaum-haine,yoon, bueno-marcellan1,bueno-marcellan2, marcellan}.
In Geometry,  the theory of transformations of surfaces preserving some given properties conforms a classical subject, in the list of such transformations given in the classical treatise by Einsehart \cite{eisenhart} we find the Lévy (named after Lucien Lévy) transformation \cite{levy}, which later on was named as elementary Darboux transformation and known in the orthogonal polynomials context as Christoffel transformation \cite{yoon, Szego}. The adjoint elementary Darboux or adjoint Lévy transformation is also relevant  \cite{matveev,dsm} and is referred some times as a Geronimus transformation \cite{yoon}. Finally the rational Uvarov transformation corresponds to the fundamental transformation introduced in a geometrical context by Hans Jonas \cite{jonas}, see also the discussion in \cite{eisenhart}. For further information see \cite{rogers-schief,gu}. 
The iteration formula of Christoffel transformations is due to Elwin Bruno Christoffel  \cite{christoffel}. This fact was rediscovered much  latter in the Toda context, see for example the formula (5.1.11) in \cite{matveev} for $W^+_n(N)$.

\subsubsection{Integrable systems} The papers \cite{sato0,sato} and  \cite{date1,date2,date3} settled the Lie group theoretical description of integrable hierarchies.  See \cite{mulase}  for a discussion of the role of the factorization problems, dressing procedure, and linear systems as the keys for
integrability. In this dressing setting the multicomponent integrable hierarchies of Toda type were analyzed in \cite{ueno-takasaki0,ueno-takasaki1,ueno-takasaki}. For further developments see \cite{BtK,BtK2} and \cite{kac,mmm} for the multi-component KP hierarchy and \cite{mma}  for  the multi-component Toda lattice hierarchy.
Mark Adler and Pierre van Moerbeke showed how the Gauss--Borel factorization problem appears in the theory
of the 2D Toda hierarchy and what they called the discrete KP hierarchy \cite{adler,adler-van moerbeke, adler-vanmoerbeke 0,adler-van moerbeke 1,adler-van moerbeke 1.1,adler-van moerbeke 2}. These papers clearly established  --from a group-theoretical setup-- why standard orthogonality of polynomials and integrability of nonlinear equations of Toda type where so close. In fact, the Gauss--Borel factorization of the moment matrix may be understood as the Gauss--Borel factorization of the initial condition for the integrable hierarchy. 

\subsubsection{Our previous work} In the Madrid group, based on the Gauss--Borel factorization, we have been searching further the deep links between the Theory of Orthogonal Polynomials and the Theory of Integrable Systems. In \cite{cum1} we studied the generalized orthogonal polynomials \cite{adler} and its matrix extensions from the Gauss--Borel view point. In \cite{cum2} we gave a complete study in terms of this factorization for multiple orthogonal polynomials of mixed type and characterized the integrable systems associated to them. Then, we studied Laurent orthogonal polynomials in the unit circle trough the CMV approach in \cite{carlos} and found in \cite{carlos2} the Christoffel--Darboux formula for generalized orthogonal matrix polynomials. These methods where further extended, for example we gave an alternative Christoffel--Darboux formula for mixed multiple orthogonal polynomials \cite{gerardo1} or developed the corresponding  theory of matrix  Laurent orthogonal polynomials in the unit circle and its associated Toda type hierarchy \cite{MOLPUC}.

\subsection{Layout of the paper} In \S 2, for a Borel  measure in the unit torus this monomial ordering induces an arranging of the moments in a  moment matrix, built up of growing  size rectangular blocks, as it happens for MVOPR \cite{xu0, MVOPR},  this moment matrix has  a rich structure having therefore  several important properties. In the quasi-definite case the Gauss--Borel factorization of this matrix gives multivariate biorthogonal Laurent polynomials in the unit torus and to
MVOLPUT  for positive measures. We define second kind functions that are shown to be a product of the MVOLPUT and the Fourier series of the measure. The moment matrix has  persymmetries and partial persymmetries which happen to be useful  for the finding of  Plemej type formulae for the orthogonal Laurent polynomials.  Then, another symmetry of the moment matrix, which we call string equation, leads to three-term relations and  Christoffel–Darboux formulæ.

We  study in \S 3 Laurent polynomial perturbations of the measure, which could be considered as an extension of the Christoffel transformation to this framework. Interpolation theory and the construction of appropriate sample matrices based on  poised sets allows for the construction of the Christoffel formula in this generalized scenario.
To construct poised sets we need of what we call nice Laurent polynomials, which behave nicely with the extended CMV ordering. We study  nice Laurent polynomials with the aid of Newton polytopes, and some aspects of tropical geometry. Full column rank Laurent–Vandermonde matrices and algebraic geometry lead to poised sets in this context. These Darboux transformations are studied for the Lebesgue–Haar measure. Perturbed Christoffel–Darboux kernels and  connection formulæ for the perturbed and non perturbed kernels are given 

Finally, in \S 4 we introduce integrable systems of Toda and Kadomtsev–Petviasvilii (KP )type. In particular we show how discrete Toda type flows appear and give  Miwa type expressions for the orthogonal Laurent polynomials.  A Toda type integrable hierarchy is found,  and Baker functions, Lax matrices and Zakharov--Sahabat equations are discussed. We also find bilinear equations,  Miwa shifts and corresponding vertex operators. Finally, the use  of asymptotic modules allows for the construction of KP onsite flows; i.e. to integrable nonlinear equations constructed in terms of the coefficient  associated to a given MVOLPUT, say  $\phi_{[k]}$ with a fixed $k$.

\section{Multivariate biorthogonal Laurent polynomials in the unit torus }

The algebraic torus $(\C^*)^D$,  $\C^*\coloneq\C\setminus\{0\}$, is an Abelian group under component-wise multiplication and has as its coordinate ring the ring of  Laurent polynomials $\C[\z^{\pm 1}]\equiv\C[z_1^{\pm 1}\dots,z_D^{\pm 1}]$, where we   require of $D$ independent complex variables $\z=(z_1,z_2,\dots,z_D )^\top\in(\C^*)^D$.  Therefore, the algebraic torus is the natural framework when considering  the ring of multivariate Laurent polynomials.

Given  the unit torus   $\T^D=\bigtimes\limits_{i=1}^D\T$, the Cartesian product of $D$ copies of the unit circle, and a   a Borel measure $\d\mu \in \mathcal{B}(\Omega)$ on some complex domain $\Omega$  such that its support belongs to the unit torus, $\operatorname{supp} \mu\subset \T^D$, we will study the corresponding orthogonal Laurent polynomials and their  properties. Remember that the $D$-dimensional unit polydisk  $\D^D$ has  the  unit torus $\T^D$ as its distinguished or Shilov border.  For $\z\in\T^D$ we have the parametrization 
\begin{align*}
  \z(\boldsymbol\theta)=\big(\Exp{\ii\theta_1},\dots, \Exp{\ii\theta_D}\big)^\top
\end{align*} with $\boldsymbol\theta\coloneq(\theta_1,\dots,\theta_D)^\top\in[0,2\uppi)^D$. A complex Borel measure $\mu$ supported in $\T^D$   is said to be real, respectively positive definite, if it maps measurable sets  into the real numbers, respectively into non negative numbers. When the measure is  absolutely continuous with respect to the Lebesgue--Haar measure  $\d\boldsymbol\theta=\d\theta_1\cdots\d\theta_D$ it has the form $\d\mu(\boldsymbol\theta)=w({\boldsymbol\theta}) \d \boldsymbol\theta$ and the weight $w(\boldsymbol\theta)$ must be a positive function when $\mu$ is positive definite.
The inner product  of two complex valued functions $f(\z)$ and $g(\z)$ is defined by
\begin{align*}
 \langle f,g\rangle&\coloneq\oint_{\T^D}  f(\z(\boldsymbol\theta)) \overline{ g(\z(\boldsymbol\theta))}\d\mu(\boldsymbol\theta).
\end{align*}

 \subsection{Laurent monomials.  Ordering Laurent polynomials}

 The ring of complex Laurent polynomials \begin{align*}
  \C[z_1^{\pm 1},\dots,z_D^{\pm 1}] 
 \end{align*}
 is the localization of the polynomial ring $ \C[z_1,\dots,z_D]$ by adding the formal inverses of $z_1,\dots,z_D$. The units; i.e., invertible elements in the ring, are the Laurent monomials $a\z^\q$, $a\in\C^*$.
 Following \cite{gelfand} we say

\begin{definition}
	 A Laurent monomial is a well-defined function between algebraic tori $\z^\q: (\C^*)^D \to \C^*$.
\end{definition}
 A Laurent monomial can be viewed as a  character of the algebraic torus and Laurent polynomials can be thought as finite linear combination of Laurent monomials. Any usual polynomial in the ring $\C[\z]$ can be considered as a Laurent polynomial.

 \begin{definition}\label{def:support}
\begin{enumerate}
	 	 \item Given a multi-index $\q=(\alpha_1,\dots,\alpha_D)^\top \in\Z^{D}$  we write $\z^{\q}=z_1^{\alpha_1}\cdots z_D^{\alpha_D}$ .
	 	\item  For a finite subset $\mathcal A\subset\Z^D$ we define  $\C^{\mathcal A}$ as the set of Laurent polynomials with monomials from $\mathcal A$;  $L\in\C^{\mathcal A}$ if and only if $L=\sum\limits_{\q\in\mathcal A} L_\q\z^\q$ with $L_\q\neq 0$ for $\q\in\mathcal A$.
	 	\item When $L\in\C^{\mathcal A}$ we say that $\mathcal A$ is the support  of $L$ and its convex hull is known as  its Newton polytope  $\operatorname{NP}(L)$
 \begin{align*}
 	 \operatorname{NP}(L)\coloneq\operatorname{Conv}(\mathcal A).
 	 \end{align*}
 	 \item A supporting hyperplane $H$ of the Newton polytope $\operatorname{NP}(L)$ is an affine hyperplane such that the face $F=H\cap \operatorname{NP}(L)\neq\emptyset$ and $\operatorname{NP}(L)$ is fully contained in one of the two halfspaces defined by $H$. A face is called facet if it has codimension 1. Given $\boldsymbol w\in\R^D$ we construct the associated face
 	 \begin{align*}
 	 F_{\boldsymbol w}(\operatorname{NP}(L))\coloneq
 	 \big\{\boldsymbol u\in\operatorname{NP}(L)        : (\boldsymbol u-\boldsymbol{v})\cdot\boldsymbol w\leq 0
 	 \quad \forall \boldsymbol v\in\operatorname{NP}(L)\big\}.
 	 \end{align*}
 	 Given a face $F$ of $\operatorname{NP}(L)$  its normal cone is $\mathcal N_F(\operatorname{NP}(L))\coloneq\big\{\boldsymbol w\in\R^D: F=F_{\boldsymbol w}(\operatorname{NP}(L))\big\}$, the normal fan $\mathcal N(\operatorname{NP}(L))\coloneq\big\{\mathcal N_F(\operatorname{NP}(L)): \text{$F$ is a face of  $\operatorname{NP}(L)$}\big\}$ is the collection of all normal cones.
 	 \item For any multi-index $\q\in\Z^D$ we define its longitude  $|\q|\coloneq  \sum_{a=1}^{D} |\alpha_a|$.
 	 \item  Given a Laurent polynomial $L\in\C^{\mathcal A}$ its longitude is
 	               \begin{align*}
 	               \ell(L)\coloneq\max_{\q\in\mathcal A}|\q|.
 	               \end{align*}
\end{enumerate}
 \end{definition}
 Observe that for the ring of polynomials $\C[\z]$ this longitude  is the total degree  \cite{Dunkl,MVOPR}, and 
given two polynomials $P_1,P_2\in\C[\z]$ we have $\deg (P_1P_2)=\deg P_1+\deg P_2$.\footnote{The  ring of polynomials $\C[z_1,\dots,z_D]$ is a $\Z_+$-graded ring
 $	\C[z_1,\dots,z_D]=\bigoplus_{n\in\Z_+} \C_n[z_1,\dots,z_D]$ with
 	$\C_n[z_1,\dots,z_D]$ the degree  $n$ homogenous polynomials
 	and
 $
 	\C_n[z_1,\dots,z_D]\C_m[z_1,\dots,z_D]\subseteq \C_{n+m}[z_1,\dots,z_D]$.} However, for the Laurent polynomial ring  we have $\ell (L_1L_2)\leq \ell (L_1)+\ell (L_2)$.\footnote{ This follows from $|n+m|\leq |n|+|m|$, that holds for all $n,m\in\Z$, inequality that is saturated $|n+m|=|n|+|m|$ for either couples of positive integers or couples of negative integers.} 
Notice that this order is not a monomial order; i.e., given two monomials $L_1<L_2$ then the multiplication by any other monomial $L$ does not respect in general this order and it could happen that $LL_1>LL_2$. Despite the previous observation we should stress that the  corresponding lattice  is a graded lattice.
 Our proposal is motivated
  by the   $D=1$ case where one has the CMV ordering \cite{CMV, watkins}, for  the Laurent polynomial ring $\C[z^{\pm 1}]$, in where powers $z^n$ and $z^{-n}$ go together \cite{CMV,carlos}.   Despite that there is $\Z$-grading, $\deg  z^{\pm}=\pm 1$, we reckon that it is not as useful for the analysis of orthogonal Laurent polynomials as the CMV ordering is.
     This will be relevant later on when we discuss Darboux transformations.

\begin{definition}
Given a nonnegative integer $k\in\Z_+$ we introduce 
 \begin{align*}
 [k]\coloneq \{\q\in \Z^{D}: |\q|=k\}.
 \end{align*}
\end{definition}

 \begin{pro}
 	The number of multi-indices in $\Z^D$ of longitude   $k>0$ is
 	\begin{align}\label{eq:RkD}
 	|[k]|=	\sum_{j=1}^{\min(k,D)}2^j \binom{D}{j}\binom{k-1}{j-1},
 	\end{align}
 	and  $|[0]|=1$.
 \end{pro}
 \begin{proof}
 	Recall that there are  $\tbinom {D} j$ strings containing $j$ ones and $(D-j)$ zeros, and observe that, when  counting compositions,  the number  of  ways to write $k=a_1+a_2+\cdots+a_j$, where every $a_i$ is a positive integer, is given by $\tbinom{k-1}{j-1}$. Thus, $\tbinom {D} j\tbinom{k-1}{j-1}$ is the number of  partitions $k=a_1+a_2+\cdots+a_D$, with non-negative integers $a_i\in\Z_+$ for $i\in\{1,\dots,D\}$; i.e., of weak compositions, with $(D-j)$ factors equal to zero and $j$  positive integers. We now drop the non negative  condition and allow for arbitrary integers,  $k=|a_1|+|a_2|+\cdots+|a_D|$ as we have $j$ non-vanishing components, we should multiply by 2, as we must take into account a reversal of sign, that gives the same longitude, and $2^j\binom{D}{j}\binom{k-1}{j-1}$ should be the number of multi-indices $\q\in\Z^D$ of longitude $k$ having $j$ of its  $D$ components different from zero. Finally, summing up in $j$ we get $|[k]|$.
 \end{proof}

\begin{definition}[The \texttt{longilex} order in $\Z^D$]\footnote{This not the $\texttt{deglex}$ ordering introduced in the \href{http://www.sagemath.org/doc/reference/polynomial_rings/sage/rings/polynomial/laurent_polynomial_ring.html}{SAGE package} for the Laurent ring polynomial. }\label{longilex}
First, we order according the longitude of the multi-indices; i.e.,  $\q<\q'$ whenever $|\q|<|\q'|$.
Second,
  we use the  lexicographic order for the set  $[k]$  of multi-indices of same longitude and write
  \begin{align*}
  [k]=\Big\{\q_1^{(k)},\q_2^{(k)},\dots,\q^{(k)}_{|[k]|}\Big\} \text{ with } \q_a^{(k)}<\q_{a+1}^{(k)}.
  \end{align*}
\end{definition}

 \begin{pro}\begin{enumerate}
 \item The convex hull $\operatorname{Conv}([k])$ is a regular hyper-octahedron  with  vertices  given by
 \begin{align*}
 \{V^{(k)}_{1,\pm}\coloneq(\pm k,0,\dots,0)^\top,V^{(k)}_{2,\pm}\coloneq(0,\pm k,0,\dots,0)^\top,\dots,V^{(k)}_{D,\pm}\coloneq(0,\dots,0,\pm k)^\top\}\subset\R^D.
 \end{align*}
 \item 	The  Newton polytope of a Laurent polynomial belongs to the  regular hyper-octahedron $\operatorname{NP}(L)\subseteq \operatorname{Conv}([\ell(L)])$, and at least a face of the Newton polytope $\operatorname{NP}(L)$  has a nontrivial intersection with a face of $\operatorname{Conv}([\ell(L)])$.
 	\end{enumerate}
 \end{pro}

 \begin{definition}
  We introduce the  semi-infinite vector $\chi$ constructed by using the $\texttt{longilex}$ order of  Laurent monomials
 \begin{align*}
  \chi&\coloneq \PARENS{\begin{matrix}\chi_{[0]} \\ \chi_{[1]} \\ \vdots \\ \chi_{[k]} \\ \vdots \end{matrix}}
  & \mbox{where} & &
  \chi_{[k]}&\coloneq  \PARENS{\begin{matrix} \z^{\q_1} \\  \z^{\q_2} \\\vdots \\ \z^{\q_{|[k]|}} \end{matrix}}, &  [k]=\Big\{\q_1^{(k)},\q_2^{(k)},\dots,\q^{(k)}_{|[k]|}\Big\}.
  \end{align*}
  \end{definition}
  In particular, the two first are easy to write
   \begin{align}\label{zeta}
 \chi_{[0]}&=1,&
\chi_{[1]}&=\PARENS{\begin{matrix}
     z_1^{-1}\\z_2^{-2}\\\vdots \\z_D^{-1}\\z_D\\\vdots\\z_2\\z_1
   \end{matrix}},
 \end{align}
 and for $D=2,3$ we have
 \begin{align*}
    \chi_{[2]}&=\PARENS{\begin{matrix}
     z_1^{-2}\\z_1^{-1}z_2^{-1}\\z_1^{-1}z_2\\z_2^{-2}\\z_2^2\\z_1z_2^{-1}\\z_1z_2\\z_1^2
   \end{matrix}},&
   \chi_{[2]}&=\PARENS{\begin{matrix}
     z_1^{-2}\\z_1^{-1}z_2^{-1}\\z_1^{-1}z_3^{-1}\\z_1^{-1}z_3\\z_1^{-1}z_2\\z_2^{-2}\\z_2^{-1}z_3^{-1}\\z_2^{-1}z_3\\z_3^{-2}\\z_3^2\\z_2z_3^{-1}\\z_2z_3\\z_2^2\\z_1z_2^{-1}\\z_1z_3^{-1}\\z_1z_3\\z_1z_2\\z_1^{2}
   \end{matrix}},
 \end{align*}
 respectively.

In this paper we will consider semi-infinite matrices $A$ with a block or partitioned structure induced by the \texttt{longilex} order of Definition \ref{longilex}
\begin{align*}
A&=\PARENS{\begin{matrix}
   A_{[0],[0]} & A_{[0],[1]} &  \cdots  \\
   A_{[1],[0]} & A_{[1],[1]} &  \cdots \\
   \vdots                &                 \vdots         &  \\
  \end{matrix}}, &
A_{[k],[l]}&=\PARENS{\begin{matrix}
  A_{\q^{(k)}_1,\q^{(l)}_1} &   \dots & A_{\q^{(k)}_1,\q^{(l)}_{|[l]|} }\\
  \vdots & & \vdots\\
  A_{\q^{(k)}_{|[k]|},\q^{(l)}_1} &  \dots & A_{\q^{(k)}_{|[k]|},\q^{(l)}_{|[l]|} }
  \end{matrix}} \in\C^{|[k]|\times |[l]|}.
\end{align*}
We use the notation $0_{[k],[l]}\in\C^{|[k]|\times|[l]|}$ for the rectangular zero matrix, $0_{[k]}\in\C^{|[k]|}$ for the zero vector, and $\I_{[k]}\in\C^{|[k]|\times|[k]|}$ for the identity matrix. For the sake of simplicity and if there is no confusion we prefer to write $0$ or $\I$ for the zero or identity matrices, and we implicitly assume that the sizes of these matrices are the ones indicated by its position in the partitioned matrix.

\begin{pro}
	The linear space  $\C_{(k)}[\z^{\pm 1}]$ of Laurent polynomials of longitude $k$ has dimension $N_k\coloneq\dim\C_{(k)}[\z^{\pm 1}]$ given by
	\begin{align*}
	N_k=\sum_{l=0}^k|[l]|.
	\end{align*}
\end{pro}

\subsection{The moment matrix}
\begin{definition}\label{def:moment} Given a Borel measure $\mu$ with support in the unit torus $\T^D$ the corresponding  moment matrix $G_\mu$ is given by
\begin{align}\label{moment}
 G_\mu\coloneq\oint_{\T^D} \chi(\z(\boldsymbol\theta)) \d \mu(\boldsymbol\theta) \chi(\z(-\boldsymbol\theta))^\top.
\end{align}
 For the sake of simplicity when not needed we omit the subscript $\mu$ and write $G$ instead of $G_\mu$. We write the moment matrix  in block form
  \begin{align*}
  G=  \PARENS{\begin{matrix}
   G_{[0],[0]} & G_{[0],[1]} &  \dots \\
   G_{[1],[0]} & G_{[1],[1]} &  \dots \\
   \vdots                &   \vdots              &
  \end{matrix}}
   \end{align*}
with each entry being a rectangular  matrix  with complex coefficients
 \begin{align}
 \label{eq:Gkl}
  G_{[k],[l]}\coloneq &\oint_{\T^D}  \chi_{[k]}(\z(\boldsymbol\theta)) \d \mu(\boldsymbol\theta) \chi_{[l]}(\z(-\boldsymbol\theta))^\top\\
 =&\PARENS{\begin{matrix}
  G_{\q^{(k)}_1,\q^{(l)}_1} &   \dots & G_{\q^{(k)}_1,\q^{(l)}_{|[l]|} }\\
  \vdots & & \vdots\\
  G_{\q^{(k)}_{|[k]|},\q^{(l)}_1} &  \dots & G_{\q^{(k)}_{|[k]|},\q^{(l)}_{|[l]|} }
  \end{matrix}} \in  \C^{|[k]|\times |[l]|}, &
    G_{\q_i^{(k)},\q^{(l)}_j}&\coloneq \oint_{\T^D} \Exp{\ii(\q^{(k)}_i-\q^{(l)}_j)\cdot\boldsymbol\theta}\d\mu(\boldsymbol\theta)
     \in \C. 
     \notag
 \end{align}
Truncated  moment matrices are given by
 \begin{align*}
  G^{[k]}&\coloneq
  \PARENS{\begin{matrix}
   G_{[0],[0]} &  \cdots & G_{[0],[k-1]} \\
   \vdots                        &   & \vdots \\
   G_{[k-1],[0]}  &  \cdots & G_{[k-1],[k-1]}
  \end{matrix}},
 \end{align*}
 and for $l\geqslant k$ we will also use the following bordered truncated moment matrix
\begin{align*}
   G^{[k]}_l&\coloneq
  \PARENS{\begin{array}{ccc}
   G_{[0],[0]} &  \cdots & G_{[0],[ k-1]} \\
   \vdots                        &   & \vdots \\
   G_{[ k-2],[0]}  &  \cdots & G_{[ k-2],[ k-1]}\\[1pt]
   \hline
   G_{[l],[0]}& \dots & G_{[l],[ k-1]}
  \end{array}}
 \end{align*}
  where we have replaced the last row of blocks, $\PARENS{\begin{matrix}
  G_{[ k],[0]}& \dots & G_{[ k-1],[ k-1]} \end{matrix}}$, of the truncated moment matrix $G^{[ k]}$ by the row of blocks $\PARENS{\begin{matrix}
  G_{[l],[0]}& \dots & G_{[l],[ k-1]} \end{matrix}}$, we also need a similar matrix but replacing rows by columns
  \begin{align*}
   \hat G^{[k]}_l&\coloneq
  \PARENS{\begin{array}{ccc|c}
   G_{[0],[0]} &  \cdots & G_{[0],[ k-2]}&G_{[0],[l]} \\
   \vdots  &                      &  \vdots  & \vdots \\
   G_{[ k-1],[0]}  &  \cdots & G_{[ k-1],[ k-2]}&G_{[ k-1],[l]}
  \end{array}}.
 \end{align*}
\end{definition}
Let us extend to this scenario the concept of quasi-definite and positive definite
\begin{definition}
	The measure and its moment matrix are quasi-definite if all its principal block minors are not singular
	\begin{align*}
	\det G^{[k]}&\neq 0, & k\in\{0,1,\dots\}.
	\end{align*}
	When all the minors are positive
		\begin{align*}
		\det G^{[k]}&> 0, & k\in\{0,1,\dots\}.
		\end{align*}
	 we say that it is a definite positive moment matrix
\end{definition}
Notice, that instead of positive definite we could request the moment matrix to be  definite; i.e., all block principal minors are positive or all are negative.

We are now ready to discuss some aspects regarding the Gauss--Borel factorization of this moment matrix.
\begin{pro}\label{pro:gauss}\begin{enumerate}
  \item A quasi-definite moment matrix $G$ admits the following block Gauss--Borel  factorization
 \begin{align}\label{cholesky}
  G&=S^{-1} H \big(\hat S^{-1}\big)^{\dagger},
  \end{align}
  with
\begin{align*}
  S^{-1}&=\PARENS{\begin{matrix}
  \I_{|[0]|}    &             0                &  0                      &  \cdots            \\
  (S^{-1})_{[1],[0]}        & \I_{|[1]|} &     0                      &   \cdots         \\
  (S^{-1})_{[2],[0]}        & (S^{-1})_{[2],[1]} & \I_{|[2]|} &      \\
           \vdots                       &        \vdots                  &                               &\ddots
  \end{matrix}}, \\
  H&=\PARENS{\begin{matrix}
H_{[0]}           &   0         &     0     &\cdots    \\
0                 & H_{[1]} &   0             &    \cdots       \\
0                  &    0            & H_{[2]} &                        \\
\vdots   &   \vdots &              &     \ddots       \\
  \end{matrix}},\\
   \hat S^{-1}&=\PARENS{\begin{matrix}
  \I_{|[0]|}     &             0                &  0                      &  \cdots            \\
  (\hat S^{-1})_{[1],[0]}        & \I_{|[1]|} &     0                      &   \cdots         \\
  (\hat S^{-1})_{[2],[0]}        & (\hat S^{-1})_{[2],[1]} & \I_{|[2]|} &      \\
           \vdots                       &        \vdots                  &                               &\ddots
  \end{matrix}}.
  \end{align*}
Moreover, the quasi-tau matrices $H_{[k]}$ are non singular, $k\in\{0,1,\dots\}$, and
\begin{align*}
\det G^{[ l]}&=\prod_{k=0}^{ l-1} \det H_{[k]}.
\end{align*}

\item The measure  is real if and only if  its moment matrix is Hermitian  $G=G^\dagger$. In this case  
\begin{align*}
\hat  S&= S, &
  H^\dagger &=H.
\end{align*}
\item The measure is definite positive if and only if   the moment matrix is definite positive. In this case, the quasi-tau matrices $H_{[k]}$ are definite positive matrices for $k\in\{0,1,\dots\}$.
\end{enumerate}
\end{pro}
\begin{proof}
See Appendix \ref{proof1}.
\end{proof}

Quasi-determinants are extensions of determinants ---more appropriately of quotient of determinants---, and fulfill the heredity principle, quasi-determinants of quasi-determinants are quasi-determinants.  The Gel'fand school has given a very complete study on the subject, see \cite{gelfand,quasidetermiant6,quasidetermiant7,quasidetermiant8}. However, in this paper we require of the generalization given by Olver in \cite{olver}.
A last quasi-determinant version of the above  result can be given
\begin{pro}\label{qd1}
If the last quasi-determinants of the truncated moment matrices are invertible
\begin{align*}
\det  \Theta_*(G^{[k]})\neq& 0, & k=1,2,\dots
\end{align*}
the Gauss--Borel factorization \eqref{cholesky} can be performed where
\begin{align*}
  H_{[k]}&=\Theta_*(G^{[k+1]}), &
  (S^{-1})_{[k],[ l]}&=\Theta_*(G^{[ l+1]}_k)\Theta_*(G^{[ l+1]})^{-1},&
  (\hat S^{-1})_{[k],[ l]}&=\big(\Theta_*(G^{[ l+1]})^{-1}\Theta_*(\hat G^{[ l+1]}_k)\big)^\dagger. \end{align*}
\end{pro}
\begin{proof}
  It is just a consequence of Theorem 3 of \cite{olver}.
\end{proof}
\begin{definition}
	The matrices $H_{[k]}$ are called quasi-tau matrices. We introduce  the   first subdiagonal  matrices
	\begin{align*}
	\beta_{[k]}&\coloneq S_{[k],[k-1]},&
	\hat\beta_{[k]}&\coloneq \hat S_{[k],[k-1]},  & k\geqslant 1,
	\end{align*}
	which take values in the linear space of rectangular matrices $\C^{|[k]|\times|[k-1]|}$ and also define
	\begin{align*}
	\beta&=\PARENS{\begin{matrix}
		0 & 0 & 0 &0&\cdots\\
		\beta_{[1]}& 0 & 0&0&\cdots\\
		0&\beta_{[2]}&0&0&\cdots\\
		0&0&\beta_{[3]}&0&\cdots\\
		\vdots&\vdots&\ddots&\ddots&\ddots
		\end{matrix}}, & \hat\beta&=\PARENS{\begin{matrix}
		0 & 0 & 0 &0&\cdots\\
		\hat\beta_{[1]}& 0 & 0&0&\cdots\\
		0&\hat\beta_{[2]}&0&0&\cdots\\
		0&0&\hat\beta_{[3]}&0&\cdots\\
		\vdots&\vdots&\ddots&\ddots&\ddots
		\end{matrix}}
	\end{align*}
\end{definition}

An immediate consequence of Proposition \ref{qd1} is
\begin{pro}
	The first subdiagonal matrices have the following  quasi-determinantal expressions
	\begin{align*}
	\beta_{[k]}&=-\Theta_*(G^{[k]}_k)\Theta_*(G^{[k]})^{-1}, &
	\hat\beta_{[k]}&=-\Big(\Theta_*(\hat G^{[k]}_k)\Theta_*(G^{[k]})^{-1}\Big)^\dagger.
	\end{align*}
\end{pro}

\subsection{Orthogonal Laurent polynomials in the unit torus}
\label{MVOLPUT}
With the aid of the Gauss--Borel factorization we  introduce
\begin{definition}
 We define the multivariate Laurent polynomials in  $D$ complex variables $z_1,\dots,z_D$
 \begin{align}\label{eq:polynomials}
 \begin{aligned}
  \Phi&\coloneq S\chi =\PARENS{\begin{matrix}
    \phi_{[0]}\\
    \phi_{[1]}\\
    \vdots
  \end{matrix}}, & \phi_{[k]}(\z)&=\sum_{ l=0}^k S_{[k],[ l]} \chi_{[ l]}(\z) =\PARENS{\begin{matrix}
    \phi_{\q^{(k)}_1}\\
    \vdots\\
    \phi_{\q^{(k)}_{|[k]|}}
  \end{matrix}},&
  \phi_{\q^{(k)}_i}&=\sum_{ l=0}^k\sum_{j=1}^{|[ l]|} S_{\q^{(k)}_i,\q^{( l)}_j} \z^{\q^{( l)}_j},\\
 \hat \Phi&\coloneq\hat S\chi =\PARENS{\begin{matrix}
    \hat \phi_{[0]}\\
    \hat \phi_{[1]}\\
    \vdots
  \end{matrix}}, & \hat \phi_{[k]}(\z)&=\sum_{ l=0}^k \hat S_{[k],[ l]} \chi_{[ l]}(\z) =\PARENS{\begin{matrix}
    \hat \phi_{\q^{(k)}_1}\\
    \vdots\\
   \hat \phi_{\q^{(k)}_{|[k]|}}
  \end{matrix}},&
 \hat \phi_{\q^{(k)}_i}&=\sum_{ l=0}^k\sum_{j=1}^{|[ l]|} S_{\q^{(k)}_i,\q^{( l)}_j} \z^{\q^{( l)}_j}.
 \end{aligned}
 \end{align}
\end{definition}

Observe that $\phi_{[k]}(\z)=\chi_{[k]}(\z)+\beta_{[k]}\chi_{[k-1]}(\z)+\cdots$ is a vector constructed with the multivariate Laurent polynomials $\phi_{\q_i}(\z)$ of longitude  $k$, each of which has only one monomial of longitude $k$; i. e., we can write $\phi_{\q_i}(\z)=\z^{\q_i}+Q_{\q_i}(\z)$, with $\ell(Q_{\q_i})<k$; the same holds for $\hat \phi_{[k]}$.

\begin{pro}
The two sets of  Laurent polynomials $\{\phi_{[k]}\}_{k=0}^\infty$ and $\{\hat \phi_{[k]}\}_{k=0}^\infty$ form a biorthogonal system; i.e.,
 \begin{align}\label{eq:biorthogonality}
  \oint_{\T^D} \phi_{[k]}(\z(\boldsymbol\theta)) \d \mu(\boldsymbol\theta) \big(\hat \phi_{[ l]}(\z(\boldsymbol\theta))\big)^{\dagger}= \delta_{k, l} H_{[k]}.
   \end{align}
\end{pro}
\begin{proof}
 It is straightforward from the Gauss--Borel factorization.
\end{proof}
This biorthogonality leads to the following orthogonality relations
\begin{align*}
 \oint_{\T^D} \phi_{[k]}(\z(\boldsymbol\theta)) \d \mu(\boldsymbol\theta) \big(\chi_{[ l]}(\z(\boldsymbol\theta))\big)^\dagger=\oint_{\T^D} \chi_{[ l]}(\z(\boldsymbol\theta)) \d \mu(\boldsymbol\theta) \big(\hat \phi_{[k]}(\z(\boldsymbol\theta))\big)^{\dagger}&=0,&  l&=0,\dots,k-1,\\
 \oint_{\T^D} \phi_{[k]}(\z(\boldsymbol\theta)) \d \mu(\boldsymbol\theta) \big(\chi_{[k]}(\z(\boldsymbol\theta))\big)^\dagger=\oint_{\T^D} \chi_{[k]}(\z(\boldsymbol\theta)) \d \mu(\boldsymbol\theta) \big(\hat \phi_{[k]}(\z(\boldsymbol\theta))\big)^{\dagger}&=H_{[k]},
\end{align*}
\begin{pro}
When $\d\mu(\boldsymbol\theta)$ is positive definite the set  $\big\{\phi_{[k]}(\z)\big\}_{k=0}^\infty$ is an orthogonal set of  Laurent polynomials in the unit torus $\T^D$; i.e.,
 \begin{align}\label{biorthogonality}
  \oint_{\T^D} \phi_{[k]}(\z(\boldsymbol\theta)) \d \mu(\boldsymbol\theta) (\phi_{[ l]}(\z(\boldsymbol\theta)))^{\dagger}= \delta_{k, l} H_{[k]}.
   \end{align}
\end{pro}
For $k= l$ we get
\begin{align*}
  \oint_{\T^D} \|\phi_{[k]}(\z(\boldsymbol\theta))\|^2 \d \mu(\boldsymbol\theta)= H_{[k]},
   \end{align*}
   and therefore the quasi-tau matrices $H_{[k]}$ may be viewed as the squared norm matrices of the MVOLPUT.
 Notice that we are talking about biorthogonality or orthogonality in a block form. For example, for real measures what we have is that the blocks of different longitude are orthogonal, which is fine, however the Laurent polynomials of the same length are not truly orthogonal but satisfy
 \begin{align*}
  \oint_{\T^D} \phi_{\q^{(k)}_i}(\z(\boldsymbol\theta)) \d \mu(\boldsymbol\theta) \overline{\phi_{\q^{(k)}_j}(\z( \boldsymbol\theta))}= \delta_{k, l} H_{\q^{(k)}_i,\q^{(k)}_j}.
   \end{align*}
In the positive definite case  the quasi-tau matrix $H_{[k]}$ is a positive definite  Hermitian matrix and we could find and orthogonal transformation giving standard orthogonal Laurent polynomials. But if we do so we spoil the symmetry that leads to fundamental properties, as we will see later, of these polynomials.

\begin{pro}
The following quasi-determinantal formul\ae hold true
  \begin{align*}
 \phi_{[k]}(\z)&=\Theta_*\PARENS{
 \begin{matrix}
       G_{[0],[0]} &  \dots & G_{[0],[k-1]} & \chi_{[0]}(\z)\\
         \vdots      &            &   \vdots            &  \vdots\\
       G_{[k],[0]} & \dots & G_{[k],[k-1]} & \chi_{[k]}(\z)
      \end{matrix}
      }, &
 \big(\hat \phi_{[k]}(\z)\big)^\dagger&=\Theta_*\PARENS{
 \begin{matrix}
       G_{[0],[0]} & \dots & G_{[0],[k]}\\
        \vdots      &       &         \vdots\\
       G_{[k-1],[0]} &  \dots & G_{[k-1],[k]}\\[1pt]
      \big(\chi_{[0]}(\z)\big)^\dagger & \dots &
     \big(\chi_{[k]}(\z)\big)^\dagger
      \end{matrix}}
 \end{align*}
\end{pro}

\subsection{Holomorphic extensions of the  Fourier series of the measure and second kind functions}

First  we recall some basic facts regarding  the analysis in  several complex variables, see \cite{begehr,scv0,scv1,scv2,rudin0} for more information.
 Given the vector  with positive components $\boldsymbol r=(r_1,\dots,r_D)^\top\in\R_+^D$, the polydisk
 	\begin{align*}
 	\D^D(\boldsymbol r)=\big\{\boldsymbol z=(z_1,\dots,z_D)^\top:|z_i|<r_i, i\in\{1,\dots,D\}\big\}\subset\C^D
 	\end{align*}
 	centered at the origin of polyradius $\boldsymbol r$ has as its distinguished boundary, also known as Shilov border,  the $D$-dimensional torus
 	\begin{align*}
 	\T^D(\boldsymbol r)=\big\{\boldsymbol z\in\C^D: |z_i|=r_i,  i\in\{1,\dots,D\}\big\}.
 	\end{align*}
 	For any two polyradii $\boldsymbol r$ and $\boldsymbol R$ the associated polyannulus centered at the origin is
 	\begin{align*}
 	A^D(\boldsymbol r, \boldsymbol R)\coloneq\big \{\z\in\C^D: r_i< z_i< R_i, i\in\{1,\dots,D\}\big\}.
 	\end{align*}
 Let us recall that a  set $A\subset \C^D$ is a complete Reinhardt domain if the unit polydisk $\D^D$ acts on it by componentwise multiplication.
  The polydisk of convergence of a power series is such that any other polydisk  $\D^D(\boldsymbol r')$ with $r_j<r_j'$ for some $j\in\{1,\dots,D\}$ contains points in where the power series diverge.

 We remind the reader that any set $A\subset\C^D$ is a Reinhardt domain\footnote{Sometimes named $D$-circled domain or circled domain.} if the unit torus $\T^D$ acts on it (for  every $\boldsymbol c\in A$ and  $\Exp{\operatorname{i}\boldsymbol\theta}\in\T^D$
 we have that $(\Exp{\operatorname{i}\theta_1}c_1,\dots, \Exp{\operatorname{i}\theta_D}c_D)^\top\in A$).
 The domain of convergence $\mathscr D_L\subset\C^D$ of a Laurent series $L(\z)=\sum_{\q}L_\q z^\q$ is a Reinhardt domain.
 Recall that for all polyradii $\boldsymbol r$ and $\boldsymbol R$ the annulus $A^D(\boldsymbol r,\boldsymbol R)$ is a Reinhardt domain and that any Reinhardt domain is the union of polyannuli.
 The Laurent  series is locally normally summable in its domain of convergence and therefore locally absolutely uniformly summable.\footnote{A Laurent series $\sum_{\q\in\Z^D}L_\q\z^{\q}$ is locally normally summable if for any compact set $K\subset\Ds_L$ there exists $C>0$ and $\theta\in(0,1)$ such that $|L_{\q}\z^{\q}|\leq C\theta^{|\q|}$ for $\z\in K$ and $\q\in\Z^D$.}
  The function $L(\z)$ is holomorphic (holomorphic in each variable $z_i$, $i\in\{1,\dots, D\}$) in $\Ds_L$, which is its domain of holomorphy.
Conversely, given a holomorphic function $L(\z)$ in $A^D_{\z_0}(\boldsymbol r,\boldsymbol R)$ (a polyannullus centered at $\z_0\in\C^D$),
 	and a polyradius $\boldsymbol \rho$ such that $r_i<\rho_i<R_i$, $i\in\{1,\dots,D\}$, then
 	\begin{align*}
 	L(\z)&=\sum_{\q\in\Z^D}L_{\q}(\z-\z_0)^{\q},&
 	L_{\q}&=\frac{1}{(2\pi\operatorname{i})^D}\int_{\T^D(\z_0,\boldsymbol\rho)}\frac{L(\z)}{(\z-\z_0)^{\q}} \d z_1 \dots \d z_D,
 	\end{align*}
 	where $\T^D(\z_0,\boldsymbol\rho)$ is the distinguished border of the polycircle centered at $\z_0$ with polyradius $\boldsymbol \rho$.

We consider the multivariate Fourier series $\hat\mu(\boldsymbol\theta)$ of the measure $\mu$ given in terms of its moments or  Fourier coefficients
\begin{align*}
  c_{\q}\coloneq&\frac{1}{(2\uppi)^D}\int_{\T^D}\Exp{-\ii \q\cdot\boldsymbol\theta}\d\mu(\boldsymbol\theta), &\q\in&\Z^D, & &\text{and} &
\hat\mu(\boldsymbol\theta)\coloneq  & \sum_{\q\in\Z^D} c_\q\Exp{\ii \q\cdot\boldsymbol\theta}.
\end{align*}
As  $c_\q(\bar\mu)=\overline{c_{-\q}(\mu)}$, for real measures we deduce that $\overline{c_{-\q}}=c_\q$ and, consequently,
$ \hat{\bar \mu}(\boldsymbol\theta)=\bar {\hat\mu}(-\boldsymbol\theta)$. Let
\begin{align*}
D(\T^D)\cong\{f\in C^\infty(\R^D): f(\x+2\uppi\q)=f(\x), \forall\x\in\R^D,\forall \q\in\Z^D\}
\end{align*}
be the   linear space of test functions. Then, the Fourier series always converges in $ D'(\T^D)$, the space of distributions on the $D$-dimensional unit torus \cite{rudin},  so that
$\int_{\T^D}\hat \mu (\boldsymbol\theta) f(\boldsymbol\theta)\d\boldsymbol\theta=\int_{\T^D}f(\boldsymbol\theta)\d\mu(\boldsymbol\theta)$,
$\forall f\in D(\T^D)$.
For an absolutely continuous measure $\d\mu(\boldsymbol\theta)=w(\boldsymbol\theta)\d\boldsymbol\theta$ we can write $\d\mu(\boldsymbol\theta)=\hat\mu(\boldsymbol\theta)\d\boldsymbol\theta$.  If we assume that $\mu$ is a Radon measure, i.e. locally finite\footnote{$\forall \boldsymbol\theta\in\T^D $ exists a neighbourhood $U$ such that $\mu(U)<\infty$.} and inner regular,\footnote{
$\mu(\mathcal B)=\sup_{K\subset\mathcal B} \mu(K)$
 where the  $K$ are compact subsets of the Borel set $\mathcal B$.} then the associated linear functional $f\mapsto \int_{\T^D}f(\boldsymbol\theta)\d\mu(\boldsymbol\theta)$ is continuous and therefore is a distribution.

Given the Fourier coefficients or moments $c_\q$ of the measure $\d\mu$, generating  its Fourier series $\hat{\mu}(\boldsymbol{\theta})$, we have the corresponding Laurent series
	\begin{align*}
	\hat{\mu}(\z)\coloneq\sum_{\q\in\Z^D}c_\q z^\q
	\end{align*}
	that converges in a Reinhardt domain $\mathcal D_\mu$.  This domain, which  is a union of polyannulli,  is the of domain of holomorphy of $\hat{\mu}(\z)$, and the series is locally absolutely uniformly summable there.
Notice that the Reinhardt domain of the Laurent series of the measure belongs entirely to the algebraic torus $\mathcal D_\mu\subset (\C^*)^D$.

In terms of the Fourier coefficients the moment matrix \eqref{eq:Gkl} can be expressed as
\begin{align*}
      G_{\q_i^{(k)},\q^{( l)}_j}=&(2\pi)^D c_{\q^{( l)}_j-\q^{(k)}_i}.
\end{align*}
 Motivated by Definition \ref{eq:polynomials}
 we now consider
 \begin{definition}\label{def:second}
The second kind functions are defined by
\begin{align*}
  \mathcal C&=(S^{-1})^\dagger\chi(\z),& \hat {\mathcal C}&=(\hat S^{-1})^\dagger\chi(\z).
\end{align*}
 \end{definition}
 We now show how these functions are linked with the MVOLPUT
\begin{pro}\label{pro:secondMVOLPUT}
Let $\mathcal D_\mu$ be the Reinhardt domain of convergence of the Laurent series $\hat{\mu}(\z)$.
Then, the second kind functions can be written in terms of $\hat{\mu}(\z)$ and the MVOLPUT as follows
\begin{align*}
\mathcal C (\z)&= (2\uppi)^D(H^\dagger)^{-1}\hat \Phi(\z)\bar{\hat \mu}(\z^{-1}), \quad \z^{-1}\in\mathcal D_\mu, & \hat{\mathcal C}(\z)=&(2\uppi)^DH^{-1}\Phi(\z)\hat\mu(\z), \quad \z\in\mathcal D_\mu.
\end{align*}
\end{pro}
\begin{proof}
See Appendix \ref{proof2}.
\end{proof}

\subsection{Persymmetries}

The moment matrix is a structured matrix. Indeed, by construction according to the \texttt{longilex} order, it is subject to important symmetries. Let us discuss  now the persymmetry that it fulfills. 
We notice that if $\q\in[k]$ then $-\q\in[k]$; moreover, using the \texttt{longilex} order of Definition \ref{longilex}
we find
\begin{align*}
\q^{(k)}_{|[k]|+1-i}=-\q^{(k)}_i.
\end{align*}

To model this fact we introduce
\begin{definition}\label{exchange}
Given any non negative integer  $m\in\Z_+$ we consider the exchange matrix
\begin{align*}
\mathcal E_m&=\PARENS{
\begin{matrix}
  0      & 0      & \cdots & 0      & 0      & 1      \\
  0      & 0      & \cdots & 0      & 1      & 0      \\
  0      & 0      & \cdots & 1      & 0      & 0      \\
  \vdots & \vdots &        & \vdots & \vdots & \vdots \\
  0      & 1      & \cdots & 0      & 0      & 0      \\
  1      & 0      & \cdots & 0      & 0      & 0
\end{matrix}}\in\C^{m\times m}, &
\big(\mathcal E_{m}\big)_{i,j} &\coloneq \ccases{
1, & j = m - i + 1, \\
0, & j \ne m - i + 1.}
 \end{align*}
Then, we introduce the following semi-infinite block diagonal reversal matrix
 \begin{align*}
\eta&\coloneq\diag(\mathcal E_{|[0]|},\mathcal E_{|[1]|},\dots).
\end{align*}
\end{definition}
The exchange matrix is also known as reversal matrix, backward identity matrix,  or standard involutory permutation matrix, see \cite{horn,golub}.
\begin{pro}\label{eta}
The following properties hold true
\begin{align*}
 \eta= \eta^{-1}&=\eta^{\dagger}=\eta^{\top}.
 \end{align*}
The vector $\chi$  fulfills
\begin{align*}
 \eta \chi(\z)&=\chi(\z^{-1}), & \big(\chi(\z)\big)^\dagger \eta &=\big(\chi(\z^{-1})\big)^{\dagger},
 \end{align*}
 where $\z^{-1}\coloneq (z_1^{-1},\dots,z_D^{-1})^\top$;   in components, the previous relation reads
  \begin{align*}
  \big(\mathcal E_{|[k]|}\chi_{[k]}(\z)\big)_i=\z^{\q^{(k)}_{|[k]|+1-i}}=\z^{-\q^{(k)}_i}.
 \end{align*}
\end{pro}
 In terms of the reversal matrix $\mathcal E_m$ a matrix $M\in\C^{m\times m}$ is said to be  persymmetric, see \cite{golub,horn,zhang}, if $\mathcal E_m M=M^\top\mathcal E_m$. We proceed to extend this concept to semi-infinite matrices
 \begin{definition}
   \label{persymmetry}
   A block semi-infinite matrix $M$ is persymmetric if
   \begin{align*}
     \eta M=M^\top\eta;
   \end{align*}
    i.e.,  if its blocks satisfy  $\mathcal E_{|[k]|}M_{[k],[ l]}=\big(M_{[ l],[k]}\big)^\top\mathcal E_{|[ l]|}$.
 \end{definition}
Notice that in the diagonal, $k= l$, one recovers the standard persymmetry property of the  diagonal square blocks.

\begin{pro}The moment matrix $G$ is a  persymmetric semi-infinite matrix
\begin{align}\label{eq:persymmetry_moment}
  \eta G \eta =G^\top.
\end{align}
 \end{pro}
 \begin{proof}
To prove it we perform the following sequence of equalities
\begin{align*}
  \eta G\eta=&\oint_{\T^D}\eta\chi(\z(\boldsymbol\theta))\d\mu(\boldsymbol\theta)\big(\chi(\z(\boldsymbol\theta))\big)^\dagger\eta\\
  =&\oint_{\T^D}\chi(\z(-\boldsymbol\theta))\d\mu(\boldsymbol\theta)\big(\chi(\z(-\boldsymbol\theta))\big)^\dagger\\
  =&\oint_{\T^D}\overline{\chi(\z(\boldsymbol\theta))}\d\mu(\boldsymbol\theta)\big(\chi(\z(\boldsymbol\theta))\big)^\top\\
  =&\Big(\oint_{\T^D}\chi(\z(\boldsymbol\theta))\d\mu(\boldsymbol\theta)\big(\chi(\z(\boldsymbol\theta))\big)^\dagger\Big)^\top\\
  =&G^\top.
\end{align*}
 \end{proof}

 \begin{pro}
 	The following properties hold true
 	\begin{enumerate}
 		\item The matrices $H$ are persymmetric $\eta H \eta = H^{\top}$ and
 		\begin{align}\label{que relacion!}
 		\bar {\hat S} &= \eta S\eta.
 		\end{align}
For real measures we have $\eta H\eta=\bar H$.
 		\item The following is satisfied
 		\begin{align*}
 		\eta \beta \eta&=\bar{\hat \beta}.
 		\end{align*}
 		\item The MVOLPUT fulfill
 		\begin{align}\label{phi-eta}
 		\eta \Phi (\z)&=\bar{\hat \Phi}(\z^{-1}), & \eta \hat \Phi(\z)&= \bar \Phi(\z^{-1}).
 		\end{align}
 	\end{enumerate}
 \end{pro}
 \begin{proof}
 		 From  the persymmetry property \eqref{eq:persymmetry_moment} we get
 		\begin{align*}
 		\eta G
 		=& G^\top \eta,
 		\end{align*}
 		and using  the  Gauss--Borel factorization \eqref{cholesky} we obtain
 		\begin{align*}
\bar{\hat S}  \eta S^{-1} = H^\top  (S^\top)^{-1}\eta\hat S^{\dagger}H^{-1}=\eta,
 		\end{align*}
 		and, as we have a lower triangular matrix on the LHS and a upper triangular matrix on the RHS, the only option for is to be a diagonal matrix, i.e., equal to the reversal matrix $\eta$. 
      \end{proof}

 We discuss now some interesting matrices, that we name as partial exchange or reversal matrices which are useful in the finding of  interesting parity properties of the MVOLPUT.
For that aim we need to introduce
\begin{definition}
We consider the  signature matrices $I_a\in \C^{D\times D}$, $a\in\{1,\dots,D\}$, these are diagonal matrices with their diagonal coefficients being $1$ but  for the
$a$-th entry  which is $-1$.
\end{definition}
With the help of these signature matrices we define
\begin{definition}
We consider partial reversal matrices $\eta_a$, $a\in\{1,\dots,D\}$  which are block diagonal semi-infinite matrices
with coefficients given by
\begin{align*}
 \big(\eta_a\big)_{\boldsymbol{\alpha}^{(k)}_j,\boldsymbol{\alpha}^{(k)}_r}\coloneq
 \delta_{\boldsymbol{\alpha}^{(
 		k)}_j,I_a \boldsymbol{\alpha}^{(k)}_r}=
 \delta_{I_a \boldsymbol{\alpha}^{(k)}_j,\boldsymbol{\alpha}^{(k)}_r}.
\end{align*}
\end{definition}
In Appendix \ref{examples} we give some examples for the cases $D=1,2$ see .

\begin{pro}
For $a,b\in\{1,\dots, D\}$ the partial reversal  matrices fulfill
\begin{align*}
\eta_a \eta_b&=\eta_b\eta_a,  &
\eta_a^2 &= \mathbb{I}, &  \prod_{a=1}^{D}\eta_a=\eta.
 \end{align*}
\end{pro}
\begin{proof}
  It is easy to realize that
    \begin{align*}
 \sum_{r=1}^{|[k]|} (\eta_a)_{\boldsymbol{\alpha}^{(k)}_j,\boldsymbol{\alpha}^{(k)}_r}
        (\eta_b)_{\boldsymbol{\alpha}^{(k)}_r,\boldsymbol{\alpha}^{(k)}_l}=
        \delta_{I_a \boldsymbol{\alpha}^{(k)}_j,I_b\boldsymbol{\alpha}^{(k)}_l}=
        \delta_{I_a I_b \boldsymbol{\alpha}^{(k)}_j,\boldsymbol{\alpha}^{(k)}_l}.
\end{align*}
The signature matrices $I_a$ are square roots of the identity  $I_a^2=\mathbb{I}$ and so are the corresponding partial exchange matrices $\eta_a$.
 Finally, from
   \begin{align*}
 \sum_{r_1,\dots,r_{D-1}=1}^{|[k]|} (\eta_1)_{\boldsymbol{\alpha}^{(k)}_j,\boldsymbol{\alpha}^{(k)}_{r_1}}(\eta_2)_{\boldsymbol{\alpha}^{(k)}_{r_1},\boldsymbol{\alpha}^{(k)}_{r_2}}\cdots
        (\eta_D)_{\boldsymbol{\alpha}^{(k)}_{r_D-1},\boldsymbol{\alpha}^{(k)}_l}=
        \delta_{I_1\cdots I_D \boldsymbol{\alpha}^{(k)}_j,\boldsymbol{\alpha}^{(k)}_l},
\end{align*}
 but $I_1\cdots I_D=-\I$ and therefore we recover the complete exchange matrix $\eta$.
\end{proof}
%

\begin{definition}
\begin{enumerate}
	\item 	We identify $\{1,\dots,D\}\equiv\Z_D$ as the $D$ cyclic group $\Z_D$, being $D$ the identity  for the sum $a+D=a\, (\operatorname{mod} D)$.
	\item   For each variable $a\in\Z_D$ we introduce the partial reversal operator  $\mathcal P_a:(\C^*)^D\to(\C^*)^D$ given by  
\begin{align*}
	\mathcal P_a(z_1,\dots,z_a,\dots,z_D)\coloneq(z_1,\dots,z_a^{-1},\dots, z_D).
\end{align*}
	\item  This partial reversal operator extends to the algebraic torus $(\C^*)^D$  the partial parity operator in the unit torus $\mathcal P_a:\T^D\to\T^D$, $(\theta_1,\dots,\theta_a,\dots,\theta_D)\mapsto(\theta_1,\dots,-\theta_a,\dots,\theta_D)$.
	\item Its action on the measure will be denoted by  $\mathcal P_a^*\d\mu(\boldsymbol{\theta})=\d\mu(\mathcal P_a\boldsymbol{\theta})$.
	\item Given a subset $\sigma\in 2^{\Z_D}$ we consider the reversal operator $\mathcal{P}_\sigma\coloneq\prod\limits_{a\in\sigma}\mathcal P_a$ defined in the algebraic torus,
	the partial reversal matrix $\eta_\sigma\coloneq \prod\limits_{a \in \sigma} \eta_a$, the complex vector  $\z_ {\sigma}\coloneq \mathcal P_\sigma\z\in(\C^*)^D$
	 and the corresponding moment matrix $G^{(\sigma)}\coloneq G_{\mathcal P^*_\sigma\mu}$.
	\item The  $LU$ factorization $G^{(\sigma)}=\big(S^{(\sigma)}\big)^{-1}H^{(\sigma)} \Big(\big(\hat S^{(\sigma)}\big)^{-1}\Big)^\dagger$ leads to the corresponding partial reversal transformed MVOLPUT $\Phi^{(\sigma)}\coloneq S^{(\sigma)}\chi$ and $\hat\Phi^{(\sigma)}\coloneq \hat S^{(\sigma)}\chi$ and second kind functions $\mathcal C^{(\sigma)}\coloneq \Big(\big(S^{(\sigma)}\big)^{-1}\Big)^\dagger\chi$ and $\hat {\mathcal C}^{(\sigma)}\coloneq \Big(\big(\hat S^{(\sigma)}\big)^{-1}\Big)^\dagger\chi$.
		\item   For a subset $\sigma\in 2^{\Z_D}$  we consider its complement $ \complement\sigma$  defined by the following two conditions: $\sigma \bigcup\complement \sigma=\{1,2,\dots,D\}$ and
		$\sigma \bigcap \complement\sigma=\varnothing$.
\end{enumerate}	

\end{definition}
\begin{pro}
\begin{enumerate}
	\item Given a subset $\sigma\subseteq\Z_D$ the partial reversal matrices $\eta_\sigma$ act on the monomial vector as follows
	\begin{align*}
\eta_\sigma\chi(\z)=\chi(\z_\sigma).
	\end{align*}
The above property models, in  unit torus variables $\boldsymbol\theta$,   the partial parity operation of reversing the signs of appropriate angles  $\theta$'s.
\item  The following equations holds true
 \begin{align*}
G^{(\sigma)}=  \eta_\sigma G \eta_\sigma= \eta_{{\complement\sigma}} G^\top  \eta_{{\complement\sigma}}=\Big(G^{({\complement\sigma})}\Big)^\top.
\end{align*}
\item  The partial reversal transformed MVOLPUT and second kind functions are connected to the original MVOLPUT and second kind functions by
\begin{align*}
  \Phi^{(\sigma)}(\z)=& \eta_{{\complement\sigma}}\bar {\hat \Phi}(\z_{\complement\sigma}), &
  \hat \Phi^{(\sigma)}(\z)=& \eta_{{\complement\sigma}}\bar \Phi( \z_{\complement\sigma}), &
    \mathcal C^{(\sigma)}(\z)=& \eta_{{\complement\sigma}}\bar {\hat {\mathcal C}}(\z_{\complement\sigma}), &
    \hat {\mathcal C}^{(\sigma)}(\z)=&  \eta_{{\complement\sigma}}\bar {\mathcal C}( \z_{\complement\sigma}).
\end{align*}
\end{enumerate}
\end{pro}
\begin{proof}
We only prove the third statement being the others straightforward. Assuming the factorization
$G^{(\sigma)}=\big(S^{(\sigma)}\big)^{-1}H^{(\sigma)} \Big(\big(\hat S^{(\sigma)}\big)^{-1}\Big)^\dagger$  and recalling $  \eta_{{\complement\sigma}}G   \eta_{{\complement\sigma}}=\big(G^{(\sigma)}\big)^{\top} $ we conclude
\begin{align*}
  \eta_{{\complement\sigma}}=\bar{\hat{S}}^{(\sigma)}  \eta_{{\complement\sigma}} S^{-1}&=\big(H^{(\sigma)}\big)^{\top} \left(\bar{\hat{S}}   \eta_{{\complement\sigma}} \big(S^{(\sigma)}\big)^{-1} \right)^{\top} H^{-1}.
\end{align*}
Therefore,
\begin{align*}
S^{(\sigma)}=&  \eta_{{\complement\sigma}}\bar{\hat{S}}  \eta_{{\complement\sigma}}, &
H^{(\sigma)}=& \eta_{{\complement\sigma}} H^{\top}  \eta_{{\complement\sigma}}, &\hat S^{(\sigma)}=&  \eta_{{\complement\sigma}}\bar S  \eta_{{\complement\sigma}},
\end{align*}
and the result follows.
\end{proof}
\subsection{Plemej type formulae for MVOLPUT }

Here we discuss integral expressions for the second kind functions and Plemej type formul{\ae}. We follow the papers by Mohammed \cite{mohammed} and \cite{mohammed2}.

\begin{definition}\label{orthant}
 For each subset $\sigma\in 2^{\Z_D}$ we define its right boundary by
	\begin{align*}
	\partial\sigma\coloneq&\big\{i\in\Z_D: i\in\sigma \text{ and } i+1\in\complement\sigma\big\}.
	\end{align*}
	Then,  we introduce the orthant polydisk as the  polydomain $(\D^D)_\sigma \coloneq\bigtimes\limits_{i=1}^D A_i$ with
$A_i\coloneq\C\setminus \bar\D$ if $i\in\sigma$ and $A_i\coloneq\D$ when $i\in\complement\sigma$
and define, for $\sigma\neq \Z_D$, the corresponding integer  orthants
\begin{align*}
(\Z^D)_\sigma\coloneq &\bigtimes_{i=1}^D Z_i, &
Z_i\coloneq&\ccases{
	\Z_-, & i\in(\sigma\setminus\partial\sigma),\\
	\Z_<, & i\in\partial\sigma,\\
	\Z_+, & i\in(\complement \sigma\setminus\partial(\complement\sigma)),\\
	\Z_>, &  i\in \partial\complement\sigma,
	}
\end{align*}
where $\Z_\gtrless\coloneq\Z_\pm\setminus\{0\}$. For $\sigma=\Z_D$ we define
\begin{align*}
(\Z^D)_{\Z_D}\coloneq\big(\Z_-\big)^D\setminus\{\boldsymbol 0\},
\end{align*}
where $\boldsymbol 0\in\Z^D$ is the zero vector.
\end{definition}

\begin{pro}\label{pro:hyperoctants}
	The set of multi-indices splits into disjoint integer  orthants
	\begin{align*}
	\Z^D=&\bigcup_{\sigma\in 2^{\Z_{D}}}(\Z^D)_\sigma, & \text{ with $(\Z^D)_\sigma\cap(\Z^D)_{\sigma'}=\emptyset$  for
		$\sigma\neq\sigma'$.}
	\end{align*}	
\end{pro}\begin{proof}
See Appendix \ref{proof3}.
\end{proof}

\begin{definition}\label{def:cauchy kernel}
For each subset $\sigma\in 2^{\Z_D}$	we define the Fourier orthant vector
\begin{align*}
(\chi_\sigma)_\q \coloneq\ccases{
	0, & \q\in\Z^D\setminus(\Z^D)_\sigma,\\
	\z^\q, & \q\in (\Z^D)_\sigma,
}
\end{align*}
and the Cauchy-Mohammed kernel $C_\sigma(\z,\boldsymbol \zeta)$, for $\boldsymbol \zeta\in\T^D$ and $\z\in(\D^D)_\sigma$, as follows
\begin{align*}
C_\sigma(\z,\boldsymbol \zeta)\coloneq &\big(\chi(\boldsymbol{\zeta})\big)^\dagger\chi_\sigma(\z).
\end{align*}
\end{definition}
\begin{pro}
	We have
	\begin{align*}
	\chi=\sum_{\sigma\in 2^{\Z_D}}\chi_\sigma.
	\end{align*}
\end{pro}
\begin{proof}
	It is a consequence of Proposition \ref{pro:hyperoctants}.
\end{proof}
\begin{pro}\label{pro:cauchy-mohammed}
	The Cauchy--Mohammed kernel has the following expression
\begin{align*}
C_\sigma(\z,\boldsymbol \zeta)=\ccases{
	\Big(\prod\limits_{i\in(\sigma\setminus\partial\sigma)}\dfrac{z_i}{z_i-\zeta_i}\Big)
\Big(\prod\limits_{i\in\partial\sigma}\dfrac{\zeta_i}{z_i-\zeta_i}\Big)
\Big(\prod\limits_{i\in(\complement\sigma\setminus\partial\complement\sigma)}\dfrac{\zeta_i}{\zeta_i-z_i}\Big)
\Big(\prod\limits_{i\in\partial\complement\sigma}\dfrac{z_i}{\zeta_i-z_i}\Big), & \sigma\neq\Z_D,\\
\Big[\Big(\prod\limits_{i\in(\sigma\setminus\partial\sigma)}\dfrac{z_i}{z_i-\zeta_i}\Big)-1\Big], &\sigma=\Z_D,
}
\end{align*}
for $\boldsymbol{\zeta}\in\T^D$ and $\z\in(\D^D)_\sigma$.
\end{pro}
\begin{proof}
	For $\sigma\neq\Z_D$ from the Definition \ref{def:cauchy kernel} we have
	\begin{align*}
C_\sigma(\z,\boldsymbol \zeta)=& 	\sum_{\q\in(\Z^D)_\sigma}\boldsymbol{\zeta}^{-\q}\z^\q\\
=&\prod_{i=1}^{D}\big(\sum_{\alpha_i\in Z_i} \zeta_i^{-\alpha_i}z_i^{\alpha_i}\big).
\end{align*}
Now, we deduce that
\begin{align*}
\sum_{\alpha_i\in Z_i} \zeta_i^{-\alpha_i}z_i^{\alpha_i}=&
\ccases{
\sum\limits_{n=0}^\infty \big(\zeta_iz_i^{-1}\big)^n	, & i\in(\sigma\setminus\partial\sigma),\\
\sum\limits_{n=1}^\infty \big(\zeta_iz_i^{-1}\big)^n, & i\in\partial\sigma,\\
\sum\limits_{n=0}^\infty \big(\zeta_i^{-1}z_i\big)^n, & i\in(\complement \sigma\setminus\partial(\complement\sigma)),\\
\sum\limits_{n=1}^\infty \big(\zeta_i^{-1}z_i\big)^n, &  i\in \partial\complement\sigma,
}
	\end{align*}
hence, using that $\z\in(\D^D)_\sigma$ and  the basic results on geometrical series we find
	\begin{align*}
	C_\sigma(\z,\boldsymbol \zeta)=	\Big(\prod_{i\in(\sigma\setminus\partial\sigma)}\frac{1}{1-\zeta_iz_i^{-1}}\Big)
	\Big(\prod_{i\in\partial\sigma}\frac{\zeta_iz_i^{-1}}{1-\zeta_iz_i^{-1}}\Big)
	\Big(\prod_{i\in(\complement\sigma\setminus\partial\complement\sigma)}\frac{1}{1-\zeta_i^{-1}z_i}\Big)\Big(\prod_{i\in\partial\complement\sigma}\frac{\zeta_i^{-1}z_i}{1-\zeta_i^{-1}z_i}\Big),
	\end{align*}
	and we get the result. For $\sigma=\Z_D$ the result follows similarly.
\end{proof}
 This result shows that  these integral kernels  $(-1)^{|\sigma|}C_\sigma$ coincide with the ones in \cite{mohammed,mohammed2}.
%

\begin{definition}\label{def:mohammed-hyperoctants}
	Given $\z\in(\D^D)_\sigma$ we introduce  the corresponding orthant components of the second kind functions
	\begin{align*}
	{\mathcal C}_{\sigma}(\z)\coloneq&\big(S^{-1}\big)^\dagger\chi_\sigma(\z), &
	\hat{\mathcal C}_{\sigma}(\z)\coloneq&\big(\hat S^{-1}\big)^\dagger\chi_\sigma(\z).
	\end{align*}
\end{definition}
Observe that we could also write
\begin{pro}
		For $\z\in(\D^D)_\sigma$ the  orthant components of the second kind functions are integral transforms according to the Cauchy--Mohammed kernels
		\begin{align*}
		{\mathcal C}_{\sigma}(\z)
		=&\big(H^{-1}\big)^\dagger\oint_{\T^D}\hat \Phi(\Exp{\operatorname{i}\boldsymbol{\theta}})
		\d\bar\mu(\boldsymbol{\theta}) C_\sigma\big(\z,\Exp{\operatorname{i}\boldsymbol{\theta}}\big), &
		\hat{\mathcal C}_{\sigma}(\z)
		= & H^{-1}\oint_{\T^D} \Phi(\Exp{\operatorname{i}\boldsymbol{\theta}})
		\d\mu(\boldsymbol{\theta}) C_\sigma\big(\z,\Exp{\operatorname{i}\boldsymbol{\theta}}\big).
		\end{align*}
\end{pro}
\begin{proof}
	It follows  from Definition \ref{def:mohammed-hyperoctants} that
	\begin{align*}
	{\mathcal C}_{\sigma}(\z)=&\big(H^{-1}\big)^\dagger\oint_{\T^D}\hat S\chi(\Exp{\operatorname{i}\boldsymbol{\theta}})\d\bar\mu(\boldsymbol{\theta})\Big(\chi(\Exp{\operatorname{i}\boldsymbol{\theta}})\Big)^\dagger\chi_\sigma(\z),
	&
	\hat{\mathcal C}_{\sigma}(\z)=&H^{-1}\oint_{\T^D}S\chi(\Exp{\operatorname{i}\boldsymbol{\theta}})\d\mu(\boldsymbol{\theta})\Big(\chi(\Exp{\operatorname{i}\boldsymbol{\theta}})\Big)^\dagger\chi_\sigma(\z).
	\end{align*}
	and Proposition \ref{pro:cauchy-mohammed} gives the desired result.
\end{proof}

\begin{definition}
	Given a function $f:\T^D\to\C$ the function  $f^{[\sigma]}$ denotes its holomorphic extension to $(\D^D)_\sigma$ having $f$ as its boundary value at $\T^D$.
\end{definition}

\begin{theorem}[Plemej formul{\ae}]
Take an absolutely continuous measure  $\d\mu(\boldsymbol{\theta})=\hat{\mu}(\boldsymbol{\theta})\d\boldsymbol{\theta}$ and assume  that Hölder continuity  for $\Phi_\q\hat{\mu},\hat\Phi_\q\bar{\hat{\mu}} \in C^\alpha(\T^D,\C)$,  $0<\alpha<1$, holds. Then,  we have
	\begin{multline*}
\lim_{\substack{\z\in(\D^D)_\sigma\\\z\to\boldsymbol{\zeta}}}\oint_{\T^D}\hat
	\Phi_\q(\Exp{\operatorname{i}\boldsymbol{\theta}})C_\sigma(\z,
	\Exp{\operatorname{i}\boldsymbol{\theta}})
	\d\bar\mu(\boldsymbol{\theta})+	
	\lim_{\substack{\z\in(\D^D)_{\complement\sigma}\\\z\to\boldsymbol{\zeta}}}
	\oint_{\T^D}\hat\Phi_\q(\Exp{\operatorname{i}\boldsymbol{\theta}})
	C_{\complement\sigma}(\z,\Exp{\operatorname{i}\boldsymbol{\theta}})
	\d\bar\mu(\boldsymbol{\theta})\\=(2\uppi)^D\Big(\big({\hat\Phi}_\q{\bar{\hat{\mu}}}\big)^{[\sigma]}(\boldsymbol{\zeta})
	+\big({\hat\Phi}_\q{\bar{\hat{\mu}}}\big)^{[\complement\sigma]}(\boldsymbol{\zeta})\Big),
	\end{multline*}
	and
	\begin{multline*}
	\lim_{\substack{\z\in(\D^D)_\sigma\\\z\to\boldsymbol{\zeta}}}\oint_{\T^D}\Phi_\q(\Exp{\operatorname{i}\boldsymbol{\theta}})C_\sigma(\z,\Exp{\operatorname{i}\boldsymbol{\theta}})\d\mu(\boldsymbol{\theta})+
	\lim_{\substack{\z\in(\D^D)_{\complement\sigma}\\\z\to\boldsymbol{\zeta}}}\oint_{\T^D}\Phi_\q(\Exp{\operatorname{i}\boldsymbol{\theta}})C_{\complement\sigma}(\z,\Exp{\operatorname{i}\boldsymbol{\theta}})\d\mu(\boldsymbol{\theta})\\=(2\uppi)^D\Big(\big({\Phi}_\q{\hat{\mu}}\big)^{[\sigma]}(\boldsymbol{\zeta})
	+\big({\Phi}_\q{\hat{\mu}}\big)^{[\complement\sigma]}(\boldsymbol{\zeta})\Big),
	\end{multline*}
	for $\boldsymbol{\zeta}\in\T^D$ and $\q\in\Z^D$. 
	Equivalently,
	\begin{align*}
	\sum_{\sigma\in 2^{\Z_D}}	\lim_{\substack{\z\in(\D^D)_\sigma\\\z\to\boldsymbol{\zeta}}}\oint_{\T^D}\hat\Phi_\q(\Exp{\operatorname{i}\boldsymbol{\theta}})C_\sigma(\z,\Exp{\operatorname{i}\boldsymbol{\theta}})\d\bar\mu(\boldsymbol{\theta})=&(2\uppi)^D{\hat\Phi}_\q(\boldsymbol{\zeta}){\bar{\hat{\mu}}}(\boldsymbol{\zeta}),\\\sum_{\sigma\in 2^{\Z_D}}	
	\lim_{\substack{\z\in(\D^D)_\sigma\\\z\to\boldsymbol{\zeta}}}\oint_{\T^D}\Phi_\q(\Exp{\operatorname{i}\boldsymbol{\theta}})C_\sigma(\z,\Exp{\operatorname{i}\boldsymbol{\theta}})\d\mu(\boldsymbol{\theta})=&(2\uppi)^D{\Phi}_\q(\boldsymbol{\zeta})\hat\mu(\boldsymbol{\zeta}).
	\end{align*}
\end{theorem}
\begin{proof}
	For the proof we need some observations regarding \cite{mohammed} and \cite{mohammed2}. Obviously,   the MVOLPUT belong to the Wiener algebra of the $D$-dimensional unit torus, 
	\begin{align*}
	\Phi_\q,\hat{\Phi}_\q\in W^D=\Big\{f(\z)=\sum_{\q\in\Z^D}c_\q \z^\q: \quad \|f\|_{W^D}\coloneq \sum_{\q\in\Z^D}|c_\q|<\infty \Big\}.
	\end{align*}
	We assume that $\d\mu=\hat\mu(\boldsymbol{\theta})\d\boldsymbol\theta$ with $\hat\mu\in W^D$, and consequently, given that we are dealing with Laurent polynomials,  $\Phi_\q\hat\mu, \hat\Phi_\q\hat\mu\in W^D$.
	Following \cite{mohammed,mohammed2} for any Hölder continuous complex function on the unit torus $\varphi\in C^\alpha(\T^D,\C)$, $0<\alpha<1$, one defines its integral transformations $\phi^{(\sigma)}$ and its boundary integral conjugates $\phi_*^{(\sigma)}$ as follows
	\begin{align*}
	\phi^{(\sigma)}(\z)\coloneq& \frac{1}{(2\uppi)^D}\oint_{\T^D}\varphi\big(\Exp{\operatorname{i}\boldsymbol{\theta}}\big)C_\sigma\big(\z,\Exp{\operatorname{i}\boldsymbol{\theta}}\big)\d\boldsymbol{\theta}, & \z&\in(\D^D)_\sigma,\\
	\phi_*^{(\sigma)}(\boldsymbol\eta)\coloneq& \frac{1}{\uppi^D}\oint_{\T^D}\phi^{(\sigma)}\big(\Exp{\operatorname{i}\boldsymbol{\theta}}\big)C\big(\Exp{\operatorname{i}\boldsymbol{\theta}},\boldsymbol\eta\big)\d\boldsymbol{\theta}, & \boldsymbol\eta&\in\T^D,
	\end{align*}
	where
	\begin{align*}
	C(\boldsymbol\zeta,\boldsymbol{\eta})\coloneq\prod_{i=1}^D\frac{1}{\zeta_i-\eta_i}.
	\end{align*}
Observe that
	\begin{align*}
	\phi^{(\sigma)}(\boldsymbol\eta)=& \frac{1}{\uppi^D}\oint_{\T^D}\phi_*^{(\sigma)}\big(\Exp{\operatorname{i}\boldsymbol{\theta}}\big)C\big(\Exp{\operatorname{i}\boldsymbol{\theta}},\boldsymbol\eta\big)\d\boldsymbol{\theta}, & \boldsymbol\eta&\in\T^D,
	\end{align*}
	and for $\boldsymbol{\zeta}\in\T^D$ we have Plemej's type formul\ae
	\begin{align*}
	\phi^{(\sigma)}(\boldsymbol{\zeta})+
	\phi^{(\complement\sigma)}(\boldsymbol{\zeta})=&\varphi^{[\sigma]}(\boldsymbol{\zeta})+\varphi^{[\complement\sigma]}(\boldsymbol{\zeta}),\\		\sum_{\sigma\in 2^{\Z_D}}\phi_*^{(\sigma)}(\boldsymbol{\zeta})=&2^D\phi(\boldsymbol{\zeta}).
		\end{align*}
	where the holomorphic extension $\varphi^{[\sigma]}(\z)$ to $(\D^D)_\sigma$  is obtained from the Fourier series of $\varphi$ by disregarding the mult-indices not in $(\Z^D)_\sigma$; this   holomorphic  function in $(\D^D)_\sigma$  has $\varphi$ as its boundary value at the unit torus $\T^D$. Thus, the orthant expansion
	\begin{align*}
	\sum_{\sigma\in 2^{\Z_D}}\phi^{(\sigma)}(\boldsymbol{\zeta})=2^D\varphi(\boldsymbol{\zeta})
	\end{align*}
holds. 
\end{proof}
This Plemej formul{\ae} can be recasted for the second kind functions as follows
\begin{cor}
For any point in the unit torus,  $\boldsymbol{\zeta}\in\T^D$, the MVOLPUT and its second kind functions   satisfy  
	\begin{align*}
	\sum_{\sigma\in 2^{\Z_D}}\lim_{\substack{\z\in(\D^D)_\sigma\\\z\to\boldsymbol{\zeta}}}\mathcal C_\sigma (\z) =& (2\uppi)^DH^\dagger\hat{\Phi}(\boldsymbol{\zeta})\bar{\hat{\mu}}(\boldsymbol{\zeta}),&
	\sum_{\sigma\in 2^{\Z_D}}\lim_{\substack{\z\in(\D^D)_\sigma\\\z\to\boldsymbol{\zeta}}}\hat{\mathcal C}_\sigma (\z) =& (2\uppi)^DH {\Phi}(\boldsymbol{\zeta}){\hat{\mu}}(\boldsymbol{\zeta}),
	\end{align*}
	or equivalently
	\begin{align*}
	\lim_{\substack{\z\in(\D^D)_\sigma\\\z\to\boldsymbol{\zeta}}}\mathcal C_\sigma (\z) +  \lim_{\substack{\z\in(\D^D)_{\complement\sigma}\\\z\to\boldsymbol{\zeta}}}\mathcal  C_{\complement\sigma}(\z) =&(2\uppi)^D H^\dagger\Big(\big(\hat{\Phi}\bar{\hat{\mu}}\big)^{[\sigma]}(\boldsymbol{\zeta}) +\big(\hat{\Phi}\bar{\hat{\mu}}\big)^{[\complement\sigma]}(\boldsymbol{\zeta})\Big),\\
	\lim_{\substack{\z\in(\D^D)_\sigma\\\z\to\boldsymbol{\zeta}}} \hat{\mathcal C}_\sigma (\z) +
	\lim_{\substack{\z\in(\D^D)_{\complement\sigma}\\\z\to\boldsymbol{\zeta}}}	\hat{\mathcal C}_{\complement\sigma}(\z) =& (2\uppi)^DH\Big(\big({\Phi}{\hat{\mu}}\big)^{[\sigma]}(\boldsymbol{\zeta})
	+\big({\Phi}{\hat{\mu}}\big)^{[\complement\sigma]}(\boldsymbol{\zeta})\Big).
	\end{align*}
\end{cor}

 \subsection{String equations for the moment matrix}

 \begin{definition}\label{def:upsilon}
 We introduce the spectral matrices
 $\Upsilon_a$,  $a\in\{1,\dots,D\}$,   which are tri-diagonal block matrices
\begin{align*}
  \Upsilon_a\coloneq&\PARENS{
  \begin{matrix}
    0 & (\Upsilon_a)_{[0],[1]}&0&0&\dots\\
    (\Upsilon_a)_{[1],[0]} &0& (\Upsilon_a)_{[1],[2]}&0&\dots\\
    0&    (\Upsilon_a)_{[2],[1]} &0& (\Upsilon_a)_{[2],[3]}\\
    0& 0&    (\Upsilon_a)_{[3],[1]} &0&\ddots\\
    \vdots&\vdots&&\ddots&\ddots
  \end{matrix}
  },
\end{align*}
with entries given by
   \begin{align*}
 (\Upsilon_a )_{\boldsymbol{\alpha}^{(k)}_i,\boldsymbol{\alpha}^{(k+1)}_j}\coloneq&
 \delta_{\boldsymbol{\alpha}^{(k)}_i+\boldsymbol{e}_a,\boldsymbol{\alpha}^{(k+1)}_j}=
 \delta_{\boldsymbol{\alpha}^{(k)}_i,\boldsymbol{\alpha}^{(k+1)}_j-\boldsymbol{e}_a},
  \\
  (\Upsilon_a )_{\boldsymbol{\alpha}^{(k+1)}_i,\boldsymbol{\alpha}^{(k)}_j}\coloneq&
 \delta_{\boldsymbol{\alpha}^{(k+1)}_i+\boldsymbol{e}_a,\boldsymbol{\alpha}^{(k)}_j}=
 \delta_{\boldsymbol{\alpha}^{(k+1)}_i,\boldsymbol{\alpha}^{(k)}_j-\boldsymbol{e}_a}.
\end{align*}
where $\{\boldsymbol{e}_a\}_{a=1}^D$ is the canonical basis in $\C^D$.
 \end{definition}
 \begin{pro}\label{importante}
The spectral matrices satisfy
\begin{align*}
  \Upsilon_a\Upsilon_b=&\Upsilon_b\Upsilon_a, &
  \Upsilon_a^{-1}=&\Upsilon_a^\top, & \eta_a \Upsilon_b&=\ccases{
                     \Upsilon_b \eta_a & a \neq b,\;  a,b\in\{1,\dots, D\}, \\
                      \Upsilon_b^{-1} \eta_a & a=b,
                   }, 
\end{align*}
and the  very important spectral property
\begin{align*}
\Upsilon_a\chi(\z)=&z_a\chi(\z).
\end{align*} 
 \end{pro}
 \begin{proof}
 	See Appendix \ref{proof4}.
 \end{proof}

For examples with $D=1,2$ see Appendix \ref{examples}.

 \begin{definition}\label{def:upsilon power}
   For any $\q=(\alpha_1,\dots,\alpha_D)\in\Z^D$ we introduce the matrix
   \begin{align}\label{upsilon_alpha}
     \Upsilon_{\q}\coloneq \prod_{a=1}^D\big(\Upsilon_a\big)^{\alpha_a}.
   \end{align}
 \end{definition}
 \begin{pro}\label{pro:upsilon}
   For the spectral matrices the following properties hold
   \begin{enumerate}
     \item Given a multi-index $\q\in\Z^D$ we have $\Upsilon_{-\q}=(\Upsilon_\q)^{-1}=\big(\Upsilon_\q\big)^\top=\eta\Upsilon_\q\eta$; i.e., $\Upsilon_\q$ is an orthogonal persymmetric block matrix.
         \item $\Upsilon_\q$ is a $2|\q|+1$-diagonal block matrix, with at most $|\q|$ block superdiagonals and $|\q|$ block subdiagonal. If $|\q|$ is odd, the main diagonal and the even superdiagonals and subdiagonals vanish, if $|\q|$ is even then the odd block superdiagonals and subdiagonals are zero.
     \item For any couple of multi-indices $\q_1,\q_2\in\Z^D$ we have $\Upsilon_{\q_1+\q_2}=\Upsilon_{\q_1}\Upsilon_{\q_2}$.
     \item The following eigen-value property holds $\Upsilon_\q \chi(\z)=\z^\q\chi(\z)$.
   \end{enumerate}
 \end{pro}

  \begin{proof}
  \begin{enumerate}
\item From \eqref{upsilon_alpha} and the properties of the exchange matrix $\eta$ we get
\begin{align*}
  \eta    \Upsilon_{\q}\eta=\prod_{a=1}^D\big(\eta\Upsilon_a\eta\big)^{\alpha_a}
\end{align*}
and recalling that $\{\Upsilon_a\}_{a=1}^D$ is an Abelian set of real persymmetric orthogonal matrices; i.e., $\eta\Upsilon_a\eta=\Upsilon_a^\top=\Upsilon_a^{-1}$ we deduce that we get
\begin{align*}
  \eta    \Upsilon_{\q}\eta=&\prod_{a=1}^D\big(\Upsilon_a^\top\big)^{\alpha_a}= \prod_{a=1}^D\big(\Upsilon_a^{-1}\big)^{-\alpha_a}
  =\prod_{a=1}^D\big(\Upsilon_a\big)^{-\alpha_a}\\
  =&\Upsilon_\q^\top=\Upsilon_{\q}^{-1}=\Upsilon_{-\q}.
\end{align*}
\item As $\Upsilon_a$ are tri-diagonal block matrices then the product of $|\q|$ of them will be a $2|\q|+1$-diagonal block matrix, and recalling that the main diagonal vanishes in $\Upsilon_a$ we deduce the rest of the statement.
    \item It is a consequence of \eqref{upsilon_alpha} and the commutativity of the $\Upsilon_a$ among them.
    \item Follows also from the eigen-value property described in Proposition \ref{importante}.
    \end{enumerate}
 \end{proof}

We are now ready to show a very important symmetry

\begin{pro}[The string equations]
 For any multi-index $\q\in\Z^D$ the moment matrix satisfies
 \begin{align}\label{string}
  \Upsilon_\q G=G\Upsilon_\q.
 \end{align}
\end{pro}
\begin{proof}
  It is a consequence of \eqref{moment} and the following sequence of equalities
  \begin{align*}
    \Upsilon_\q G=&\Upsilon_\q\oint_{\T^D}\chi\big(\z(\boldsymbol\theta)\big)\d\mu(\boldsymbol\theta)\big(\chi\big(\z(-\boldsymbol\theta)\big)\big)^\top\\=&
    \oint_{\T^D} \Upsilon_\q\chi\big(\z(\boldsymbol\theta)\big)\d\mu(\boldsymbol\theta)\big(\chi\big(\z(-\boldsymbol\theta)\big)\big)^\top
  \\=&
     \oint_{\T^D}\Exp{\ii\q\cdot\boldsymbol\theta}\chi\big(\z(\boldsymbol\theta)\big)\d\mu(\boldsymbol\theta)\big(\chi\big(\z(-\boldsymbol\theta)\big)\big)^\top\\=&
   \oint_{\T^D}\chi\big(\z(\boldsymbol\theta)\big)\d\mu(\boldsymbol\theta)\big(\Exp{\ii\q\cdot\boldsymbol\theta}\chi\big(\z(-\boldsymbol\theta)\big)\big)^\top\\=&
  \oint_{\T^D}\chi\big(\z(\boldsymbol\theta)\big)\d\mu(\boldsymbol\theta)\big(\Upsilon_\q^\top\chi\big(\z(-\boldsymbol\theta)\big)\big)^\top\\=&
  \oint_{\T^D}\chi\big(\z(\boldsymbol\theta)\big)\d\mu(\boldsymbol\theta)\big(\chi\big(\z(-\boldsymbol\theta)\big)\big)^\top\Upsilon_\q\\
                 =&G\Upsilon_\q.
  \end{align*}
\end{proof}

\subsection{Three-term relations}

The symmetries of the moment matrix  suggest
\begin{definition}\label{def:jacobi}
We consider the Jacobi matrices
 \begin{align*}
  J_\q&\coloneq S\Upsilon_\q S^{-1}, &
    \hat J_\q&\coloneq\hat S\Upsilon_\q\hat S^{-1}
    \end{align*}
    and the reversal Jacobi matrices
    \begin{align*}
 C_\q&\coloneq \bar{\hat S}\eta
 \Upsilon_\q S^{-1}.
 \end{align*}
 \end{definition}

 \begin{pro}\label{JC}
  The following properties hold true
  \begin{enumerate}
    \item $J_\q=H\hat J_{-\q}^\dagger H^{-1}$.
    \item $J_\q$  is a $(2|\q|+1)$-diagonal block semi-infinite matrix
    \item We have $J_0=\I$ .
    \item For any couple of multi-indices $\q_1,\q_2\in\Z^D$ we have $J_{\q_1+\q_2}=J_{\q_1}J_{\q_2}$. In particular,  $(J_\q)^{-1}=J_{-\q}$ $\forall\q\in\Z^D$.
    \item We have
    \begin{align*}
 C_{\q}&=H^{\top} C_{\q}^{\top} H^{-1},  & C_0&=\eta, &
 C_{\q}^{-1}&=S \eta \Upsilon_{\q}  \big(\bar {\hat S}\big)^{-1},
\end{align*}
and consequently $C_\q$ is a $(2|\q|+1)$-diagonal block semi-infinite matrix.
    \item The following properties are satisfied
    \begin{align*}
 \eta J_{\q}&=C_{\q}, & \eta \bar{ \hat J}_\q&=C_{\q}^{-1}, & \eta J_{\q} \eta&=\bar {\hat J}_{-\q}.
\end{align*}
      \end{enumerate}
\end{pro}
\begin{proof}
	See Appendix \ref{proof5}.
\end{proof}

\begin{pro}
The Jacobi matrices satisfy the  following eigen-value relations
 \begin{align*}
  J_{\q} \Phi(\z)&=\z^{\q} \Phi(z), & \hat J_{\q} \hat \Phi (\z)&= \z^{\q} \hat \Phi (\z),
 \end{align*}
 and the following relations are fulfilled as well
 	\begin{align*}
 		C_{\q} \Phi (\z)&=\z^{\q} \bar{\hat \Phi} (\z^{-1}), &
 		C_{\q}^{-1} \bar {\hat \Phi} (z)&= z^{\q} \Phi (\z^{-1}).
 	\end{align*}
 \end{pro}

The basic Jacobi matrices $J_a=S\Upsilon_aS^{-1}$ are tri-diagonal block semi-infinite matrices
\begin{align*}
J_a&=\PARENS{
 \begin{matrix}
(J_a)_{[0],[0]}  &(J_a)_{[0],[1]}   &                0          &     0                      &           0         &\dots\\
(J_a)_{[1],[0]}    &(J_a)_{[1],[1]}        & (J_a)_{[1],[2]}    &           0         & 0 &\dots\\
                     0&  (J_a)_{[2],[1]} &  (J_a)_{[2],[2]}      &  (J_a)_{[2],[3]}       &  0                  &\dots\\
           0          &          0           & (J_a)_{[3],[2]}       &  (J_a)_{[3],[3]}       &(J_a)_{[3],[4]}  & \\
           \vdots          &        \vdots             &          \;\;  \;\; \;\;\;\;  \;\; \ddots               &      \;\;  \;\; \;\;\;\;  \;\;\ddots      &  \;\;  \;\; \;\;\;\;  \;\; \ddots            &
 \end{matrix}
 }
\end{align*}
with blocks given in terms of quasi-tau matrices and first subdiagonal coefficients by
\begin{align}\label{eq:jacobi}
\ccases{
\begin{aligned}
 (J_{a})_{[0][0]}&=-(\Upsilon_a)_{[0],[1]} \beta_{[1]}=
 -H_{[0]}\hat\beta_{[1]}^\dagger(\Upsilon_a)_{[1],[0]}H_{[0]}^{-1},\\
 (J_a)_{[k],[k-1]}&=H_{[k]}(\Upsilon_a)_{[k],[k-1]}H_{[k-1]}^{-1},  \\
 (J_{a})_{[k][k]}&=\beta_{[k]}(\Upsilon_a)_{[k-1],[k]}-(\Upsilon_a)_{[k],[k+1]}\beta_{[k+1]} \\
 &=H_{[k]}\big[ (\Upsilon_a)_{[k],[k-1]}\hat\beta^\dagger_{[k]}-\hat\beta^\dagger_{[k+1]}(\Upsilon_a)_{[k+1],[k]} \big]H_{[k]}^{-1}\\
 (J_a)_{[k],[k+1]}&=(\Upsilon_a)_{[k],[k+1]}         
\end{aligned} } \end{align}
Similarly, $\hat J_{a}$  has the same form and the blocks are obtained from the previous formul{\ae} with the replacements
\begin{align*}
 H &\longleftrightarrow H^{\dagger}, &  \beta &\longleftrightarrow \hat \beta.
\end{align*}

The matrices $C_a$ are tri-diagonal with blocks given by
\begin{align*}
	(C_{a})_{[0][0]}&=-(\eta \Upsilon_a)_{[0][1]}\beta_{[1]}=-H_{[0]}^{\top}\beta_{[1]}^\top (\eta \Upsilon_a)_{[1][0]}H_{[0]}^{-1},\\
	(C_a)_{[k],[k-1]}&=H_{[k]}^{\top}(\eta \Upsilon_a)_{[k],[k-1]}H_{[k-1]}^{-1},  \\
	(C_{a})_{[k][k]}&=\bar { \hat \beta}_{[k]}(\eta \Upsilon_a)_{[k-1],[k]}-(\eta \Upsilon_a)_{[k],[k+1]}\beta_{[k+1]} \\
	&=H_{[k]}^{\top}\big[ (\eta \Upsilon_a)_{[k],[k-1]}\bar {\hat\beta}^\top_{[k]}-\beta^\top_{[k+1]}(\eta \Upsilon_a)_{[k+1],[k]} \big]H_{[k]}^{-1},\\
	(C_a)_{[k],[k+1]}&=(\eta \Upsilon_a)_{[k],[k+1]},
\end{align*}
and $C_a^{-1}$  is also tri-diagonal with blocks as above with the replacement given by the following prescription
\begin{align*}
	H &\longleftrightarrow H^{\top}, &  \beta &\longleftrightarrow \bar{\hat \beta}.
\end{align*}

\begin{definition} With the notation
	\begin{align*}
	\Upsilon _{-a}&\coloneq \Upsilon_a^{-1}, & a\in\{1,\dots, D\},
	\end{align*}
  the vector
	 $\n=(n_{-1},\dots,n_{-D},n_D,\dots,n_1)^\top\in\C^{2D}$
	 and $\hat{\n}\coloneq\mathcal E_{2D}\bar \n$
	  we consider:
\begin{enumerate}
	\item associated longitude one homogeneous Laurent polynomials
	\begin{align*}
	L_{\n}(\z)&\coloneq  
	\sum_{a=1}^D \frac{n_{-a}}{z_a}+\sum_{a=1}^D n_a z_a,&
	L_{\hat \n}(\z)&\coloneq
	\sum_{a=1}^D \bar n_{-a}z_a+\sum_{a=1}^D \frac{\bar n_a}{ z_a}.
	\end{align*}
\item associated slashed semi-infinite matrices
\begin{align*}
 \centernot\n& \equiv L_{\n}(\bUpsilon)\coloneq
 \sum_{a=1}^D ( n_{-a}\Upsilon_a^{-1}+n_a\Upsilon_a  ),&
 \centernot{\hat{\n} }&\equiv
L_{\hat \n}(\bUpsilon)\coloneq \sum_{a=1}^D \bar n _{-a}\Upsilon_a+\sum_{a=1}^D \bar n^+_a\Upsilon_a^{-1}.
\end{align*}

\end{enumerate}
\end{definition}
Observe that
\begin{pro}\label{dagger-upsilon}
	The following formul{\ae} holds
	\begin{align*}
\centernot\n^\dagger&=\,\centernot {\hat \n}, & \centernot\n \chi(\z)&= n(\z)\chi(\z).
	\end{align*}
\end{pro}

\begin{pro}[Three-term relations]
The MVOLPUT satisfy the following recursion type relations
\begin{align*}
L_{\n}(\z)\phi_{[k]}(\z)& = H_{[k]} \centernot\n_{[k],[k-1]}H_{[k-1]}^{-1}  \phi_{[k-1]}(\z)+
\big(\beta_{[k]}\centernot\n_{[k-1],[k]}-\centernot\n_{[k],[k-1]}\beta_{[k+1]}\big) \phi_{[k]}(\z)
+\centernot\n_{[k],[k+1]}\phi_{[k+1]}(\z),\\
L_{\n}(\z)\hat \phi_{[k]}(\z) &= (H_{[k]})^\dagger\centernot\n_{[k],[k-1]}(H_{[k-1]}^{-1})^\dagger \hat \phi_{[k-1]}(\z)+
\big( \hat\beta_{[k]}\centernot\n_{[k-1],[k]}-\centernot\n_{[k],[k-1]} \hat\beta_{[k+1]}\big)  \hat\phi_{[k]}(\z)
+\centernot\n_{[k],[k+1]} \hat\phi_{[k+1]}(\z),
\end{align*}
and its conjugate partners
{
	\begin{align*}
&\begin{multlined}
L_{\n}(\z)\bar {\hat \phi}_{[k]} (\z^{-1})=
	H_{[k]}^{\top}(\eta \centernot\n)_{[k],[k-1]}H_{[k-1]}^{-1} \phi_{[k-1]} (\z) +\big(\bar { \hat \beta}_{[k]}(\eta \centernot\n)_{[k-1],[k]}-(\eta \centernot\n)_{[k],[k+1]}\beta_{[k+1]}\big)\phi_{[k]}(\z)\\+
	(\eta \centernot\n)_{[k],[k+1]}\phi_{[k+1]} (\z),\
\end{multlined}\\
&\begin{multlined}
L_{\n}(\z)\phi_{[k]} (\z^{-1})=
	H_{[k]} (\eta \centernot\n)_{[k],[k-1]}\big(H_{[k-1]}^{\top}\big)^{-1} \bar {\hat \phi}_{[k-1]} (\z) +\big(\beta_{[k]}(\eta \centernot\n)_{[k-1],[k]}-(\eta \centernot\n)_{[k],[k+1]}\bar { \hat \beta}_{[k+1]}\big)\bar {\hat \phi}_{[k]} (\z)\\+
	(\eta \centernot\n)_{[k],[k+1]}\bar {\hat \phi}_{[k+1]} (\z).
\end{multlined}
\end{align*}}
\end{pro}

\subsection{Christoffel--Darboux formul\ae}
\begin{definition}
The Christoffel--Darboux kernel is defined by
\begin{align*}
K^{(k)}(\z_1,\z_2)\coloneq
  \sum_{ l=0}^{k-1}\left(\hat \phi_{[ l]} (\z_1)\right)^\dagger H_{[ l]}^{-1} \phi_{[ l]}(\z_2)
\end{align*}
\end{definition}
\begin{pro}\label{pro:basis}
Both  sets of MVOLPUT $\big\{\phi_q\big\}_{|\q|\leq k}$ and $\big\{\hat \phi_q\big\}_{|\q|\leq k}$ are linear basis for the linear space $\C_k[\z^{\pm 1}]$ of Laurent polynomials of longitude less or equal to $k$. 
\end{pro}
Thus, a Laurent polynomial $L$ can be always expanded as
\begin{align*}
L(\z)&=\sum_{\substack{\q\in\Z^D\\0\leqslant |\q|\ll\infty}}\hat \lambda_{\q} \hat\phi_{\q}(\z)\\
&=\sum_{\substack{\q\in\Z^D\\0\leqslant |\q|\ll\infty}} \lambda_{\q} \phi_{\q}(\z),
\end{align*}
with coefficients $\hat \lambda_{\q},\lambda_{\q}\in\C$.

The corresponding projections are
\begin{definition}\label{def:projections}
We introduce the two projectors onto $\C^{[k]}[\z^{\pm 1}]$ as the following $k$-th truncated sums  
\begin{align*}
\Pi^{(k)}(L)&\coloneq \sum_{\substack{\q\in\Z^D\\0\leqslant |\q|\leqslant k}}\lambda_{\q}\phi_{\q}(\z),
&
\hat\Pi^{(k)}(L) &\coloneq \sum_{\substack{\q\in\Z^D\\0\leqslant |\q|\leqslant k}}\hat \lambda_{\q}
 \hat\phi_{\q}(\z) ,
\end{align*}
\end{definition}
 Due to the biorthogonality \eqref{eq:biorthogonality} we have
 \begin{pro}\label{pro:linear expansion coeff}
 \begin{enumerate}
 \item The coefficients in the MVOLPUT expansions of a Laurent polynomial $L$ are given by
 \begin{align*}
\lambda_\q&=\sum_{|\q'|=|\q|}\int_{\T^D}L(\Exp{\operatorname{i}\boldsymbol{\theta}})
\overline{\hat\phi_{\q'}(\Exp{\operatorname{i}\boldsymbol{\theta}})} (H^{-1})_{\q',\q}
\d\mu(\boldsymbol\theta), &\hat\lambda_\q&=\sum_{|\q'|=|\q|}(H^{-1})_{\q,\q'}\int_{\T^D}L(\Exp{\operatorname{i}\boldsymbol{\theta}})\overline{\phi_{\q'}(\Exp{\operatorname{i}\boldsymbol{\theta}})}\d\bar\mu(\boldsymbol\theta)
 \end{align*}
 \item
 The projections given in Definition \ref{def:projections} can be expressed in terms of the Christoffel--Darboux kernel as
 \begin{align*}
\Pi^{(k)}(L)(\z)=&\oint_{\T^D} L(\Exp{\operatorname{i}\boldsymbol{\theta}})K^{(k+1)}(\Exp{\operatorname{i}\boldsymbol{\theta}},\z)
\d\mu(\boldsymbol{\theta}),
&
\hat\Pi^{(k)}(L) (\z)=&\oint_{\T^D} \overline{K^{(k+1)}(\z,\Exp{\operatorname{i}\boldsymbol{\theta}})}
 L(\Exp{\operatorname{i}\boldsymbol{\theta}})
\d\bar\mu(\boldsymbol{\theta}).
\end{align*}
\end{enumerate}
 \end{pro}

Therefore, these kernels are  the integral kernels of the projection operators in the space of  multivariate Laurent polynomials of   longitude  equal or less than $k$. This interpretation   leads at once to
 \begin{pro}[Reproducing property] $K^{(k)}(\z_1,\z_2)=\int_{\T^D}K^{(k)}(\z_1,\Exp{\operatorname{i}\boldsymbol{\theta}})K^{(k)}(\Exp{\operatorname{i}\boldsymbol{\theta}},\z_2)\d\mu(\boldsymbol{\theta})$.
  \end{pro}
  From the Gauss--Borel factorization we deduce
\begin{pro}[ABC Theorem]
  An Aitken--Berg--Coller type formula
  \begin{align*}
K^{(k)}(\z_1,\z_2)=\big(\chi^{[k]}(\z_1)\big)^{\dagger} (G^{[k]} )^{-1}\chi^{[k]}(\z_2),
\end{align*}
holds.
\end{pro}

\begin{pro}
  The Christoffel--Darboux  kernel satisfies
  \begin{align*}
K^{(k)}(\z_1,\z_2)
=K^{(k)}(\bar\z_2^{-1},\bar \z_1^{-1}).
  \end{align*}
\end{pro}
\begin{proof}
It is a consequence of \eqref{phi-eta} and the persymmetry of $H$:
\begin{align*}
K^{(k)}(\z_1,\z_2)&=\sum_{ l=0}^{k-1} \big(\phi_{[ l]}(\bar\z_1^{-1})\big)^{\top}\mathcal E_{|[ l]|} H^{-1}_{[ l]} \phi_{[ l]}(\z_2)
\\&
=
\sum_{=0}^{k-1} (\hat \phi_{[ l]}(\z_1))^{\dagger} H^{-1}_{[ l]}\mathcal E_{|[ l]|} \bar{ \hat \phi}_{[ l]}(\z_2^{-1})
\\&
=K^{(k)}(\bar\z_2^{-1},\bar\z_1^{-1}).
\end{align*}
\end{proof}
\begin{pro} The Christoffel--Darboux type  formulae hold
 \begin{align*}
K^{[k]}(\z_1,\z_2)&=\frac{(\hat \phi_{[k]}(\z_1))^{\dagger}\centernot\n_{[k],[k-1]}H_{[k-1]}^{-1}\phi_{[k-1]}(\z_2)-
(\hat \phi_{[k-1]}(\z_1))^{\dagger}H_{[k-1]}^{-1}\centernot\n_{[k-1],[k]}\phi_{[k]}(\z_2)}{L_{\n}(\z_2)-L_{\n}(\bar\z_1^{-1})}\\
&=\frac{(\phi_{[k]}(\bar\z_1^{-1}))^{\top}(\eta \centernot\n)_{[k],[k-1]}H_{[k]}^{-1}\phi_{[k-1]}(\z_2)-
(\phi_{[k-1]}(\z_1^{-1}))^{\top}(H_{[k-1]}^{\top})^{-1}(\eta \centernot\n)_{[k-1],[k]}\phi_{[k]}(\z_1)}{L_{\n}(\z_2)-L_{\n}(\bar\z_1^{-1})}\\
&=\frac{(\hat \phi_{[k]}(\z_1))^{\dagger}(\centernot\n\eta)_{[k],[k-1]}(H_{[k-1]}^{\top})^{-1}\bar { \hat \phi}_{[k-1]}(\z_2^{-1})-
(\hat \phi_{[k-1]}(\z_1))^{\dagger}H_{[k-1]}^{-1}(\centernot\n\eta)_{[k-1],[k]}\bar{ \hat \phi}_{[k]}(\z_2^{-1})}{L_{\n}(\z_2)-L_{\n}(\bar\z_1^{-1})}.
\end{align*}
\end{pro}
\begin{proof}
In order to get the first of the three expressions one has to consider the two possible ways of computing
the expression $(\hat \phi^{[k+1]} (\z_1))^{\dagger} \left[(H^{-1} (L_{\n}(\boldsymbol{J}))^{[k+1]}\right] \phi^{[k+1]}(\z_2)$ letting the operator between brackets
act on the polynomial located on its right or left hand side; i.e.,
\begin{align*}
\begin{multlined}
  (\hat \phi^{[k+1]} (\z_1))^{\dagger} \left[ (H^{-1}(L_{\n}(\boldsymbol{J}))^{[k+1]} \phi^{[k+1]}(\z_2)\right]=L_{\n}(\z_2)(\hat \phi^{[k+1]} (\z_1))^{\dagger}(H^{[k+1]})^{-1}\phi^{[k+1]}(\z_2)\\-
 (\hat \phi_{[k]}(\z_1))^{\dagger} H_{[k]}^{-1} (\Upsilon_{a})_{[k][k+1]} \phi_{[k+1]}(\z_2),
\end{multlined}
\\
\begin{multlined}
 \left[(\hat \phi^{[k+1]} (\z_1))^{\dagger} (H^{-1}L_{\n}(\boldsymbol{J}))^{[k+1]}\right] \phi^{[k+1]}(\z_2)=L_{\n}(\bar \z_1^{-1})(\hat \phi^{[k+1]} (\z_1))^{\dagger}(H^{[k+1]})^{-1}\phi^{[k+1]}(\z_2)\\-
 (\hat \phi_{[k+1]}(\z_1))^{\dagger} (\Upsilon_{a})_{[k+1][k]}H_{[k]}^{-1} \phi_{[k]}(\z_2).
\end{multlined}
\end{align*}
Since both results must coincide, subtracting both we get  the desired expression. The remaining two equalities can
be proven in two different ways. The first and simpler one is to use the expression we have just proven and remember the relations
between $\phi$ and $\hat \phi$ established by $\eta$. A different approach is to compare the two possible ways of computing\begin{align*}
(\phi^{[k+1]} (\bar \z_1^{-1})^{\top} \{([H^{\top}]^{-1}C_a)^{[k+1]}\} \phi^{[k+1]}(\z_2)
\end{align*}
and
\begin{align*}
(\hat \phi^{[k+1]} (\z_1))^{\dagger} \left[([C_a^{-1}]^{\top}[H^{-1}]^{\top})^{[k+1]}\right] \phi^{[k+1]}(\z_2).
\end{align*}
\end{proof}

\section{Darboux transformations}
We now discuss perturbations of the measure given by the multiplication  by a Laurent polynomial. This is the extension, in the multivariate Laurent orthogonal polynomial setting of the Christoffel transformation, or in a geometric-differenttial setting of a Lèvy transformation, or in the theory of integrable systems of a direct Darboux transformation.
\subsection{Laurent polynomial perturbations of the measure}

\begin{definition}
	Given a multivariate Laurent polynomial
	the corresponding  Darboux transformation of the measure is the following perturbed measure
	\begin{align}\label{eq:darboux-measure-deformation}
	T\d\mu(\boldsymbol\theta)&= L(\z(\boldsymbol\theta))\d\mu(\boldsymbol\theta).
	\end{align}
\end{definition}
\begin{pro}\label{pro:real_measue}
	The perturbed measure  $L\d\mu$ is real if and only if
$\bar L (\z^{-1})=L(\z)$;
i.e.,  when the Laurent polynomial can be written as
\begin{align*}
L(\z)=L_0+\sum_{k=1}^{  l }\sum_{i=1}^{|[k]|/2}(L_{\q_{i}^{(k)}}z^{\q_{i}^{(k)}}+\bar L_{\q_{i}^{(k)}} z^{-\q_{i}^{(k)}}).
\end{align*}
When we have
\begin{align*}
L_0> 2\sum_{k=1}^{  l }\sum_{i=1}^{|[k]|/2}|L_{\q_{i}^{(k)}}|
\end{align*}
definite positiveness of $T\d\mu$ is ensured.
\end{pro}
\begin{proof}
	The perturbed the measure can be written
	\begin{align*}
	L(\z(\boldsymbol{\theta}))\d\mu(\boldsymbol{\theta})=\Big(L_0+2
	\sum_{k=1}^{  l }\sum_{i=1}^{|[k]|/2}|L_{\q_{i}^{(k)}}|\cos\big(\q_{i}^{(k)}\cdot\boldsymbol{\theta}+\arg L_{\q_{i}^{(k)}}\big)\Big)\d\mu(\boldsymbol{\theta}).
	\end{align*}
from where the sufficient condition for positiveness follows.
\end{proof}

Another manner of ensuring positivity is as follows. For $D=1$ the  Fej\'{e}r-Riesz factorization \cite{fejer,riesz} allows us for expressing a  Laurent polynomial which is non negative in the circle $\T$ in the following manner $L(z)=\bar Q(z^{-1})Q(z)$, for a polynomial  $Q$; notice that this expression  when evaluated on the circle $\T$ takes the form
$L(\Exp{\operatorname{i}\theta})=|Q(\Exp{\operatorname{i}\theta}) |^2$.
In the multivariate scenario  the situation is quite different. As proven in \cite{Dritschel} and discussed further in \cite{geronimo} Fej\'{e}r-Riesz factorization can be extended; indeed, we can write any multivariate Laurent polynomial  strictly positive in $\T^D$
as the Dirtschel's finite sum of squared magnitudes of multivariate polynomials
\begin{align}\label{eq:suma}
L(\z)=&\sum_{i=1}^{r} \bar Q_i(\z^{-1})Q_i(\z),& &L(\Exp{\operatorname{i}\boldsymbol{\theta}})=\sum_{i=1}^r|Q_i(\Exp{\operatorname{i}\boldsymbol{\theta}})|^2.
\end{align}

\begin{pro}\label{pro:string-polynomial}
	The moment matrices $TG$  and  $G$  satisfy
	\begin{align*}
	TG=L(\bUpsilon)G=G\;L(\bUpsilon).
	\end{align*}
\end{pro}
\begin{proof}
		We prove the first equality	
	\begin{align}\label{moment}
	TG=&\oint_{\T^D} \chi(\z(\boldsymbol\theta)) (T\d \mu(\boldsymbol\theta)) \chi(\z(-\boldsymbol\theta))^\top & \text{from Definition \ref{def:moment}}
	\\
	=&\oint_{\T^D} \chi(\z(\boldsymbol\theta)) ( L(\z(\boldsymbol\theta))\d \mu(\boldsymbol\theta)) \chi(\z(-\boldsymbol\theta))^\top &  \text{from \eqref{eq:darboux-measure-deformation}}\\
	=& \oint_{\T^D} (L(\bUpsilon)\chi)(\z(\boldsymbol\theta)) \d \mu(\boldsymbol\theta) \chi(\z(-\boldsymbol\theta))^\top &\text{from \eqref{importante}} \\
	=&L(\bUpsilon) G.
	\end{align}
	The second equation follows similarly.
\end{proof}
%
We proceed to introduce a semi-infinite matrix that models the Darboux transformation
\begin{definition}\label{def:resolvents}
	The resolvents are
	\begin{align*}
	\omega \coloneq& (TS)L(\bUpsilon)S^{-1}, & \hat{\omega}^\dagger \coloneq \big(\hat S^{-1}\big)^\dagger L(\bUpsilon)(T\hat S)^\dagger.
	\end{align*}
\end{definition}


We also introduce
\begin{definition}\label{def:adjoint-resolvent}
	The adjoint resolvents are
		\begin{align*}
		\hat M &\coloneq \hat ST\hat S^{-1},& M&\coloneq STS^{-1}.
		\end{align*}
\end{definition}

\begin{pro}\label{pro:adjoint-resolvent}
	The resolvents and the adjoint resolvents satisfy
	\begin{align*}
	\hat M^\dagger =& (TH^{-1})\omega H,&
	M = & H\hat{\omega}^\dagger  (TH)^{-1}.
	\end{align*}
\end{pro}
\begin{proof}
	Proposition \ref{moment} and the corresponding Gauss--Borel factorizations lead to
	\begin{align*}
	(TS)^{-1}(TH) \big((T\hat S)^{-1}\big)^\dagger =L(\bUpsilon)S^{-1}H\big(\hat S^{-1})^\dagger=S^{-1}H\big(\hat S^{-1})^\dagger L(\bUpsilon),
	\end{align*}
	and an appropriate cleaning of these equations to
	\begin{align*}
(TH) \big(\hat S(T\hat S)^{-1}\big)^\dagger  &=(TS)L(\bUpsilon)S^{-1}H, &
S	(TS)^{-1}(TH) &=H\big(\hat S^{-1})^\dagger L(\bUpsilon)(T\hat S)^\dagger,
	\end{align*} 
	from where the result follows.
\end{proof}
\begin{pro}
	The resolvents $\omega$ and $\hat\omega$ are block upper triangular matrices having all their block superdiagonals above the $m$-th block superdiagonal equal to zero. The adjoint resolvents $M$ and $\hat M$ are block lower unitriangular matrices having all their  block subdiagonals below the $m$-th block subdiagonal equal to zero. 
\end{pro}

	\begin{proof}
		The reader should  notice that, if the longitude of the perturbing Laurent polynomial is $\ell(L)=m$, the semi-infinite matrix $L(\bUpsilon)$ is a  $(2  m +1)$-banded matrix with only the  first $  m $ block superdiagonals, the main diagonal and the first $  m $ subdiagonals different from zero;  hence,
	Definition \ref{def:resolvents}   implies that  $\omega$ has all its superdiagonals above the $  m $-th superdiagonal equal to zero and that $\hat \omega^\dagger$ has all its subdiagonals below the $  m $-th subdiagonal equal to zero.
Observe also that Definition  \ref{def:adjoint-resolvent} implies that the adjoint resolvents  $M$
and $\hat M$ are block lower unitriangular matrices. 
Now, Proposition \ref{pro:adjoint-resolvent} linking resolvents and adjoint resolvents gives the satated result.
	\end{proof}
Regarding the deformed MVOLPUT and the original ones we have
\begin{pro}
	The following relations among Darboux perturbed  and original MVOLPUT hold true
	\begin{align}\label{eq:darboux.mstep}
	\omega \Phi(\z) &= L(\z)T\Phi(\z), &
	\hat\omega \hat\Phi(\z) &= \bar L(\z^{-1})T\hat\Phi(\z),\\
	\label{eq:darboux.mstep2}
	MT\Phi(\z)&=\Phi(\z), &
	\hat MT\hat \Phi(\z)&=\hat \Phi(\z),
	\end{align}
\end{pro}
\begin{proof}
	To prove equations \eqref{eq:darboux.mstep} just observe that
	\begin{align*}
	\omega\Phi(\z)&= (TS)L(\bUpsilon)S^{-1} S\chi=(TS)L(\bUpsilon)\chi(\z)=L(\z)T\Phi(\z), \\
	\hat	\omega\hat\Phi(\z)&= (T\hat S)(L(\bUpsilon))^\dagger\hat S^{-1} S\chi=(T\hat S)(L(\bUpsilon))^\dagger\chi(\z)=\bar L(\z^{-1})T\hat \Phi(\z).
	\end{align*}
	The relations \eqref{eq:darboux.mstep2} are a consequence of  Proposition \ref{pro:adjoint-resolvent}.
\end{proof}
\subsection{Interpolation, sample matrices  and poised sets}
\begin{definition}
	The algebraic hypersurface of zeroes of the Laurent polynomial $L$ in the algebraic torus is denoted by
	\begin{align*}
	Z(L)\coloneq\big\{\boldsymbol p\in(\C^*)^D: L(\boldsymbol p)=0\big\}.
	\end{align*}
\end{definition}


 \begin{definition}\label{def:sample}
 	A  set of nodes in the algebraic torus
 	\begin{align*}
 	\mathcal N_{ k, m}\coloneq\{\boldsymbol p_j\}_{j=1}^{r_{ k, m}}\subset(\C^*)^D
 	\end{align*}
 	is a set with $r_{k,m}\coloneq N_{ k+ m-1}-N_{ k-1}=|[ k]|+\dots+|[ k+
 	 m-1]|$nodes. Given these nodes we consider
 	the corresponding  sample matrices
 	\begin{align}
 	\Sigma_k^m\coloneq &
 	\PARENS{
 		\begin{matrix}
 		\phi_{[k]}(\boldsymbol p_1) & \dots & \phi_{[k]}(\boldsymbol p_{r_{ k, m}}) \\\vdots& &\vdots\\
 		\phi_{[k+m-1]}(\boldsymbol p_1) & \dots & \phi_{[ k+ m-1]}(\boldsymbol p_{r_{ k, m}})
 		\end{matrix}
 	}\in\C^{r_{k,m}\times r_{k,m}},\label{eq:sample}\\
 	\Sigma_{[k,m]}\coloneq&\big(\phi_{[k+m]}(\boldsymbol p_1), \dots , \phi_{[ k+ m]}(\boldsymbol p_{r_{ k, m}} )\big)\in\C^{|[k+ m]|\times r_{ k, m}}.\notag
 	\end{align}
 \end{definition}

 \begin{lemma}\label{lemma:resolvent}
 	When the set of nodes $\mathcal N_{k,m}\subset Z(L)$ belongs to the algebraic hypersurface of the Laurent polynomial $L$ of longitude $m=\ell(L)$, the resolvent coefficients satisfy
 	\begin{align*}
 	\omega_{[k],[k+m]}\Sigma_{[k,m]}+
 	(\omega_{[k],[k]},\dots,\omega_{[k],[k+m-1]})\Sigma_k^m=0.
 	\end{align*}
 \end{lemma}
 Now we introduce an  important type of node sets within
  the algebraic hypersurface of the Laurent polynomial $L$ of longitude $m=\ell(L)$
\begin{definition}
	We say that the set $\mathcal N_{k,m}\subset Z(L)$  is a poised set if the sample matrix given in \eqref{eq:sample} is non singular
	\begin{align*}
	\det \Sigma_k^m\neq 0.
	\end{align*}
\end{definition}
\begin{theorem}[A Christoffel formula]\label{theorem:the deal}
	For a poised set of nodes $\mathcal N_{k,m}\subset Z( L)$  the Darboux transformation of the MVOLPUT can be expressed in terms of the original ones as the following last quasi-determinantal expression
	\begin{align*}
	T\phi_{[k]}(\z)=  \frac{(L(\boldsymbol\Upsilon)_{[k],[k+m]}}{L(\z)}
	\Theta_*\PARENS{\begin{array}{c|c}
		\Sigma^m_k &\begin{matrix}
		\phi_{[k]}(\z) \\ \vdots \\ \phi_{[k+m-1]}(\z)
		\end{matrix}\\\hline
		\Sigma_{[k,m]} & \phi_{[k+m]}(\z)
		\end{array}}.
	\end{align*}
\end{theorem}
\begin{proof}
	Observe that Lemma  \ref{lemma:resolvent} together with $\omega_{[k],[k+m]}=\big(L(\boldsymbol \Upsilon)\big)_{[k],[k+m]}$ implies
	\begin{align*}
	(\omega_{[k],[k]},\dots,\omega_{[k],[k+m-1]})=-\big(L(\boldsymbol \Upsilon)\big)_{[k],[k+m]}\Sigma_{[k,m]}\big(\Sigma_k^m\big)^{-1}.
	\end{align*}
	and from $L(\z)T\phi(\z)=\omega \phi(\z)$ the result follows.
\end{proof}
\begin{pro}\label{pro:on the resolvent coefficients}
	The resolvent coefficients can be expressed as last quasi-determinants
	\begin{align*}
	\omega_{[k],[k+  l ]}=  (L(\bUpsilon))_{[k],[k+  m ]}
	\Theta_*\PARENS{\begin{array}{c|c}
		\begin{matrix}	\Sigma^m_k
		\end{matrix} &
		\begin{matrix}
		0_{[k],[k+  l ]}\\
		\vdots\\
		0_{[k+  l -1],[k+  l ]}\\
		\I_{[k+  l ]}\\
		0_{[k+  l +1],[k+  l ]}\\\vdots
		\end{matrix},
		\\ \hline
		\Sigma_{[k,m]} &
		0_{[k+  m ],[k+  l ]}			
		\end{array}}.
	\end{align*}
	for $  l =0,\dots,  m -1$.
\end{pro}


\begin{pro}\label{ostias}
	We have  the following Darboux transformation formul{\ae} for the quasi-tau matrices
	\begin{align*}
	TH_{[k]}=&(L(\bUpsilon))_{[k],[k+  m ]}
	\Theta_*\PARENS{\begin{array}{c|c}
		\begin{matrix}	\Sigma^m_k
		\end{matrix} &
		\begin{matrix}
		H_{[k]}\\
		0_{[k+1],[k]}\\
		\vdots
		\end{matrix}
		\\ \hline
		\Sigma_{[k,m]} &
		0_{[k+  m ],[k]}		
		\end{array}}.
	\end{align*}
	and for  the transformed first subdiagonal coefficients
	\begin{align}\label{eq:beta-darboux}
	(T\beta)_{[k]}(L(\bUpsilon))_{[k-1],[k+  m -1]}=
	(L(\bUpsilon))_{[k],[k+  m -1]}+L(\bUpsilon))_{[k],[k+  m ]}
	\Theta_*\PARENS{\begin{array}{c|c}
		\begin{matrix}	\Sigma^m_k
		\end{matrix} &
		\begin{matrix}
		0_{[k],[k+  m -1]}\\
		\vdots\\
		0_{[k+  m -2],[k+  m +1]}\\
		\I_{[k+  m +1]}
		\end{matrix}
		\\ \hline
		\Sigma_{[k,m]} &
		\beta_{[k+  m ]}
		\end{array}}.
	\end{align}
\end{pro}
Observe that  \eqref{eq:beta-darboux} can be written
\begin{multline*}
(T\beta)_{[k]}=
(L(\bUpsilon))_{[k],[k+  m -1]}(L(\bUpsilon))^+_{[k+  m -1],[k-1]}\\+L(\bUpsilon))_{[k],[k+  m ]}
\Theta_*\PARENS{\begin{array}{c|c}
	\begin{matrix}	\Sigma^m_k
	\end{matrix} &
	\begin{matrix}
	0_{[k],[k+  m -1]}\\
	\vdots\\
	0_{[k+  m -2],[k+  m +1]}\\
	\I_{[k+  m +1]}
	\end{matrix}
	\\ \hline
	\Sigma_{[k,m]} &
	\beta_{[k+  m ]}
	\end{array}}(L(\bUpsilon))^+_{[k+  m -1],[k-1]}.
\end{multline*}
whenever the right inverse $(L(\bUpsilon))^+_{[k+  m -1],[k-1]}$ of $(L(\bUpsilon))_{[k-1],[k+  m -1]}$ exists.
\begin{pro}\label{pro:jacobi-LU} 
	The evaluation of the Laurent polynomial $L$ on the Jacobi matrices $\boldsymbol J=(J_1,\dots,J_D)^\top$ or on its  perturbations $T\boldsymbol J=(TJ_1,\dots,TJ_D)^\top$
	are linked to the typical Darboux  $LU$ and $UL$ factorizations of the Jacobi and Darboux transformed Jacobi matrices; i.e.,
	\begin{align*}
	L(\boldsymbol J)&= M\omega, & L(T\boldsymbol J)&=\omega M,
	\end{align*}
	respectively.
\end{pro}

\begin{proof}
	From Definition \ref{def:resolvents} we get $\omega=(TS)S^{-1} SL(\bUpsilon)S^{-1}=(TS)L(\bUpsilon)(TS)^{-1} (TS) S^{-1}$ and therefore, using Proposition \ref{pro:adjoint-resolvent} and Definition \ref{def:jacobi} we get $M\omega=L(\boldsymbol J)$, from the first equality, and $\omega M=L(T\boldsymbol J)$ from the second equality.
\end{proof}

From the first equation in the previous Proposition we get 
\begin{pro}
	The block truncations $(\mathcal Q (\boldsymbol J))^{[k]}$ admit a $LU$ factorization
	\begin{align*}
	(L (\boldsymbol J))^{[k]}=M^{[k]}\omega^{[k]}
	\end{align*}
	in terms of the corresponding truncations of the adjoint resolvent $M^{[k]}$ and resolvent $\omega^{[k]}$.
\end{pro}
\begin{pro}\label{pro:regularity-truncation-jacobi}
The truncated matrix $(L (\boldsymbol J))^{[k]}$ is  regular and
	\begin{align*}
	\det\big( (L(\boldsymbol J))^{[k]}\big)=\prod_{l=0}^{k-1}\frac{\det TH_{[l]}}{\det H_{[l]}}.
	\end{align*}
\end{pro}
\begin{proof}
	To prove this result just use Propositions \ref{pro:jacobi-LU}, the explicit expression
	\begin{align*}
	\omega_{[k],[k]}=(TH_{[k]})H_{[k]}^{-1},
	\end{align*}
	and the assumption that
	the minors of the moment matrix and the perturbed moment matrix are not zero.
\end{proof}

\subsection{Nice Laurent polynomials. Flavors from tropical geometry}
So far, for the analysis of Darboux transformations of MVOLPUT, we have followed the approach given in \cite{MVOPR-darboux} for MVOPR. However, for the polynomial ring $\C[z_1,\dots,z_D]$ we had an ordering, the graded lexicographical order, wich is also a grading: $\deg (PQ)=\deg P+\deg Q$ which is missing in this context. Now, for the Laurent polynomial ring $\C[z_1^{\pm 1},\dots,z_D^{\pm 1}]$, what we have is the longitude given in Definition \ref{def:support} which as we already stressed does not give a grading to the ring. You can readily check that for two units in different orthants $|\q+\q'|<|\q|+|\q'|$.
For example, for $D=2$  the Laurent polynomial $L(z_1,z_2)=z_1^{-2}+z_2^{-2}+1$ is of that type, observing $z_1z_2L(z_1,z_2)=z_1+z_1+z_1z_2$ we deduce
$\ell(z_1z_2L)=2<\ell (z_1z_2)+\ell (L)=3$. However, there are other polynomials that do satisfy this important property, for example $L(z_1,z_2)=z_1^{-2}+z_1^2+z_2$. Indeed, $z_1^{\alpha_1}z_2^{\alpha_2}L=z_1^{\alpha_1-2}z_2^{\alpha_2}+z_1^{\alpha_1+2}z_2^{\alpha_2}+z_1^{\alpha_1}z_2^{\alpha_2+1}$ and $\ell (z_1^{\alpha_1}z_2^{\alpha_2}L)=\max (|\alpha_1-2|+|\alpha_2|,|\alpha_1+2|+|\alpha_2|,|\alpha_1|+|\alpha_2+1|)=|\alpha_1|+|\alpha_2|+2$.

\begin{definition}
	A Laurent polynomial $L$ is  nice if
	\begin{align*}
	\ell (L M)=\ell (L)+\ell (M),
	\end{align*}
	for every Laurent polynomial $M\in\C[z_1^{\pm 1},\dots,z_D^{\pm 1}]$.
\end{definition}


\begin{pro}
If $(V_{a,+}^{(m)},V_{a,-}^{(m)})\subset \operatorname{NP}(L)$ with $\ell (L)=m$ and  $a\in\Z_D$,  then $L$ is nice.
\end{pro}
\begin{proof}
 Remember that
 \begin{align*}
\ell(z^\q L)  =\max_{\q'\in\operatorname{NP}(L)\cap\Z^D}(|\alpha'_1+\alpha_1|+\dots+|\alpha_D'+\alpha_D|)\leq |\q|+\ell(L),
 \end{align*}
 and if a pair of opposed vertices are included in the Newton polytope, say $V^{(m)}_{1,\pm}$, as $\max ( |m+\alpha_1|,|-m+\alpha_1|)=|m|+|\alpha_1|$, the bound is saturated
  \begin{align*}
  \ell(z^\q L)  = |\q|+\ell (L).
\end{align*}
\end{proof}
However, these Laurent polynomials are not the only nice Laurent polynomials, for example $L=z_1^{-2}+z_1z_2^{-1}+z_1z_2$ is also nice. Indeed, $z_1^{\alpha_1}z_2^{\alpha_2}L=z_1^{\alpha_1-2}z_2^{\alpha_2}+z_1^{\alpha_1+1}z_2^{\alpha_2-1}+z_1^{\alpha_1+1}z_2^{\alpha_2+1}$ and $\ell (z_1^{\alpha_1}z_2^{\alpha_2}L)=\max (|\alpha_1-2|+|\alpha_2|,|\alpha_1+1|+|\alpha_2-1|,|\alpha_1+1|+|\alpha_2+1|)=|\alpha_1|+|\alpha_2|+2$.
Another example is $L=\operatorname{i}z_1^{-1}z_2+2z_1z_2-\operatorname{i}z_1z_2^{-1}+2z_1^{-1}z_2^{-2}+z_1+z_1^{-1}+5$, in this case we
 have
$
 \ell (z_1^{\alpha_1}z_2^{\alpha_2}L)=\max (|\alpha_1-1|+|\alpha_2+1|,|\alpha_1+1|+|\alpha_2+1|,|\alpha_1+1|+|\alpha_2-1|,|\alpha_1|+|\alpha_2-1|)=|\alpha_1|+|\alpha_2|+2.
$

The following is taken from tropical geometry \cite{tropical}
\begin{definition}
	\begin{enumerate}
		\item The Minkowski sum of the two sets  $A,B\subset \R^D$ is given by
		\begin{align*}
		A+B=\{a+b: a\in A, b\in B\}.
		\end{align*}
		\item	For each $\boldsymbol w\in\R^D$ we can introduce an alternative order in $\Z^D$ and say $\q>_{\boldsymbol w}\q'$ if and only if  $(\q-\q')\cdot\boldsymbol w >0$.
		\item The initial form along the  weight $\boldsymbol w\in\R^D$ of a Laurent polynomial is
		\begin{align*}
		\operatorname{in}_{\boldsymbol w}L\coloneq \sum_{\mathcal A_{\boldsymbol w}}L_\q\z^\q,
		\end{align*}
		here the support is $\mathcal A_{\boldsymbol w}\coloneq \{\q\in\mathcal A: \q'\geq_{\boldsymbol w} \q\;\forall \q'\in\mathcal A\} $.
		\item The tropicalization of a Laurent polynomial also known as its tropical hypersurface is
		\begin{align*}
		\mathcal T (L)\coloneq\{\boldsymbol w\in\R^D: \operatorname{in}_{\boldsymbol w}(L)\text { is not a monomial}\}.
		\end{align*}
	\end{enumerate}
\end{definition}
\begin{pro}
	\begin{enumerate}
			\item For two Laurent polynomials $K,L\in\C[\z^{\pm 1}]$ we have the following formul{\ae} for their Newton polytopes and  faces in terms of the Minkowski sum (the tropical product of polytopes)
		\begin{align}\label{eq:sumNP}
		\operatorname{NP}(KL)=&\operatorname{NP}(K)+\operatorname{NP}(L), &
		F_{\boldsymbol w}(KL)=&F_{\boldsymbol w}(K)+F_{\boldsymbol w}(L).
		\end{align}
		\item The Newton polytope $\operatorname{NP}(\z^\q L)$ and its faces are obtained from the Newton polytope $\operatorname{NP}( L)$ via an automorphism in the Abelian group $\Z^D$ --a translation by $\q$--
		\begin{align*}
		\operatorname{NP}(\z^\q L)&=\{\q\}+\operatorname{NP}(L),
		&	
		F_{\boldsymbol w}(z^\q L)&=\{\q\} +F_{\boldsymbol w}(L).
		\end{align*}
		\item Newton polytopes of initial forms are the faces of the Newton polytopes
		\begin{align*}
		\operatorname{NP}(\operatorname{in}_{\boldsymbol w}(L))=F_{\boldsymbol w}(\operatorname{NP}(L)).
		\end{align*}
		\item The tropical hypersurface of a Laurent polynomial is
		\begin{align*}
		\mathcal T (L)=\big\{\boldsymbol w\in\R^D: \dim F_{\boldsymbol w}(L)\geq 1\big\}.
		\end{align*}
	\end{enumerate}
\end{pro}

 	The orthants of integers introduced in Definition \ref{orthant} are extended to the real space as follows
\begin{definition}
	The real orthants in $\R^D$ are
\begin{align*}
(\R^D)_\sigma\coloneq &\bigtimes_{i=1}^D R_i, &
R_i\coloneq&\ccases{
	\R_-, & i\in(\sigma\setminus\partial\sigma),\\
	\R_<, & i\in\partial\sigma,\\
	\R_+, & i\in(\complement \sigma\setminus\partial(\complement\sigma)),\\
	\R_>, &  i\in \partial\complement\sigma,
}
\end{align*}
where $\R_\gtrless\coloneq\R_\pm\setminus\{0\}$. For $\sigma=\Z$ we define
\begin{align*}
(\R^D)_{\Z_D}\coloneq\big(\R_-\big)^D\setminus\{\boldsymbol 0\}.
\end{align*}
\end{definition}
In this case  we do not need to distinguish between $\R_{\gtrless}$ and its topological closure $\overline{\R_{\gtrless}}=\R_\pm$, but for consistency with previous developments  we use  a similar notation.

\begin{pro}\label{pro:nice-longitude}
	 Among the set  of faces $\{F_{\boldsymbol w}(m): \dim F_{\boldsymbol w}(m)\geq 0\}_{\boldsymbol w\in\R^D}$ of the  regular hyper-octahedron $\operatorname{Con}([m])$ there are $2^D$ facets $\{F_\sigma(m)\}_{\sigma\in 2^{\Z_D}}$, with a facet per orthant $F_\sigma(m)\subset (\R^D)_\sigma$. Non facet faces of codimension $p\in\{2,\dots,D\}$ belong to the intersection $p$ orthants.
\end{pro}
\begin{definition}
	We introduce $\operatorname{NP}_\sigma(L)\coloneq(\operatorname{NP}(L))\cap \overline{(\R^D)_\sigma}$.
\end{definition}

 \begin{pro}\label{pro:whoisnice}
 	A Laurent polynomial $L$ is nice if and only if    $\operatorname{NP}_\sigma(L)$  contains elements of longitude   $\ell(L)$ for each $\sigma\in 2^{\Z_D}$.
 \end{pro}
 \begin{proof}
See Appendix \ref{proof6}.
\end{proof}

 \begin{pro}
 	If a Laurent polynomial $L$ is nice then its tropical hypersurface $\mathcal T (L)$ has a non trivial intersection with each orthant
 	\begin{align*}
 	\mathcal T (L)\cap \overline{(\R^D)_\sigma}\neq& \emptyset,  & \forall\sigma\in 2^{\Z_D}.
 	\end{align*}
 \end{pro}

\begin{pro}
 		If a Laurent polynomial $L$is nice   the polytopes corresponding to the translations by the generators of $\Z^D$ of its  Newton polytope, $\{\pm\ee_a\}+\operatorname{NP}(L)$ have elements of longitude $\ell(L)+1$ for each $a\in\Z_D$.
 	\end{pro}
 	 \begin{proof}
Let us suppose that a Laurent polynomial $L$ is nice. Then, for each subset $\sigma\subset\{1,\dots,D\}$ we have a vector $\q_\sigma\in(\operatorname{NP}(L))\cap\overline{(\R^D)_\sigma}$ with $|\q_\sigma|=\ell(L)$. Therefore,
 as $\ee_a\in\overline{(\R^D)_\sigma}$ for $a\in\complement\sigma$ and $-\ee_a\in \overline{(\R^D)_\sigma}$
 for $a\in\sigma$ we have $|\q_\sigma+\ee_a|=|\q_a|+1$ for all $a\in\complement\sigma$ and
 $|\q_\sigma-\ee_a|=|\q_a|+1$ whenever $a\in\sigma$. Therefore, as this happens for each orthant, we conclude that $\big(\{\pm\ee_a\}+\operatorname{NP}(L)\big)\cap \{F_{\boldsymbol w}(\operatorname{Conv}[\ell(L)+1])\}_{\boldsymbol w\in \R^D}\neq \emptyset$.
 \end{proof}

 The previous two results are not  characterizations of nicety but  properties of Laurent polynomials that indeed could have non-nice Laurent polynomials. We now give a similar result to the last one that in fact is a characterization. For that aim we need to introduce
 some geometrical elements of the regular hyper-octahedron

\begin{definition}
	  In each orthant $\overline{(\R^D)_\sigma}$ we consider the vectors
	  \begin{align*}
	  \boldsymbol u_\sigma=\sum_{a\in\complement\sigma}\ee_a-\sum_{a\in\sigma}\ee_a.
	  \end{align*}
\end{definition}
 \begin{pro}
 	The facets of the regular hyper-octahedron $\operatorname{Conv}([m])$ are given by $F_{\boldsymbol u_\sigma}(\operatorname{Conv}([m]))$.
 \end{pro}

 \begin{pro}
 A Laurent polynomial $L$is nice if and only if    $\{\boldsymbol u_\sigma\}+\operatorname{NP}(L)$ has  elements of longitude $\ell(L)+D $ for each $\sigma\in2^{\Z_D}$.
 \end{pro}
 \begin{proof}
 $\Rightarrow$	For a nice Laurent polynomial in each orthant we have a vector such that $\q_\sigma\in(\operatorname{NP}(L))\cap\overline{(\R^D)_\sigma}$ with $|\q_\sigma|=\ell(L)$. Therefore,
 	as $\q_\sigma+\boldsymbol u_\sigma \in\overline{(\R^D)_\sigma}$ so that  $|\q_\sigma+\boldsymbol u_\sigma |=\ell(L)+D$. As this happens for each orthant, we deduce that $\big(\{\boldsymbol u_\sigma\}+\operatorname{NP}(L)\big)\cap \{F_{\boldsymbol w}(\operatorname{Conv}[\ell(L)+D])\}_{\boldsymbol w\in \R^D}\neq \emptyset$.
 	
$\Leftarrow$
We proceed by contradiction and suppose that there is no such point in a given orthant $\overline{(\R^D)_\sigma}$. Consequently,  the set $\{\boldsymbol u_\sigma\}+\operatorname{NP}_\sigma(L)$ contains only vectors of longitude less than $\ell(L)+D$. But this also holds for the rest of orthants component of the Newton polytope $\operatorname{NP}_{\sigma'}(L)$, $\sigma\neq\sigma'$, as in that case, despite having vectors in the Newton polytope of the maximum longitude $\ell(L)$, when we consider the set $\{\boldsymbol u_\sigma\}+\operatorname{NP}_{\sigma'}(L)$ being the vector $\boldsymbol u_\sigma$ and the vectors in $\operatorname{NP}_{\sigma'}(L)$  in different orthants, some of the components are summations but some are subtractions, and the bound $\ell(L)+D$ is never reached.
 \end{proof}
 \begin{pro}\label{pro:prod-nice}
 	Let $L=L_1L_2$ be the product of two Laurent polynomials. Then, the Laurent polynomial $L$ is nice if and only if  $L_1,L_2$ are nice.	
 	\end{pro}
 \begin{proof}
 	See Appendix \ref{proof7}.
 \end{proof}
 	
\begin{cor}
	The set of nice polynomials is a monoid under the polynomial multiplication.
\end{cor}

 In the next picture we illustrate this situation, we have drawn   $\operatorname{Conv}([k])$, for $k=0,\dots,4$, in dashed lines, and also the nice Newton polytopes  $\operatorname{NP}(z_1^{-2}+z_1z_2^{-1}+z_1z_2)$,
 $\operatorname{NP}(\operatorname{i}z_1^{-1}z_2+2z_1z_2-\operatorname{i}z_1z_2^{-1}+2z_1^{-1}z_2^{-2}+z_1+z_1^{-1}+9)$ and $\operatorname{NP}(z_1^{-2}+z_1^2+z_2)$ in different colors, red, brown and yellow, respectively.


\begin{center}
	\begin{tikzpicture}[
	scale=1,
	axis/.style={very thick, ->, >=stealth'}
	]
	\draw[axis] (-4.5,0)  -- (4.5,0) node(xline)[right] {$\alpha_1$};
	\draw[axis] (0,-4.5) -- (0,4.5) node(yline)[above] {$\alpha_2$};
\foreach \n in {1,...,4}	\draw[thick,dashed] (-\n,0) -- (0,\n) node[above ]{\tiny$ \operatorname{Conv}([\n])$} -- (\n,0) -- (0,-\n) -- (-\n,0) ;
\foreach \n in {1,...,4}	\draw[thick,fill,gray,opacity=1/(2+\n)] (-\n,0) -- (0,\n) -- (\n,0) -- (0,-\n) -- (-\n,0) ;
\foreach \x in {-2,...,2}{
		\foreach \y in {-2,...,2}{
 \node[draw,circle,inner sep=2pt,fill,blue] at (\x+\y,\x-\y) {};}
}
\foreach \x in {-2,...,1}{
	\foreach \y in {-1,...,1}{
		\node[draw,circle,inner sep=2pt,fill,blue] at (\x+\y+1,\x-\y) {};}
}
\foreach \x in {-1,...,2}{
	\foreach \y in {-1,...,1}{
		\node[draw,circle,inner sep=2pt,fill,blue] at (\x+\y-1,\x-\y) {};}
}
\draw[red,fill,thick,opacity=0.5]  (-2,0) -- (1,1) -- (1,-1) --(-2,0);
\draw[->] (0.9,-0.2) --  (3.8, -1.7)
node[below]{\tiny$\operatorname{NP}(z_1^{-2}+z_1z_2^{-1}+z_1z_2)$};
\draw[brown,fill,thick,opacity=0.5]  (1,1) -- (-1,1) -- (-1,-1) -- (1,-1) --(1,1);
\draw[->] (0.7,0.95) --  (4.2, 2.95) node[above]{\tiny$\operatorname{NP}(\operatorname{i}z_1^{-1}z_2+2z_1z_2-\operatorname{i}z_1z_2^{-1}+2z_1^{-1}z_2^{-2}+z_1+z_1^{-1}+9)$};
\draw[yellow,fill,thick,opacity=0.3] (-2,0)--(0,1)--(2,0)--(-2,0);
\draw[->] (-1.5,0.2) --  (-3.4, 1.9) node[above]{\tiny$\operatorname{NP}(z_1^{-2}+z_1^2+z_2)$};
	\end{tikzpicture}
	\end{center}
We now show the nice Newton polytopes and tropical curves in blue for basic degree 2 nice Laurent polynomials, any other one can be obtained from these by the corresponding convex hull of the unions of these basic Newton polytopes  --observe that the two nice Newton polytopes of vertex type are not drawn here--
\begin{center}
\begin{tikzpicture}[
scale=0.7,
axis/.style={ thick, ->, >=stealth'},
]
\draw[axis] (-3.5,0)  -- (3.5,0) node(xline)[right] {$\alpha_1$};
\draw[axis] (0,-3.5) -- (0,3.5) node(yline)[above] {$\alpha_2$};
\foreach \n in {1,...,3}	\draw[thick,dashed] (-\n,0) -- (0,\n)  -- (\n,0) -- (0,-\n) -- (-\n,0) ;
\draw[red!80,dashed,opacity=0.8,fill]  (1,1) -- (-1,1) -- (-1,-1) -- (1,-1) --(1,1);
\draw[blue,very thick]  (-2.5,0)--(0,0);
\draw[blue,very thick]   (0,0)--(2.5,0);
\draw[blue,very thick]   (0,-2.5)--(0,0);
\draw[blue,very thick]   (0,0)--(0,2.5);
\end{tikzpicture}
\begin{tikzpicture}[
scale=0.7,
axis/.style={ thick, ->, >=stealth'},
]
\draw[axis] (-3.5,0)  -- (3.5,0) node(xline)[right] {$\alpha_1$};
\draw[axis] (0,-3.5) -- (0,3.5) node(yline)[above] {$\alpha_2$};
\foreach \n in {1,...,3}	\draw[thick,dashed] (-\n,0) -- (0,\n)  -- (\n,0) -- (0,-\n) -- (-\n,0) ;
\draw[red!80,thick,fill,opacity=0.8]  (0,2) -- (1,-1) -- (-1,-1) --(0,2);
\draw[blue,very thick]  (0,0) -- (0,-2.5);
\draw[blue,very thick] (0,0) -- (3,1) ;
\draw[blue,very thick]  (0,0) -- (-3,1) ;
\end{tikzpicture}
\begin{tikzpicture}[
scale=0.7,
axis/.style={ thick, ->, >=stealth'},
]
\draw[axis] (-3.5,0)  -- (3.5,0) node(xline)[right] {$\alpha_1$};
\draw[axis] (0,-3.5) -- (0,3.5) node(yline)[above] {$\alpha_2$};
\foreach \n in {1,...,3}	\draw[thick,dashed] (-\n,0) -- (0,\n)  -- (\n,0) -- (0,-\n) -- (-\n,0) ;
\draw[red!80,thick,fill,opacity=0.8]  (2,0) -- (-1,1) -- (-1,-1) --(2,0);
	\draw[blue,very thick]  (0,0) -- (-2.5,0);
	\draw[blue,very thick] (0,0) -- (1,3) ;
	\draw[blue,very thick]  (0,0) -- (1,-3) ;
\end{tikzpicture}
\end{center}
\begin{center}
\begin{tikzpicture}[
scale=0.7,
axis/.style={ thick, ->, >=stealth'},
]
\draw[axis] (-3.5,0)  -- (3.5,0) node(xline)[right] {$\alpha_1$};
\draw[axis] (0,-3.5) -- (0,3.5) node(yline)[above] {$\alpha_2$};
\foreach \n in {1,...,3}	\draw[thick,dashed] (-\n,0) -- (0,\n)  -- (\n,0) -- (0,-\n) -- (-\n,0) ;
\draw[red!80,thick,fill,opacity=0.8]  (0,-2) -- (-1,1) -- (1,1) --(0,-2);
	\draw[blue,very thick]  (0,0) -- (0,2.5);
	\draw[blue,very thick] (0,0) -- (3,-1) ;
	\draw[blue,very thick]  (0,0) -- (-3,-1) ;
\end{tikzpicture}
\begin{tikzpicture}[
scale=0.7,
axis/.style={thick, ->, >=stealth'},
]
\draw[axis] (-3.5,0)  -- (3.5,0) node(xline)[right] {$\alpha_1$};
\draw[axis] (0,-3.5) -- (0,3.5) node(yline)[above] {$\alpha_2$};
\foreach \n in {1,...,3}	\draw[thick,dashed] (-\n,0) -- (0,\n)  -- (\n,0) -- (0,-\n) -- (-\n,0) ;
\draw[red!80,thick,fill,opacity=0.8]  (-2,0) -- (1,1) -- (1,-1) --(-2,0);
\draw[blue,very thick]  (0,0) -- (2.5,0);
\draw[blue,very thick] (0,0) -- (-1,3) ;
\draw[blue!,very thick]  (0,0) -- (-1,-3) ;
\end{tikzpicture}

\end{center}

We show now some non basic nice Newton polytopes, and their tropical curves, built up by the convex hull of the union of the Newton polytopes of basic nice Laurent polynomials
\begin{center}
	\begin{tikzpicture}[
	scale=0.7,
	axis/.style={ thick, ->, >=stealth'},
	]
	\draw[axis] (-3.5,0)  -- (3.5,0) node(xline)[right] {$\alpha_1$};
	\draw[axis] (0,-3.5) -- (0,3.5) node(yline)[above] {$\alpha_2$};
	\foreach \n in {1,...,3}	\draw[thick,dashed] (-\n,0) -- (0,\n)  -- (\n,0) -- (0,-\n) -- (-\n,0) ;
	\draw[red!80,thick,fill,opacity=0.8]  (-2,0) -- (-1,1) -- (1,1) -- (2,0) -- (1,-1) -- (-1,-1) --(-2,0);
\draw[blue,very thick]  (0,0) -- (0,2.5);
\draw[blue,very thick]  (0,0) -- (0,-2.5);
\draw[blue,very thick] (0,0) -- (-3,3) ;
\draw[blue!,very thick]  (0,0) -- (-3,-3) ;
\draw[blue,very thick] (0,0) -- (3,3) ;
\draw[blue!,very thick]  (0,0) -- (3,-3) ;
	\end{tikzpicture}
	\begin{tikzpicture}[
	scale=0.7,
	axis/.style={very thick, ->, >=stealth'},
	]
	\draw[axis] (-3.5,0)  -- (3.5,0) node(xline)[right] {$\alpha_1$};
	\draw[axis] (0,-3.5) -- (0,3.5) node(yline)[above] {$\alpha_2$};
	\foreach \n in {1,...,3}	\draw[thick,dashed] (-\n,0) -- (0,\n)  -- (\n,0) -- (0,-\n) -- (-\n,0) ;
	\draw[red!80,thick,fill,opacity=0.8]  (0,2) -- (1,1) -- (1,-1) --(0,-2)--(-1,-1)--(-1,1)--(0,2);
	\draw[blue,very thick]  (0,0) -- (2.5,0);
	\draw[blue,very thick]  (0,0) -- (-2.5,0);
	\draw[blue,very thick] (0,0) -- (-3,3) ;
	\draw[blue!,very thick]  (0,0) -- (-3,-3) ;
	\draw[blue,very thick] (0,0) -- (3,3) ;
	\draw[blue!,very thick]  (0,0) -- (3,-3) ;
	\end{tikzpicture}
\end{center}
In the next picture we illustrate a case for $D=3$, we draw two regular octahedra, the interior one corresponds to the convex hull of $[1]$ and the exterior one is the octahedron corresponding to $[2]$; blue points are the multi-indices. Then, in red, we draw the Newton polytope of the nice Laurent polynomial, which is not a vertex pair polynomial, $L=z_1z_3+2z_1^{-1}z_3-z_2z_3-z_2^{-1}z_3+3z_3^{-2}$
\begin{center}

  \begin{tikzpicture}[scale=2.4,tdplot_main_coords]
  \coordinate (O) at (0,0,0);

     \draw[thick,->] (0,0,0) -- (3,0,0) node[anchor=north east]{$\alpha_1$};
      \draw[thick,->] (0,0,0) -- (0,3.5,0) node[anchor=north west]{$\alpha_2$};
      \draw[thick,->] (0,0,0) -- (0,0,3) node[anchor=south]{$\alpha_3$};

  \tdplotsetcoord{A}{1}{90}{0}    
  \tdplotsetcoord{B}{1}{90}{90}   
  \tdplotsetcoord{C}{1}{90}{180}  
  \tdplotsetcoord{D}{1}{90}{270}  
  \tdplotsetcoord{E}{1}{0}{0}     
  \tdplotsetcoord{F}{1}{180}{0}   

\tdplotsetcoord{A2}{2}{90}{0}    
 \tdplotsetcoord{B2}{2}{90}{90}   
 \tdplotsetcoord{C2}{2}{90}{180}  
 \tdplotsetcoord{D2}{2}{90}{270}  
 \tdplotsetcoord{E2}{2}{0}{0}     
\tdplotsetcoord{F2}{2}{180}{0}   

 \tdplotsetcoord{G2}{sqrt(2)}{45}{270}    
 \tdplotsetcoord{H2}{sqrt(2)}{45}{0}   
 \tdplotsetcoord{I2}{sqrt(2)}{45}{180}  
 \tdplotsetcoord{J2}{sqrt(2)}{45}{90}    

 \tdplotsetcoord{K2}{sqrt(2)}{180-45}{0}   
 \tdplotsetcoord{L2}{sqrt(2)}{180-45}{90}  
 \tdplotsetcoord{M2}{sqrt(2)}{180-45}{180}   
 \tdplotsetcoord{N2}{sqrt(2)}{180-45}{270}  

 \tdplotsetcoord{O2}{sqrt(2)}{90}{45}    
 \tdplotsetcoord{P2}{sqrt(2)}{90}{135}   
 \tdplotsetcoord{Q2}{sqrt(2)}{90}{225}  
 \tdplotsetcoord{R2}{sqrt(2)}{90}{315}    

    \node[draw,circle,inner sep=2pt,fill,blue!60] at (A) {};
    \node[draw,circle,inner sep=2pt,fill,blue!60] at (B) {};
    \node[draw,circle,inner sep=2pt,fill,blue!60] at (C) {};
    \node[draw,circle,inner sep=2pt,fill,blue!60] at (D) {};
    \node[draw,circle,inner sep=2pt,fill,blue!60] at (E) {};
    \node[draw,circle,inner sep=2pt,fill,blue!60] at (F) {};

    \draw[dashed] (C2) -- (D2) -- (A2);
    \draw[dashed](E2) -- (D2) -- (F2);
       \node[draw,circle,inner sep=2pt,fill,blue!50] at (A) {};
       \node[draw,circle,inner sep=2pt,fill,blue!50] at (B) {};
       \node[draw,circle,inner sep=2pt,fill,blue!50] at (C) {};
       \node[draw,circle,inner sep=2pt,fill,blue!50] at (D) {};
       \node[draw,circle,inner sep=2pt,fill,blue!50] at (E) {};
       \node[draw,circle,inner sep=2pt,fill,blue!50] at (F) {};

    \node[draw,circle,inner sep=2pt,fill,blue!50] at (D2) {};
     \node[draw,circle,inner sep=2pt,fill,blue!50] at (G2) {};
     \node[draw,circle,inner sep=2pt,fill,blue!50] at (N2) {};
           \node[draw,circle,inner sep=2pt,fill,blue!50] at (Q2) {};
           \node[draw,circle,inner sep=2pt,fill,blue!50] at (R2) {};
    \draw[dashed] (G2)--(F2);
  \draw[thick] (A) -- (B) -- (C) node[right] {\tiny$ \operatorname{Conv}([1])$};
  \draw[thick] (E) -- (A) -- (F);
  \draw[thick] (E) -- (B) -- (F);
  \draw[thick] (E) -- (C) -- (F);
  \draw[dashed] (C) -- (D) -- (A);
  \draw[dashed](E) -- (D) -- (F);
  \fill[cof,opacity=0.8](A) -- (B) -- (E) -- cycle;
  \fill[pur,opacity=0.8](A) -- (B) -- (F) -- cycle;
  \fill[greeo,opacity=0.8](B) -- (C) -- (E) -- cycle;
  \fill[greet,opacity=0.8](B) -- (C) -- (F) -- cycle;

    \fill[red!30,opacity=0.5](H2) -- (J2) -- (F2) -- cycle;
    \fill[red!70,opacity=0.5](J2) -- (I2) -- (F2) -- cycle;
    \fill[red!90,opacity=0.5] (H2)--(J2) -- (I2) -- (G2) -- cycle;
  \draw (I2)-- (G2) -- (H2) -- (J2) -- (F2)--(H2);
  \draw (J2)--(I2)--(F2);


   \fill[cof,opacity=0.4](A2) -- (B2) -- (E2) -- cycle;
   \fill[pur,opacity=0.4](A2) -- (B2) -- (F2) -- cycle;
   \fill[greeo,opacity=0.4](B2) -- (C2) -- (E2) -- cycle;
   \fill[greet,opacity=0.4](B2) -- (C2) -- (F2) -- cycle;
      \draw[very thick] (A2)  -- (B2) -- (C2) node[right] {\tiny$ \operatorname{Conv}([2])$}  ;
      \draw[very thick]  (E2) -- (A2) -- (F2);
      \draw[very thick]  (E2) -- (B2) -- (F2);
      \draw[very thick]  (E2) -- (C2) -- (F2);

\draw[->] (-1,-2.5)  --  (-3,-5)  node[above]{$\quad\operatorname{NP}(z_1z_3+2z_1^{-1}z_3-z_2z_3-z_2^{-1}z_3+3z_3^{-2})$};
   \node[draw,circle,inner sep=2pt,fill,blue!50] at (A2) {};
   \node[draw,circle,inner sep=2pt,fill,blue!50] at (B2) {};
   \node[draw,circle,inner sep=2pt,fill,blue!50] at (C2) {};
   \node[draw,circle,inner sep=2pt,fill,blue!50] at (E2) {};
    \node[draw,circle,inner sep=2pt,fill,blue!50] at (F2) {};
    \node[draw,circle,inner sep=2pt,fill,blue!50] at (H2) {};
    \node[draw,circle,inner sep=2pt,fill,blue!50] at (I2) {};
    \node[draw,circle,inner sep=2pt,fill,blue!50] at (J2) {};
    \node[draw,circle,inner sep=2pt,fill,blue!50] at (K2) {};
    \node[draw,circle,inner sep=2pt,fill,blue!50] at (L2) {};
    \node[draw,circle,inner sep=2pt,fill,blue!50] at (M2) {};
      \node[draw,circle,inner sep=2pt,fill,blue!50] at (O2) {};
      \node[draw,circle,inner sep=2pt,fill,blue!50] at (P2) {};

  \end{tikzpicture}
\end{center}

\subsection{Laurent--Vandermonde  matrices,  poised sets and algebraic geometry}\label{sec:poised}
Here we follow  \cite{MVOPR-darboux}. The proofs can be easily translated  from there.
\begin{definition}
	We  consider the Laurent--Vandermonde type matrix
	\begin{align*}
	\mathcal V_{k}^m\coloneq \big(\chi^{[k+m]}(\boldsymbol p_1) ,\dots,
	\chi^{[k+m]}(\boldsymbol p_{r_{k,m}})\big)\in \C^{N_{k+m-1}\times r_{k,m}},
	\end{align*}
	made up of truncated  vectors \emph{}of multivariate units $\chi^{[k+m]}(\z)$ evaluated at the nodes.
	We also consider the following truncation  $S_k^m\in\C^{r_{k,m}\times N_{k+m-1}}$ of the lower unitriangular factor $S$ of the Gauss--Borel factorization of the moment matrix
	\begin{align}\label{eq:S-slice}
	S_k^m\coloneq
	\PARENS{
		\begin{matrix}
		S_{[k],[0]} & S_{[k],[1]} &\dots &\I_{|[k]|} & 0_{[k],[k+1]}&\dots & 0_{[k],[k+m-1]}\\
		S_{[k+1],[0]} & S_{[k+1],[1]} &\dots & S_{[k+1],[k]} &\I_{|[k+1]|}&\dots & 0_{[k+1],[k+m-1]}\\
		\vdots & \vdots &  & & & \ddots &\vdots\\
		S_{[k+m-1],[0]} & S_{[k+m-1],[1]} &\dots & & & S_{[k+m-1],[k+m-2]} &\I_{|[k+m-1]|}
		\end{matrix}}.
	\end{align}
\end{definition}
We have the following factorization
\begin{align}\label{lemma}
\Sigma_k^m=S_k^m
\mathcal V_{k}^m.
\end{align}
so that
\begin{align}\label{kers}
\operatorname{Ker} \mathcal V^m_k &\subset	\operatorname{Ker} \Sigma_k^m, &
\operatorname{Im}  \Sigma^m_k &\subset	\operatorname{Im} S_k^m=\C^{r_{k,m}}.&
\end{align}
The poisedness of $\mathcal N_{k,m}$ happens iff
$\operatorname{Ker} \Sigma_k^m =\{0\}$
or equivalently iff
$\dim\operatorname{Im}\Sigma_k^m=r_{k,m}$.

\begin{pro}\label{pro:rank}
	For a poised set $\mathcal N_{k,m}$ the multivariate Laurent--Vandermonde $\mathcal V_k^m$ is a full column rank matrix; i.e., $	\dim\operatorname{Im}\mathcal V_k^m=r_{k,m}$.
\end{pro}

%

The study  of the orthogonal complement of the rank; i.e, the linear subspace $\big(\operatorname{Im}\mathcal V_k^m\big)^\perp\subset \C^{N_{k+m-1}}$ of vectors orthogonal to the image $\operatorname{Im}\mathcal V_k^m$ where $v\in \big(\operatorname{Im}\mathcal V_k^m\big)^\perp$ if
$v^\dagger\mathcal V_k^m=0$, will be useful in the study of poised sets.  As $\operatorname{Im}\mathcal V_k^m\oplus\big(\operatorname{Im}\mathcal V_k^m\big)^\perp=\C^{N_{k+m-1}}$ we have
\begin{align*}
\dim \big(\operatorname{Im}\mathcal V_k^m\big)^\perp +\dim \big(\operatorname{Im}\mathcal V_k^m\big)=N_{k+m-1}.
\end{align*}
\begin{pro}\label{pro:dimensions}
	The  Laurent--Vandermonde matrix $\mathcal V^m_k$ has full column rank if and only if
	\begin{align*}
	\dim \big(\operatorname{Im}\mathcal V_k^m\big)^\perp = N_{k+m-1}-r_{k,m}=N_{k-1}.
	\end{align*}
\end{pro}
\begin{definition}
	For a   Laurent polynomial $K\in\C[z_1^{\pm 1},\dots,z_D^{\pm 1}]$  the corresponding  principal ideal is $(K)\coloneq\{\z^\q K(\z): \q\in\Z^D\}$. Then, we introduce its truncations by taking its intersection with the Laurent polynomials of  longitude less  than $k+m$.  We use the following notation for the truncated principal ideals
	\begin{align*}
	(K)_{k+m-1}\coloneq\C\big\{\z^\q K(\z):\q\in\Z^D\big\}_{\ell(\z^\q K(\z))<k+m}.
	\end{align*}
\end{definition}
It happens that the elements in the orthogonal complement of the rank of the Laurent--Vandermonde matrix are Laurent polynomials with zeroes at the nodes in the algebraic torus
\begin{pro}
	As linear spaces the orthogonal complement of the Laurent--Vandermonde matrix $\big(\operatorname{Im}\mathcal V_k^m\big)^\perp$ and the space of polynomials of longitude less than $k+m$ and  zeroes at  $\mathcal N_{k,m}$ are isomorphic.
\end{pro}
\begin{proof}
	The linear bijection  is
	\begin{align*}
	v=(\bar v_i)_{i=0}^{N_{k+m-1}-1}\in\big(\operatorname{Im}\mathcal V_k^m\big)^\perp&
	\leftrightarrow
	K(\z)=\sum_{i=0}^{N_{k+m-1}-1}v_i \z^{\q_i}
	\end{align*}
	where $K(\z)$ do have zeroes at $\mathcal N_k^m$.
	Indeed, we observe that a vector $v=(v_i)_{i=0}^{N_{k+m-1}-1}\in\big(\operatorname{Im}\mathcal V_k^m\big)^\perp$ can be identified with the   Laurent polynomial $K(\z)=\sum_{i}v_i \z^{\q_i}$ which cancels, as a consequence of $v^\dagger\mathcal V_k^m=0$, at the nodes.
\end{proof}
Given this linear isomorphism, for any Laurent polynomial $K$ with $\ell(K)<k+m$ with zeroes  at $\mathcal N_k^m$ we write $K\in \big(\operatorname{Im}\mathcal V_k^m\big)^\perp$.
\begin{pro}\label{pro:K1}
	Given a Laurent polynomial $K\in \big(\operatorname{Im}\mathcal V_k^m\big)^\perp$ then
	\begin{align*}
	(K)_{k+m-1}\subset \big(\operatorname{Im}\mathcal V_k^m\big)^\perp,
	\end{align*}
	or equivalently
	\begin{align*}
	(K)_{k+m-1}^\perp\supseteq \operatorname{Im}\mathcal V_k^m.
	\end{align*}
\end{pro}
\begin{pro}\label{pro:dimensions-nice}
	If a Laurent polynomial $K$ is nice then
	\begin{align*}
	\dim (K)_{k+\ell(K)-1}=N_{k-1},
	\end{align*}
	and if $K$ is not nice we have
	 	\begin{align*}
	 	\dim (K)_{k+\ell(K)-1}>N_{k-1}.
	 	\end{align*}
\end{pro}
\begin{proof}
	For a nice polynomial $K$ we have $\ell(\z^\q K)=|\q|+\ell(V)$ so that
		\begin{align*}
			(K)_{k+\ell(K)-1}=\C\big\{\z^\q K(\z):\q\in\Z^D\big\}_{|\q |<k}
		\end{align*}
		and therefore
			\begin{align*}
			\dim (K)_{k+\ell(K)-1}=|[k-1]|+\cdots+|[0]|=N_{k-1}.
			\end{align*}
However, if $K$ is not nice then 	$\ell(\z^\q K)<|\q|+\ell(K)$. Therefore, for all $\q\in\Z^D$ with longitude less that $|\q|<k$
we get $\ell(\z^\q K)<k+\ell(K)$, but there are also $\q\in\Z$ with $|\q|\geq k$ so that $\ell(\z^\q K)<k+\ell(K)$ and therefore the linear dimension of the truncated principal ideal is bigger than $N_{k-1}$.
\end{proof}

After this negative result, which shows that non nice Laurent polynomial perturbations of the measure are not well suited for the sample matrix approach to the Darboux  transformation, we give a positive result. Following the ideas of \cite{MVOPR-darboux} we find that nice Laurent polynomials are indeed nice with the sample matrix trick.
\begin{theorem}\label{teorema:ideal-vandermonde}
Given a nice Laurent polynomial $L$ the Laurent--Vandermonde matrix $\mathcal V^m_k$ has  full column rank  if and only if
		\begin{align*}
		(L)_{k+\ell(L)-1}=\big(\operatorname{Im}(\mathcal V^m_k)\big)^\perp.
		\end{align*}
\end{theorem}
\begin{proof}
	It is a consequence of the niceness of the Laurent polynomial $L$ and Propositions \ref{pro:K1}, \ref{pro:dimensions-nice} and \ref{pro:dimensions}.
\end{proof}

From the spectral properties of the shift matrices we deduce
\begin{pro}\label{pro:matrix-ideal}
	The row $(L(\boldsymbol\Upsilon))_\q$, $\q\in\Z^D$, is the  \texttt{longilex} ordering of the entries in the corresponding polynomial $\z^\q L(\z)$.
\end{pro}
Thus, in some way $L(\bUpsilon)$ encodes the same information as the principal ideal of $(L)$ does. To make this  observation  formal we first consider 

\begin{definition}
	We introduce the matrix  $(L(\boldsymbol{\Upsilon}))^{[k,m]} \in\C^{N_{k-1}\times r_{m,k}}$ given by
	{\begin{align*}\hspace*{-1cm}
		(L(\boldsymbol{\Upsilon}))^{[k,m]} \coloneq
		\ccases{	\PARENS{\small
				\begin{matrix}
				(L(\boldsymbol \Upsilon))_{[0],[k]}& \dots & (L(\boldsymbol \Upsilon))_{[0],[m-k+1]}  & 0_{[0],[m-k]} & \dots & 0_{[0],[k+m-1]}\\
				\vdots  &  &&\ddots &\ddots & \vdots & \\
				(L(\boldsymbol \Upsilon))_{[k-2],[k]} &\dots& 	& &(L(\boldsymbol \Upsilon))_{[k-2],[k+m-2 ]} & 0_{[k-2],[k+m-1 ]}\\
				(L(\boldsymbol \Upsilon))_{[k-1],[k]} &\dots&  &	& \dots&(L(\boldsymbol \Upsilon))_{[k-1],[k+m-1 ]}
				\end{matrix}
			}, &k\leq m,\\
			\PARENS{
				\begin{matrix}
				0_{[0],[k]}&0_{[0],[k+1]} &\dots & 0_{[0],[k+m-1]}\\
				\vdots      &   \vdots&       & \vdots\\
				0_{[k-m],[k]} & 0_{[k-m],[k+1]}  &\dots & 0_{[k-m],[k+m-1]}\\\hline
				(L(\boldsymbol \Upsilon))_{[k-m+1],[k]} & 0_{[k-m+1],[k+1]} & \dots & 0_{[k-m+1],[k+m-1]}\\
				(L(\boldsymbol \Upsilon))_{[k-m+2],[k]} & 	(L(\boldsymbol \Upsilon))_{[k-m+2],[k+1]} & \ddots & \vdots\\
				\vdots & \vdots &\ddots& 0_{[k-2],[k+m-1 ]} \\
				(L(\boldsymbol \Upsilon))_{[k-1],[k]} & (L(\boldsymbol \Upsilon))_{[k-1],[k+1]}  &\dots& 	(L(\boldsymbol \Upsilon))_{[k-1],[k+m-1 ]}
				\end{matrix}
			}, &k\geq m.
		}
		\end{align*}}
	We collect this matrix and the truncation $L(\boldsymbol{\Upsilon})^{[k]}$ in  
	\begin{align*}
	\big(L(\boldsymbol{\Upsilon}))_k^m\coloneq (L(\boldsymbol{\Upsilon})^{[k]}, L(\boldsymbol{\Upsilon})^{[k,m]}\big)\in\C^{N_{k-1}\times N_{k+m-1}}.
	\end{align*}
\end{definition}

\begin{pro}\label{pro:idel-vandermonde2}
	We have the following isomorphism 
	\begin{align*}
	K=\sum_{|\q|<k+m}K_\q\z^\q\in(L)_{k+m-1}\Leftrightarrow 
	(K_{\q_0},\dots, K_{\q_{N_{k+m-1}}})=(a_0,\dots, a_{N_{k-1}})\big(L(\boldsymbol{\Upsilon})\big)^m_k
	\end{align*}
	between the truncated ideal $(L)_{k+m-1}$ and the orbit of $\C^{N_{k-1}}$ under the linear morphism 
	$(L\boldsymbol{\Upsilon}))_k^m.$. 
	Here we have ordered the multi-indices $\q$ in $L$ according to the \texttt{longilex} order,  $\q_0<\dots<\q_{N_{k+m-1}}$.
\end{pro}
\begin{proof}
	Just recall that $K$ is going to be a linear combination of the polynomials $\z^\q L(\z)$, $|\q|<k$. Thus,  the row vector by arranging its coefficients according to our order must in the rank of $(L\boldsymbol{\Upsilon}))_k^m$.
\end{proof}

Now we have to show that full rankness of the Laurent--Vandermonde matrix is a sufficient condition
\begin{theorem}\label{theorem:poised-fullrank}
	Let $L\in\C[\z^{\pm 1}]$ be a nice Laurent polynomial with $\ell(L)=m$. Then, the set of nodes $\mathcal N_{k,m}\subset Z(L)$ is poised if and only if  the Laurent--Vandermonde matrix $\mathcal V^m_k$ has full column  rank.
\end{theorem}
\begin{proof}
	Let uvv s assume the contrary, then the sample matrix  $\Sigma_k^m$ is singular, and we can find a nontrivial linear dependence among its rows $(\Sigma_k^m)_i$, $i\in\{1,\dots,r_{k,m}\}$ of the form
	\begin{align*}
	\sum_{i=1}^{r_{k,m}}c_i(\Sigma_k^m)_i&=0,
	\end{align*}
	for some nontrivial scalars $\{c_1,\dots,c_{r_{k,m}}\}$.
	But, according to Lemma \ref{lemma} $(\Sigma_k^m)_i=(S^m_k)_i\mathcal V^m_k$, where $(S^m_k)_i$ is the $i$-th row of $S^k_m$ and we can write
	\begin{align*}
	\Big(\sum_{i=1}^{r_{k,m}}c_i(S^m_k)_i\Big)\mathcal V^m_k=0,
	\end{align*}
	so that
	\begin{align*}
	\sum_{i=1}^{r_{k,m}}c_i(S^m_k)_i\in\big(\operatorname{Im}\mathcal V^m_k\big)^\perp,
	\end{align*}
	and given the column full rankness of the Laurent--Vandermonde matrix, see Theorem \ref{teorema:ideal-vandermonde},
	we can write
	\begin{align*}
	(c_1,\dots,c_{r_{k,m}})S^m_k\in \big(\operatorname{Im}\mathcal V^m_k\big)^\perp,
	\end{align*}
and following Proposition  \ref{pro:idel-vandermonde2} we get
	\begin{align}\label{eq:SQ}
	(a_0,\dots, a_{N_{k-1}})\big(L(\boldsymbol{\Upsilon})\big)^m_k=(c_1,\dots,c_{r_{k,m}})S^m_k,
	\end{align}
	for some non trivial set of $c$'s.
	
Now, we recall that  $S$ is  lower unitriangular by blocks 
	\begin{align*}
	S=\PARENS{
		\begin{array}{c | c |c c}
		S^{[k]}  & 0  & 0 &\dots\\
		\hline
		\multicolumn{2}{c|} {S^m_k} & 0 &\cdots\\[2pt]\hline
	\multicolumn{4}{c}{*}
			\end{array}
		}
			\end{align*}
	and that either $L(\boldsymbol \Upsilon)$ or $L(\boldsymbol J)$ are block banded matrices with only the first $m$ block superdiagonals non zero
	\begin{align*}
	L(\boldsymbol{\Upsilon})&=\PARENS{\begin{array}{cc|c}
		\big(L(\boldsymbol{\Upsilon})\big)^{[k]} & \big(L(\boldsymbol{\Upsilon})\big)^{[k,m]} &0\\\hline
		*&*&*
		\end{array}}, & 
	L(\boldsymbol{J})&=\PARENS{\begin{array}{cc|c}
		\big(L(\boldsymbol{J})\big)^{[k]} & \big(L(\boldsymbol{J})\big)^{[k,m]} &0\\\hline
		*&*&*
		\end{array}}.
	\end{align*}
	Now, we focus on the relation
	\begin{align*}
	SL(\boldsymbol \Upsilon)=L(\boldsymbol J)S.
	\end{align*}
	which can be written   as it  follows
	\begin{align}\label{eq:laleche}
	S^{[k]}\big(L(\boldsymbol{\Upsilon})\big)^m_k=
	\big(L(\boldsymbol J)\big)^{[k,m]} S^m_k+\big ((L(\boldsymbol J))^{[k]}S^{[k]},0\big).
	\end{align}
	If we assume that \eqref{eq:SQ}  holds and left multiply \eqref{eq:laleche} with the nonzero vector $(a_0,\dots, a_{N_{k-1}})\big(S^{[k]}\big)^{-1}$ we get
	\begin{align*}
	(a_0,\dots, a_{N_{k-1}})\big(L(\boldsymbol{\Upsilon})\big)^m_k=
	(a_0,\dots, a_{N_{k-1}})\big(S^{[k]}\big)^{-1}\big(L(\boldsymbol J)\big)^{[k,m]}S^m_k+
\big (	(a_0,\dots, a_{N_{k-1}})\big(S^{[k]}\big)^{-1}(L(\boldsymbol J))^{[k]}S^{[k]},0\big).
	\end{align*}
	Thus,  \eqref{eq:SQ}  holds if and only if
	\begin{align*}
	(a_0,\dots, a_{N_{k-1}})\big(S^{[k]}\big)^{-1}\big(L(\boldsymbol J)\big)^{[k]}S^{[k]}=0,
	\end{align*}
	or equivalently if and only if $\det \Big(\big(L(\boldsymbol J)\big)^{[k]}\Big)=0$.
	Thus, recalling first Proposition \ref{pro:regularity-truncation-jacobi} and second our inital assumption that the  measures $\d\mu$ and its perturbation $L\d\mu$ do have MVOLPUT, the results follows.
\end{proof}

 We say that a Laurent polynomial is irreducible if it can not be written as the product of non invertible Laurent polynomials ---remember that the invertible Laurent polynomials are the units--- . As the ring of Laurent polynomials $\C[z_1^{\pm 1},\dots,z_D^{\pm 1}]$ is  a UFD\footnote{Up to units there is a unique factorization in terms of irreducible polynomials} any prime polynomial is irreducible.
 A Laurent polynomial $L$, different from zero or a unit,  is prime if  whenever $L$ divides the product of two Laurent polynomials $L_1L_2$ then $L$  divides $L_1$ or divides $L_2$. Equivalently, $L$ is a prime Laurent polynomial if and only if the principal ideal $(L)$ is a nonzero prime ideal.\footnote{An ideal in the ring of Laurent polynomials is a prime ideal if whenever the product of two Laurent polynomials $L_1L_2$ belongs to the ideal we can ensure that at least one of the factors belongs to the given ideal.}

 Any Laurent polynomial $L(\z)=\z^\q Q(\z)$ can be written as the product of a unit $z^\q$ and a polynomial $Q(\z)$, and among the possible $\q$ there exists only one of minimum longitude.\footnote{Take among all the mult-indices in the support of the Laurent polynomial those with the most negative component, say $-\lambda_i\in\Z_-$, in each direction $i\in\{1,\dots,D\}$, then can write   $ L(\z)=z^{-\boldsymbol{\lambda}} Q(\z)$, and all the other units are of the form $-\boldsymbol{\lambda}+\Z_-^D$.}
 To each irreducible Laurent polynomial $L\in\C[z_1^{\pm 1},\dots,z_D^{\pm 1}]$ it corresponds  a unique irreducible polynomial $Q\in\C[z_1,\dots,z_D]$ and  a unique unit $a\z^\q$, with $L(\z)=a\z^\q Q(\z)$.  Consequently, $Z(L)=Z(Q)\cap (\C^*)^D$ is the algebraic hypersurface $Q(\z)=0$ within the algebraic torus.

 \begin{pro}
 	Any nice Laurent polynomial $L$ could be expressed as the product of a unit $u$ with different nice irreducible  Laurent polynomials $\{L_1,\dots,L_N\}$ with multiplicities $\{d_1,\dots,d_N\}$
 	\begin{align*}
 	L=uL_1^{d_1}\cdots L_N^{d_N}
 	\end{align*}
 	in a unique form up to unities.
 \end{pro}
 \begin{proof}
 	Any Laurent polynomial can be factored in terms of its prime factors. But according to Proposition \ref{pro:prod-nice} if $L$ is nice the factors must be nice as well.
 \end{proof}

\begin{pro}\label{pro:laurent-hilbert}
Let  $L=uL_1\cdots L_N$  be a product of a unit $u\in\C[\z^{\pm 1}]$ with $N$ different prime Laurent polynomials $\{L_1,\dots,L_N\}$, where none of them are units, and suppose that  $K\in\C[\z^{\pm 1}]$ is such that  $Z(L)\subset Z(K)$ then $K\in (L)$.
\end{pro}
\begin{proof}
   Let us consider the product of a unit $u\in\C[\z^{\pm 1}]$ with $N$ different  Laurent irreducible polynomials $\{L_1,\cdots, L_N\}$, none of them units in $\C[\z^{\pm 1}]$ ,  and let $\{Q_1,\dots,Q_N\}$ be
   the corresponding irreducible polynomials,  none of them are units in $\C[\z]$; then, the related algebraic hypersurface is reducible
 $Z(L)=\cup_{i=1}^NZ(L_i)=\Big(\cup_{i=1}^NZ(Q_i)\Big)\cap (\C^*)^D$.
  Now, given $L=uL_1\cdots L_N$ let $K\in\C[z_1^{\pm 1},\dots,z_D^{\pm 1}]$ be a Laurent polynomial such that  $Z(L)\subset Z(K)$, and  write $K=\z^\q P$ where $\q$ has minimum longitude and $P\in\C[z_1,\dots,z_D]$ is a polynomial which is not a unit in $\C[\z]$. Then, we know that  $\Big(\cup_{i=1}^NZ(Q_i)\Big)\cap (\C^*)^D\subset Z(P)$, thus as $Z(P)$, being the null set of a continuous function, is a closed set,  we have  that $\cup_{i=1}^NZ(Q_i)=\overline{\Big(\cup_{i=1}^NZ(Q_i)\Big)\cap (\C^*)^D}\subset Z(P)$. Thus, according to the Hilbert’s Nullstellensatz  $P\in\sqrt{(Q)}$, $Q\coloneq \prod_{i=1}^NQ_i$;  being different prime polynomials this is a radical ideal and, consequently, $P\in (Q)$. From where we immediately conclude $K\in (L)$.
\end{proof}

\begin{theorem}\label{theorem:nice}
Let  $L=uL_1\cdots L_N$  be a product of a unit $u\in\C[\z^{\pm 1}]$ with $N$ different prime Laurent polynomials $\{L_1,\dots,L_N\}$, where none of them are units, and $m=\ell(L)$.
Then, the  node set $\mathcal N_{k,m}\subset Z( L)\subset( \C^*)^D$ is  poised if the nodes dot not belong to any further algebraic hypersurface, different from $Z(L)$, of a Laurent polynomial  $\hat L$ with non trivial truncated ideal
	$(\hat L)_{k+m-1}\neq\{0\}$.
\end{theorem}
\begin{proof}
	Given that $L$ is a nice Laurent polynomial we have that
	\begin{align*}
\dim	(L)_{k+m-1}=N_{k-1}
		\end{align*}
and as  $(L)_{k+m-1}\subseteq \big(\operatorname{Im}\mathcal V_k^m\big)^\perp$ we conclude that
$\dim \big(\mathcal V_k^m\big)^\perp \geq N_{k-1}$. From Proposition \ref{pro:laurent-hilbert} we see that there are no more linear constraints derived from the inclusion of the node set in the algebraic hypersurface of $L$.
	Hence, if the nodes do not belong to any further algebraic hypersurface of a Laurent polynomial  $\hat L$ such that
	$(\hat L)_{k+m-1}\neq\{0\}$ we have $\dim \big(\mathcal V_k^m\big)^\perp = N_{k-1}$ and the set is  poised.
	\end{proof}

\subsection{Darboux transformations for the Lebesgue--Haar measure}
As an example let us consider  Darboux perturbations for the Lebesgue--Haar measure in the unit torus $\T^D$
\begin{align*}
\d\mu(\boldsymbol{\theta})=\frac{1}{(2\pi)^D}\d\theta_1\cdots\d\theta_D
\end{align*}
the MVOLPUT in this case are the units in the Laurent polynomial ring or one may say the multivariate Fourier basis
\begin{align*}
\phi_{[k]}=\chi_{[k]}.
\end{align*}
Now, the moment matrix is the identity matrix  $G=\I$, therefore the transformed or perturbed moment matrix will be
\begin{align*}
TG=L(\boldsymbol{\Upsilon}).
\end{align*}
Consequently, in order to have orthogonal polynomials, i.e. a Gaussian--Borel factorization, its block principal minors should be regular
\begin{align*}
\det\Big(\big(L(\boldsymbol{\Upsilon})\big)^{[k]}\Big)&\neq 0, & k\in\{1,2,\dots\}.
\end{align*}
Which is achieved whenever the Laurent polynomial $L$ is definite positive, i.e.,
        \begin{align*}
       \det\Big(\big(L(\boldsymbol{\Upsilon})\big)^{[k]}\Big)&> 0, & k\in\{1,2,\dots\}.
        \end{align*}

To illustrate this result let us explore a simple case,  take for $D=2$ the following nice definite positive Laurent polynomial $L=z_1+z_1^{-1}+z_2+z_2^{-1}+5$,
that preserves reality and  definite positiveness. The original measure is the Lebesgue measure $\d\theta_1\d\theta_2$ in $\T^2$ while the perturbed measure is
$(2\cos(\theta_1)+2\cos(\theta_2)+5)\d\theta_1\d\theta_2$. Now the longitude is $\ell(L)=1$ and the Laurent polynomial  can be  written  as follows
 \begin{align*}
 L=&z_1^{-1}z_2^{-1}Q(z_1,z_2), & Q\coloneq&z_1^2z_2+z_1z_2^2
+z_2+z_1+5z_1z_2
 \end{align*}
 being $Q\in\C[z_1,z_2]$ an irreducible polynomial (according to Maple). The corresponding real algebraic curve  plot produced by Maple is
 \begin{center}
 \includegraphics[scale=0.5] {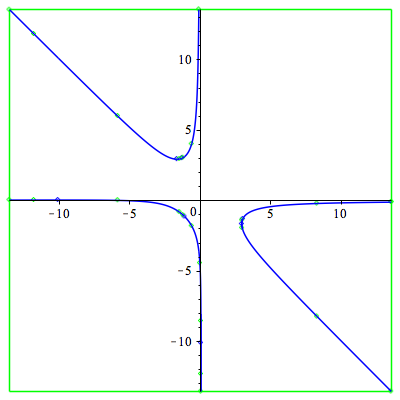}
 \end{center}
 Observe that for $(z_1,z_2)\in Z(L)$ we have
 \begin{align*}
 	z_2&=\frac{1}{2}(-z_1-z_1^{-1}-5\pm F(z_1)),&
 	z_2^{-1}&=\frac{1}{2}(-z_1-z_1^{-1}-5\mp F(z_1)),
 \end{align*}
 where $F\coloneq\sqrt{(z_1+z_1^{-1}+5)^2-4}$.

We know that poised sets of zeroes of $L$ do exist as long as we impose to these nodes not to belong to any other algebraic curve of a Laurent polynomial $\hat L$ having a non trivial truncated ideal $(\hat L)_{k}$.

Let us give some details for the first non trivial case with $k=1$; now we have $r_{1,1}=4$ and the set of nodes is $\mathcal N_{1,1}=\{\boldsymbol p_{1},\boldsymbol p_{2},\boldsymbol p_{3},\boldsymbol p_{4}\}\subset Z(Q)\subset (\C^*)^2$ where $\boldsymbol p_i=(p_{i,1},p_{i,2})^\top$. The sample matrix is
\begin{align*}
\Sigma^1_1=
\PARENS{
	\begin{matrix}
	p_{1,1}^{-1} & p_{2,1}^{-1} &p_{3,1}^{-1}  & p_{4,1}^{-1} \\
	p_{1,2}^{-1} & p_{2,2}^{-1} &p_{3,2}^{-1}  & p_{4,2}^{-1} \\
		p_{1,2} & p_{2,2} & p_{3,2}  & p_{4,2}\\
			p_{1,1} & p_{2,1} & p_{3,1} & p_{4,1}
	\end{matrix}
	}\in \C^{4\times 4}.
\end{align*}

 Thus, if we request to the nodes to not belong to another algebraic curve as said before, we know that it is rank 4 and the set is poised.
A similar conclusion is obtained directly, the determinant can be easily computed to be
\begin{align*}
\det \Sigma^1_1&=	\frac{1}{4}\begin{vmatrix}
p_{1,1}^{-1} & p_{2,1}^{-1} &p_{3,1}^{-1}  & p_{4,1}^{-1} \\
-5+F(p_{1,1})&-5+F(p_{2,1})&-5+F(p_{3,1}) & -5+F(p_{4,1})\\
-5-F(p_{1,1})&-5-F(p_{2,1})&-5-F(p_{3,1}) & -5-F(p_{4,1})\\
p_{1,1} & p_{2,1} & p_{3,1} & p_{4,1}
\end{vmatrix}\\
&=5\begin{vmatrix} 1& 1& 1& 1\\
p_{1,1} & p_{2,1} & p_{3,1} & p_{4,1}\\
p_{1,1}^{-1} & p_{2,1}^{-1} &p_{3,1}^{-1}  & p_{4,1}^{-1} \\
F(p_{1,1})&F(p_{2,1})&F(p_{3,1}) & F(p_{4,1})
\end{vmatrix}
\end{align*}
And we see, by direct substitution, that generically it does not cancel.

For those sets of nodes the MVOLPUT of longitude $1$ are
          {\begin{multline*}\phi_{[1]}(z_1,z_2)=\frac{1}{z_1+z_1^{-1}+z_2+z_2^{-1}+5}
     	\PARENS{
     		\begin{matrix}
     		1&1&1&0&0&0&0&0\\
     		0&1&0&1&0&1&0&0\\
     		0&0&1&0&1&0&1&0\\
     		0&0&0&0&0&1&1&1
     		\end{matrix}
     	}
     	\left[
     	\PARENS{
     		\begin{matrix}
     	z_1^{-2}\\z_1^{-1}z_2^{-1}\\z_1^{-1}z_2\\z_2^{-2}\\z_2^2\\z_1z_2^{-1}\\z_1z_2\\z_1^2
     	\end{matrix}
     	}\right.
     	\\
     	-\left.
     \PARENS{
     	\begin{matrix}
     	p_{1,1}^{-2}&p_{2,1}^{-2}&p_{3,1}^{-2}&p_{4,1}^{-2}
     	\\
     	p_{1,1}^{-1}p_{1,2}^{-1} &p_{2,1}^{-1}p_{2,2}^{-1} &p_{3,1}^{-1}p_{3,2}^{-1} & p_{4,1}^{-1}p_{4,2}^{-1}
     	\\
     	p_{1,1}^{-1}p_{1,2} & p_{2,1}^{-1}p_{2,2} & p_{3,1}^{-1}p_{3,2} & p_{4,1}^{-1}p_{4,2}
     	\\
     	p_{1,2}^{-2} & p_{2,2}^{-2} & p_{3,2}^{-2} & p_{4,2}^{-2}
     	\\
     	p_{1,2}^2 &  p_{2,2}^2 &  p_{3,2}^2 &  p_{4,2}^2
     	\\
     	p_{1,1}p_{1,2}^{-1} &   p_{2,1}p_{2,2}^{-1} &   p_{3,1}p_{3,2}^{-1}&   p_{4,1}p_{4,2}^{-1}
     	\\
     	p_{1,1}p_{1,2} &   p_{2,1}p_{2,2} &   p_{3,1}p_{3,2} &   p_{4,1}p_{4,2}
     	\\
     	p_{1,1}^2 & p_{2,1}^2 & p_{3,1}^2 & p_{4,1}^2
     	\end{matrix}}
     \PARENS{\begin{matrix}
     p_{1,1}^{-1} & p_{2,1}^{-1} &p_{3,1}^{-1}  & p_{4,1}^{-1} \\
     p_{1,2}^{-1} & p_{2,2}^{-1} &p_{3,2}^{-1}  & p_{4,2}^{-1} \\
     p_{1,2} & p_{2,2} & p_{3,2}  & p_{4,2}\\
     p_{1,1} & p_{2,1} & p_{3,1} & p_{4,1}
     \end{matrix}
     }^{-1}
     \PARENS{
     	\begin{matrix}
     z_1^{-1}\\z_2^{-1}\\z_2\\z_1
     \end{matrix}
     }
     \right]
     \end{multline*}
     }
  Let us  check that we are dealing with a vector of four longitude 1 Laurent polynomials. Observe    
     \begin{align*}
     \PARENS{
     	\begin{matrix}
     	1&1&1&0&0&0&0&0\\
     	0&1&0&1&0&1&0&0\\
     	0&0&1&0&1&0&1&0\\
     	0&0&0&0&0&1&1&1
     	\end{matrix}
     }
     \PARENS{
     	\begin{matrix}
     	z_1^{-2}\\z_1^{-1}z_2^{-1}\\z_1^{-1}z_2\\z_2^{-2}\\z_2^2\\z_1z_2^{-1}\\z_1z_2\\z_1^2
     	\end{matrix}
     }&=\PARENS{
     \begin{matrix}
     z_1^{-1}(z_1^{-1}+z_2^{-1}+z_2)\\
      z_2^{-1}(z_1^{-1}+z_2^{-1}+z_1)\\
       z_2(z_1^{-1}+z_2+z_1)\\
        z_1(z_2^{-1}+z_2+z_1)
     \end{matrix}
     }
     \end{align*}
     and
     \begin{multline*}
     	\PARENS{
     		\begin{matrix}
     		1&1&1&0&0&0&0&0\\
     		0&1&0&1&0&1&0&0\\
     		0&0&1&0&1&0&1&0\\
     		0&0&0&0&0&1&1&1
     		\end{matrix}
     	}
     	\PARENS{
     		\begin{matrix}
     		p_{1,1}^{-2}&p_{2,1}^{-2}&p_{3,1}^{-2}&p_{4,1}^{-2}
     		\\
     		p_{1,1}^{-1}p_{1,2}^{-1} &p_{2,1}^{-1}p_{2,2}^{-1} &p_{3,1}^{-1}p_{3,2}^{-1} & p_{4,1}^{-1}p_{4,2}^{-1}
     		\\
     		p_{1,1}^{-1}p_{1,2} & p_{2,1}^{-1}p_{2,2} & p_{3,1}^{-1}p_{3,2} & p_{4,1}^{-1}p_{4,2}
     		\\
     		p_{1,2}^{-2} & p_{2,2}^{-2} & p_{3,2}^{-2} & p_{4,2}^{-2}
     		\\
     		p_{1,2}^2 &  p_{2,2}^2 &  p_{3,2}^2 &  p_{4,2}^2
     		\\
     		p_{1,1}p_{1,2}^{-1} &   p_{2,1}p_{2,2}^{-1} &   p_{3,1}p_{3,2}^{-1}&   p_{4,1}p_{4,2}^{-1}
     		\\
     		p_{1,1}p_{1,2} &   p_{2,1}p_{2,2} &   p_{3,1}p_{3,2} &   p_{4,1}p_{4,2}
     		\\
     		p_{1,1}^2 & p_{2,1}^2 & p_{3,1}^2 & p_{4,1}^2
     		\end{matrix}}\big(\Sigma_1^1\big)^{-1}\\=
     	\PARENS{
     		\begin{matrix}
     		-p_{1,1}^{-1}(5+p_{1,1}) &	-p_{2,1}^{-1}(5+p_{2,1}) &	-p_{3,1}^{-1}(5+p_{3,1}) &	-p_{4,1}^{-1}(5+p_{4,1})\\
     			-p_{1,2}^{-1}(5+p_{1,2}) &	-p_{2,2}^{-1}(5+p_{2,2}) &	-p_{3,2}^{-1}(5+p_{3,2}) &	-p_{4,2}^{-1}(5+p_{4,2})\\
     				-p_{1,2}(5+p_{1,2}^{-1}) &	-p_{2,2}(5+p_{2,2}^{-1}) &	-p_{3,2}(5+p_{3,2}^{-1})  &	--p_{4,2}(5+p_{4,2}^{-1}) \\
     						-p_{1,1}(5+p_{1,1}^{-1}) &	-p_{2,1}(5+p_{2,1}^{-1}) &	-p_{3,1}(5+p_{3,1}^{-1})  &	-p_{4,1}(5+p_{4,1}^{-1}) \\
     		\end{matrix}}\big(\Sigma_1^1\big)^{-1}\\=-5-\PARENS{\begin{matrix}
     		1&1&1&1\\ 1&1&1&1\\1&1&1&1\\1&1&1&1
     		\end{matrix}}\big(\Sigma_1^1\big)^{-1}.
     	\end{multline*}
Then, we notice that from 
\begin{align*}
\PARENS{\begin{matrix}
	1&1&1&1\\ 1&1&1&1\\1&1&1&1\\1&1&1&1
	\end{matrix}}\Sigma_1^1&=-5 \PARENS{\begin{matrix}
	1&1&1&1\\ 1&1&1&1\\1&1&1&1\\1&1&1&1
	\end{matrix}} &&\Longrightarrow &-\PARENS{\begin{matrix}
	1&1&1&1\\ 1&1&1&1\\1&1&1&1\\1&1&1&1
	\end{matrix}}\big(\Sigma_1^1\big)^{-1}&=\frac{1}{5}\PARENS{\begin{matrix}
	1&1&1&1\\ 1&1&1&1\\1&1&1&1\\1&1&1&1
	\end{matrix}}&
\end{align*}
we see that the \emph{marvelous event of the disappearing of the nodes} occurs as we get
 \begin{multline*}
 \PARENS{
 	\begin{matrix}
 	1&1&1&0&0&0&0&0\\
 	0&1&0&1&0&1&0&0\\
 	0&0&1&0&1&0&1&0\\
 	0&0&0&0&0&1&1&1
 	\end{matrix}
 }
 \PARENS{
 	\begin{matrix}
 	p_{1,1}^{-2}&p_{2,1}^{-2}&p_{3,1}^{-2}&p_{4,1}^{-2}
 	\\
 	p_{1,1}^{-1}p_{1,2}^{-1} &p_{2,1}^{-1}p_{2,2}^{-1} &p_{3,1}^{-1}p_{3,2}^{-1} & p_{4,1}^{-1}p_{4,2}^{-1}
 	\\
 	p_{1,1}^{-1}p_{1,2} & p_{2,1}^{-1}p_{2,2} & p_{3,1}^{-1}p_{3,2} & p_{4,1}^{-1}p_{4,2}
 	\\
 	p_{1,2}^{-2} & p_{2,2}^{-2} & p_{3,2}^{-2} & p_{4,2}^{-2}
 	\\
 	p_{1,2}^2 &  p_{2,2}^2 &  p_{3,2}^2 &  p_{4,2}^2
 	\\
 	p_{1,1}p_{1,2}^{-1} &   p_{2,1}p_{2,2}^{-1} &   p_{3,1}p_{3,2}^{-1}&   p_{4,1}p_{4,2}^{-1}
 	\\
 	p_{1,1}p_{1,2} &   p_{2,1}p_{2,2} &   p_{3,1}p_{3,2} &   p_{4,1}p_{4,2}
 	\\
 	p_{1,1}^2 & p_{2,1}^2 & p_{3,1}^2 & p_{4,1}^2
 	\end{matrix}} \PARENS{\begin{matrix}
 	p_{1,1}^{-1} & p_{2,1}^{-1} &p_{3,1}^{-1}  & p_{4,1}^{-1} \\
 	p_{1,2}^{-1} & p_{2,2}^{-1} &p_{3,2}^{-1}  & p_{4,2}^{-1} \\
 	p_{1,2} & p_{2,2} & p_{3,2}  & p_{4,2}\\
 	p_{1,1} & p_{2,1} & p_{3,1} & p_{4,1}
 	\end{matrix}
 }^{-1}\\=-5+\frac{1}{5}\PARENS{\begin{matrix}
 	1&1&1&1\\ 1&1&1&1\\1&1&1&1\\1&1&1&1
 	\end{matrix}}.
 \end{multline*} 
 Moreover, we get
  \begin{multline*}
 - \PARENS{
  	\begin{matrix}
  	1&1&1&0&0&0&0&0\\
  	0&1&0&1&0&1&0&0\\
  	0&0&1&0&1&0&1&0\\
  	0&0&0&0&0&1&1&1
  	\end{matrix}
  }
  \PARENS{
  	\begin{matrix}
  	p_{1,1}^{-2}&p_{2,1}^{-2}&p_{3,1}^{-2}&p_{4,1}^{-2}
  	\\
  	p_{1,1}^{-1}p_{1,2}^{-1} &p_{2,1}^{-1}p_{2,2}^{-1} &p_{3,1}^{-1}p_{3,2}^{-1} & p_{4,1}^{-1}p_{4,2}^{-1}
  	\\
  	p_{1,1}^{-1}p_{1,2} & p_{2,1}^{-1}p_{2,2} & p_{3,1}^{-1}p_{3,2} & p_{4,1}^{-1}p_{4,2}
  	\\
  	p_{1,2}^{-2} & p_{2,2}^{-2} & p_{3,2}^{-2} & p_{4,2}^{-2}
  	\\
  	p_{1,2}^2 &  p_{2,2}^2 &  p_{3,2}^2 &  p_{4,2}^2
  	\\
  	p_{1,1}p_{1,2}^{-1} &   p_{2,1}p_{2,2}^{-1} &   p_{3,1}p_{3,2}^{-1}&   p_{4,1}p_{4,2}^{-1}
  	\\
  	p_{1,1}p_{1,2} &   p_{2,1}p_{2,2} &   p_{3,1}p_{3,2} &   p_{4,1}p_{4,2}
  	\\
  	p_{1,1}^2 & p_{2,1}^2 & p_{3,1}^2 & p_{4,1}^2
  	\end{matrix}}\big(\Sigma_1^1\big)^{-1} \PARENS{
  	\begin{matrix}
  	z_1^{-1}\\z_2^{-1}\\z_2\\z_1
  	\end{matrix}}\\=\left[5-\frac{1}{5}\PARENS{\begin{matrix}
  	1&1&1&1\\ 1&1&1&1\\1&1&1&1\\1&1&1&1
  	\end{matrix}}\right] \PARENS{
  	\begin{matrix}
  	z_1^{-1}\\z_2^{-1}\\z_2\\z_1
  	\end{matrix}}=
-  \PARENS{
  	\begin{matrix}
  	5z_1^{-1}+1-\frac{L}{5}\\[3pt]
  	5z_2^{-1}+1-\frac{L}{5}\\[3pt]
  	5z_2+1-\frac{L}{5}\\[3pt]
  	5z_1+1-\frac{L}{5}
  	\end{matrix}
  	}=
  	  \PARENS{
  	  	\begin{matrix}
  	  	z_1^{-1}(z_1+5)-\frac{L}{5}\\[3pt]
  	  	z_2^{-1}(z_2+5)-\frac{L}{5}\\[3pt]
  	  	z_2(z_2^{-1}+5)-\frac{L}{5}\\[3pt]
  	  	z_1(z_1^{-1}+5)-\frac{L}{5}
  	  	\end{matrix}
  	  }.
  \end{multline*} 
 
 Now, we come back and explicitly evaluate $\phi_{[1]}$ to get
 \begin{align*}
 \phi_{[1]}&=\frac{1}{L}\left[\PARENS{
 	\begin{matrix}
 	z_1^{-1}(z_1^{-1}+z_2^{-1}+z_2)\\
 	z_2^{-1}(z_1^{-1}+z_2^{-1}+z_1)\\
 	z_2(z_1^{-1}+z_2+z_1)\\
 	z_1(z_2^{-1}+z_2+z_1)
 	\end{matrix}
 }+  \PARENS{
 \begin{matrix}
 z_1^{-1}(z_1+5)-\frac{L}{5}\\[3pt]
 z_2^{-1}(z_2+5)-\frac{L}{5}\\[3pt]
 z_2(z_2^{-1}+5)-\frac{L}{5}\\[3pt]
 z_1(z_1^{-1}+5)-\frac{L}{5}
 \end{matrix}
}
 \right]=\PARENS{
 	\begin{matrix}
 	z_1^{-1}-\frac{1}{5}\\[3pt]
 	 	z_2^{-1}-\frac{1}{5}\\[3pt]
 	 	 	z_2-\frac{1}{5}\\[3pt]
 	 	 	 	z_1-\frac{1}{5}
 	\end{matrix}}
 \end{align*}   
     and we see that another marvelous event occurs, 
     the Laurent polynomial in the numerator   factors out having as one of the factors the Laurent polynomial in the denominator which can be cleared out. Finally, we get the longitude one Laurent polynomials, orthogonal to 1 according to the measure $L\d\theta_1\d\theta_2$.

     For $k=2$  now we have $r_{2,1}=8$ and the set of nodes is $\mathcal N_{2,1}=\{\boldsymbol p_{1},\boldsymbol p_{2},\boldsymbol p_{3},\boldsymbol p_{4},\boldsymbol p_{5},\boldsymbol p_{6},\boldsymbol p_{7},\boldsymbol p_{8}\}\subset Z(Q)\subset (\C^*)^2$ where $\boldsymbol p_i=(p_{i,1},p_{i,2})^\top$. The sample matrix is
     \begin{align*}
     \Sigma^1_2=
 \PARENS{
 	\begin{matrix}
 	p_{1,1}^{-2}&p_{2,1}^{-2}&p_{3,1}^{-2}&p_{4,1}^{-2}&	p_{5,1}^{-2}&p_{6,1}^{-2}&p_{7,1}^{-2}&p_{8,1}^{-2}
 	\\
 	p_{1,1}^{-1}p_{1,2}^{-1} &p_{2,1}^{-1}p_{2,2}^{-1} &p_{3,1}^{-1}p_{3,2}^{-1} & p_{4,1}^{-1}p_{4,2}^{-1} &p_{5,1}^{-1}p_{5,2}^{-1} &p_{6,1}^{-1}p_{6,2}^{-1} &p_{7,1}^{-1}p_{7,2}^{-1} & p_{8,1}^{-1}p_{8,2}^{-1}
 	\\
 	p_{1,1}^{-1}p_{1,2} & p_{2,1}^{-1}p_{2,2} & p_{3,1}^{-1}p_{3,2} & p_{4,1}^{-1}p_{4,2} &
 	 	p_{5,1}^{-1}p_{5,2} & p_{6,1}^{-1}p_{6,2} & p_{7,1}^{-1}p_{7,2} & p_{8,1}^{-1}p_{8,2}
 	\\
 	p_{1,2}^{-2} & p_{2,2}^{-2} & p_{3,2}^{-2} & p_{4,2}^{-2}&
 		p_{5,2}^{-2} & p_{6,2}^{-2} & p_{7,2}^{-2} & p_{8,2}^{-2}
 	\\
 	p_{1,2}^2 &  p_{2,2}^2 &  p_{3,2}^2 &  p_{4,2}^2&
 	 	p_{5,2}^2 &  p_{6,2}^2 &  p_{7,2}^2 &  p_{8,2}^2
 	\\
 	p_{1,1}p_{1,2}^{-1} &   p_{2,1}p_{2,2}^{-1} &   p_{3,1}p_{3,2}^{-1}&   p_{4,1}p_{4,2}^{-1}&
 		p_{5,1}p_{5,2}^{-1} &   p_{6,1}p_{6,2}^{-1} &   p_{7,1}p_{7,2}^{-1}&   p_{8,1}p_{8,2}^{-1}
 	\\
 	p_{1,1}p_{1,2} &   p_{2,1}p_{2,2} &   p_{3,1}p_{3,2} &   p_{4,1}p_{4,2}&
 	 p_{5,1}p_{5,2} &	p_{6,1}p_{6,2} &   p_{7,1}p_{7,2} &   p_{8,1}p_{8,2}
 	\\
 	p_{1,1}^2 & p_{2,1}^2 & p_{3,1}^2 & p_{4,1}^2&
 		p_{5,1}^2 & p_{6,1}^2 & p_{7,1}^2 & p_{8,1}^2
 	\end{matrix}}\in \C^{8\times 8},
     \end{align*}
     and corresponding Laurent--Vandermonde matrix was explicitly above.
   Let us illustrate the argument in the proof of Theorem \ref{theorem:nice}. We have that all nodes are solutions of $L=0$ that is
 $ 5= -(z_1^{-1}+z_2^{-1}+z_2+z_1)$.
    Moreover, if we consider the ideal of $L$ and  write down  the corresponding relevant equations  at this moment, i.e., $z_1L(z_1,z_2)=0$, $z_2L(z_1,z_2)=0$, $z_1^{-1}L(z_1,z_2)=0$ and $z_1^{-1}L(z_1,z_2)=0$ we get the following relations
    \begin{align*}
    \ccases{
    z_1^{-1}L(z_1,z_2)=0,\\z_1^{-1}L(z_1,z_2)=0,\\	z_2L(z_1,z_2)=0,\\z_1L(z_1,z_2)=0,
    	}\Rightarrow
    \ccases{
    	1+5z_1^{-1}=-(z_1^{-2}+z_1^{-1}z_2+z_1^{-1}z_2^{-1}),\\
    	1+5z_2^{-1}=-(z_1z_2^{-1}+z_1^{-1}z_2^{-1}+z_2^{-2}),\\
    	 1+5z_2=-(z_1z_2+z_1^{-1}z_2+z_2^2),\\
    1+5z_1=-(z_1^2+z_1z_2+z_1z_2^{-1}).
    }
    \end{align*}
    From these observations we see that the matrix $\mathcal V^1_2$ is structured in 3 bands, being the last one our $\Sigma_2^1$, that  as the nodes belong to the algebraic hypersurface of $L$,  give us 5 linear relations among the rows. That is, the rank is at much 8. From previous results we know that if the set of nodes is  poised then the rank is 8.  But observe that there are no linear dependence relations among the rows in $\Sigma_2^1$, all the 5 relations involve rows in the previous bands, not in the last one $\Sigma_2^1$.

\subsection{Perturbed Christoffel--Darboux kernels. Connection formul{\ae} }

We discuss here how the Christoffel--Darboux kernels behave under  Darboux transformations, expressing the perturbed kernels in terms of the non-perturbed ones
\begin{definition}
	We introduce the following two truncations of the resolvent matrix.
	The first one is an upper  block triangular matrix
	\begin{align*}
	\omega^{[ l,  m ]}\coloneq\PARENS{
	\begin{matrix}
	\omega_{[ l],[ l]} &  \omega_{[ l],[ l+1]}    & \dots &  \omega_{[  l ],[  l +  m -1]}  \\
	0& \omega_{[  l +1],[  l +1]} & \dots  &  \omega_{[  l +1],[  l +  m ]} \\
	\vdots&\quad\quad\quad \ddots                    &            &  \vdots                            \\
	0& 0                          &             &  \omega_{[  l +  m -1],[  l +  m -1]}
	\end{matrix}}
	\end{align*}
	while the second one is a lower block triangular matrix
	\begin{align*}
	\omega_{[  l ,  m ]}\coloneq\PARENS{
	\begin{matrix}
	\omega_{[  l ],[  l +  m ]}     & 0&\dots & 0                                                 \\
	\omega_{[  l +1],[  l +  m +1]} &  \omega_{[  l +1],[  l +  m +2]} &        &              0                         \\
	\vdots                                                  &        & \ddots&                                 \\
	\omega_{[  l +  m -1],[  l +  m ]}            &	\omega_{[  l +  m -1],[  l +  m +1]}                   &     \dots   &       &  \omega_{[  l +  m -1],[  l +2  m -1]}
	\end{matrix}},
	\end{align*}
	and also the following block diagonal matrix
	\begin{align*}
	H^{[  l ,  m ]}\coloneq\diag(H_{[  l ]},\dots,H_{[  l  +  m -1]})
	\end{align*}
\end{definition}
Notice that when  $\omega^{[  l ,  m ]}$ and $\omega_{[  l ,  m ]}$ are glued together we get a
$  m $ block wide horizontal slice of the resolvent $\omega$.
We will also use $\omega^{[ l],[ l+  m ]}$ which denotes the truncation of the resolvent $\omega$ built up with the first $ l$ block rows and the first $ l+  m $ block columns.

\begin{theorem}[Connection formul{\ae}]
	The following formul{\ae} relating Christoffel--Darboux kernels hold true
	\begin{align*}
	K^{(  l +  m )}(\z_1 ,\z_2  )&=L(\z_2  )TK^{(  l +  m )}(\z_1 ,\z_2  )
	-\PARENS{\begin{matrix}
	T\hat{\phi}_{[  l ]}(\z_1 ) \\
	T\hat{\phi}_{[  l +1]}(\z_1 ) \\
	\vdots \\
	T\hat{\phi}_{[  l +  m -1]}(\z_1 ) \\
	\end{matrix}}^{\dagger} (TH^{[  l +  m ,  m ]})^{-1} \omega_{[  l ,  m ]}\PARENS{\begin{matrix}
	\phi_{[  l +  m ]}(\z_2  ) \\
	\phi_{[  l +  m +1]}(\z_2  ) \\
	\vdots \\
	\phi_{[  l +2  m -1]}(\z_2  )
	\end{matrix}}
	\\
	&=L(\z_2  )TK^{(  l )}(\z_1 ,\z_2  )
	+\PARENS{\begin{matrix}
	T\hat{\phi}_{[  l ]}(\z_1 ) \\
	T\hat{\phi}_{[  l +1]}(\z_1 ) \\
	\vdots \\
	T\hat{\phi}_{[  l +  m -1]}(\z_1 ) \\
	\end{matrix}}^{\dagger} (TH^{[  l ,  m ]})^{-1}\omega^{[  l ,  m ]}
	\PARENS{\begin{matrix}
	\phi_{\mathscr [l]}(\z_2  ) \\
	\phi_{[  l +1]}(\z_2  ) \\
	\vdots \\
	\phi_{[  l +  m -1]}(\z_2  )
	\end{matrix}}\\
	&=\bar{L}(\z_1 ^{-1})TK^{(  l +  m )}(\z_1 ,\z_2  )
	-\PARENS{\begin{matrix}
	\hat{\phi}_{[  l +  m ]}(\z_1 ) \\
	\hat{\phi}_{[  l +  m +1]}(\z_1 ) \\
	\vdots \\
	\hat{\phi}_{[  l +2  m -1]}(\z_1 ) \\
	\end{matrix}}^{\dagger}  (\hat{\omega}_{[  l ,  m ]})^{\dagger}  (TH^{[  l +  m ,m]})^{-1}
	\PARENS{\begin{matrix}
	T\phi_{[  l ]}(\z_2  ) \\
	T\phi_{[  l +1]}(\z_2  ) \\
	\vdots \\
	T\phi_{[  l +  m -1]}(\z_2  )
	\end{matrix}}\\
	&=\bar{L}(\z_1 ^{-1})TK^{(  l )}(\z_1 ,\z_2  )
	+\PARENS{\begin{matrix}
	\hat{\phi}_{[  l ]}(\z_1 ) \\
	\hat{\phi}_{[  l +1]}(\z_1 ) \\
	\vdots \\
	\hat{\phi}_{[  l +  m -1]}(\z_1 ) \\
	\end{matrix}}^{\dagger}  (\hat{\omega}^{[  l ,  m ]})^{\dagger}  (TH^{[  l ,  m ]})^{-1}
	\PARENS{\begin{matrix}
	T\phi_{[  l ]}(\z_2  ) \\
	T\phi_{[  l +1]}(\z_2  ) \\
	\vdots \\
	T\phi_{[  l +  m -1]}(\z_2  )
	\end{matrix}}
	\end{align*}
\end{theorem}

\begin{proof} First we consider the expression
	\begin{align*}
	(T\hat{\phi}^{[  l ,  m ]}(\z_1 ))^{\dagger} \left[ \left(\hat{M}^{\dagger} H^{-1}\right)^{[  l +  m ]} \right] \phi^{[  l ,  m ]} (\z_2  )=
	(T\hat{\phi}^{[  l ,  m ]}(\z_1 ))^{\dagger} \left[\left((TH^{-1}) \omega\right)^{[  l +  m ]}  \right] \phi^{[  l ,  m ]}(\z_2  )
	\end{align*}
	which of course is an identity due to the relations between $\hat{M}$ and $\omega$.
	We  now let the operators act to their right or to their left and
	obtain
	\begin{align*}
	\left[ (T\hat{\phi}(\z_1 ))^{\dagger} \left(\hat{M}^{\dagger} H^{-1}\right)^{[  l +  m ]} \right] \phi (\z_2  )=
	(T\hat{\phi}(\z_1 ))^{\dagger} \left[\left((TH^{-1}) \omega\right)^{[  l +  m ]} \phi(\z_2  ) \right].
	\end{align*}
	The term between brackets on the LHS is
	\begin{align*}
	\left[\left(\hat{\phi}^{[  l +  m ]}(\z_1 )\right)^{\dagger} (H^{-1})^{[  l +  m ]} \right]
	\end{align*}
	While the term between the brackets on the RHS is   more complicated. Observe that
	\begin{align*}
	\omega^{[  l +  m ]}&=
	\PARENS{\begin{array}{c}
		\omega^{[  l ],[  l +  m ]}\\
	\hline
	\begin{array}{c|c}
	0  &   \omega^{[  l ,  m ]}
	\end{array}
	\end{array}},&\omega^{[  l +  m ],[  l +2  m ]}&=
\PARENS{\begin{array}{c|c}\omega^{[  l +  m ]}&
\begin{matrix}
0\\
\hline
\omega_{[ l,  m ]}
\end{matrix}
\end{array}}.
	\end{align*}
	Hence,
	\begin{align*}
	\omega^{[  l +  m ]}\phi(z_2)=&L(z_2)\PARENS{\begin{matrix}
	T\phi_{[0]}(z_2) \\
	\vdots   \\
	T\phi_{[  l -1]}(z_2) \\
	\hline
	0              \\
	\vdots \\
	0
	\end{matrix}}+
	\PARENS{\begin{matrix}
	0    \\
	\vdots \\
	0   \\
	\hline
	 \omega^{[  l ,  m ]} \PARENS{\begin{matrix}\phi_{[  l ]}(z_2) \\ \vdots \\ \phi_{[  l +  m -1]}(z_2)  \end{matrix}}
	\end{matrix}},\\
=&L(z_2)\PARENS{\begin{matrix}
	T\phi_{[0]}(\z_2) \\
	\vdots   \\
	T\phi_{[  l -1]}(z\_2) \\
	\hline
	T\phi_{[  l ]}(\z_2)    \\
	\vdots \\
	T\phi_{[  l +  m -1]}(\z_2)
	\end{matrix}}-
	\PARENS{\begin{matrix}
	0    \\
	\vdots \\
	0   \\
	\hline
	 \omega_{[  l ,  m ]}
	 \PARENS{\begin{matrix}\phi_{[  l +  m ]}(\z_2) \\ \vdots \\ \phi_{[  l +2  m -1]}(\z_2)  \end{matrix}}
	\end{matrix}}.
	\end{align*}
To conclude let us say that the third and fourth equalities
follow similarly from the equality
	\begin{align*}
	(\hat{\phi}(\z_1 ))^{\dagger} \left[ \left(H^{-1} M\right)^{[  l +  m ]} \right] T\phi (\z_2  )=
	(\hat{\phi}(\z_1 ))^{\dagger} \left[\left(\hat{\omega}^{\dagger}(TH^{-1}) \right)^{[  l +  m ]}  \right] T\phi(\z_2  ).
	\end{align*}

\end{proof}

\section{Integrable systems of Toda and Kadomtsev--Petviasvilii type}

The idea here is to introduce deformations of the moment matrix or, equivalently, the measure. These deformations will depend on some continuous and discrete parameters.
For those cases where the initial moment matrix is Hermitian (and therefore the  measure is real) and/or definite positive we give conditions under which deformations preserve these properties.

\subsection{Discrete Toda type flows}
Let us consider the composition of two Darboux transformations, for that aim we collect what we know of each single Darboux transformation as follows.
Given two Laurent polynomials $L_1(\z)$ and $L_2(\z) $ we consider the corresponding Darboux transformed measures
\begin{align*}
T_i\d\mu(\boldsymbol{\theta})&=L_i(\Exp{\ii\boldsymbol{\theta}})\d\mu(\boldsymbol\theta), & i&=1,2,
\end{align*}
for which we have the resolvent and adjoint resolvent matrices
\begin{align*}
\omega_i&=(T_iS)L_i(\bUpsilon)S^{-1}, &  \hat\omega_i&=(T_i\hat S)L_i(\bUpsilon)\hat S^{-1},\\
M_i&=ST_iS^{-1}, & \hat M_i&=\hat ST_i\hat S^{-1},
\end{align*}
connected by
\begin{align*}
H\hat\omega_i(T_iH)^{-1}&= M_i,& (T_iH)^{-1}\omega_iH&=\hat\omega_i,
\end{align*}
and  such that the associated MVOLPUT satisfy
\begin{align*}
\begin{aligned}
\omega_i \Phi(\z)& =L_i(\z)T_i\Phi(\z), &  \hat\omega_i \hat \Phi(\z)& =\bar L_i(\z^{-1})T_i\hat\Phi(\z),\\
M_iT_i\Phi(\z)&=\Phi(\z), & M_iT_i\hat\Phi(\z)&=\hat\Phi(\z).
\end{aligned}
\label{composing}
\end{align*}
The composed Darboux transformation corresponds to the perturbed measure
\begin{align*}
T_1T_2\d\mu(\boldsymbol{\theta})\equiv T_2T_1\d\mu(\boldsymbol{\theta})=
L_1(\Exp{\ii\boldsymbol{\theta}})L_2(\Exp{\ii\boldsymbol{\theta}})\d\mu(\boldsymbol{\theta}).
\end{align*}

Observe that if we want to apply the sample matrix trick to find the resolvent we must require $L_1$ and $L_2$ to be nice, and as was proven in Proposition \ref{pro:prod-nice} the product $L_1L_2$ is nice again. Therefore, the iteration of two nice Darboux transformations is a nice Darboux transformation.
The interplay of these two Darboux transformations can be described by
\begin{pro}
	The Jacobi matrices $J_a$, $a=1,\dots, D$, fulfill the discrete Lax type equations
	\begin{align*}
	(T_iJ_a)\omega_i &=\omega_i J_a, & M_iT_iJ_a&=J_a\omega_i,
	\end{align*}
	for $i=1,2$ and $a=1,\dots, D$.
	Moreover, the following Zakharov--Shabat compatibility equation is satisfied
	\begin{align}
	(T_1\omega_2)\omega_1&=(T_2\omega_1)\omega_2, &
	M_1T_1M_2&=M_2T_2M_1.
	\end{align}
	The hatted versions of these equations also hold.
\end{pro}
\begin{proof}
	To prove the first set of equations just replace $J_a=S\Upsilon_aS^{-1}$ and the expressions for $\omega_i$ and $M_i$ given in \eqref{composing}. As for the second one,  we perform the very same replacement to get
	\begin{align*}
	(T_1\omega_2)\omega_1&=(T_1T_2S)L_2(\bUpsilon)L_1(\bUpsilon)S^{-1}, &
	M_1T_1M_2&= SL_2(\bUpsilon)L_1(\bUpsilon)((T_1T_2S)^{-1},
	\end{align*}
	from where the conclusion follows immediately from the symmetry in the interchange of subscripts.
\end{proof}
Therefore, $T_1$ and $T_2$ could be understood as two compatible discrete flows introducing in this way two directions for the associated discrete shifts. We can generate a lattice of discrete flows by considering $d$, in principle not related with $D$, possible directions or family of transformations or perturbations. For each $i\in\{1,\dots,d\}$ we need an infinite family of multivariate Laurent polynomials $\{L_{i,n_i}(\z)\}_{n_i\in \Z}$ and consider a measure that depends on the multi-index $\m=(m_1,\dots,m_d)\in\Z_+^d$ as follows
\begin{align*}
\d\mu_{\m+\boldsymbol{e}_i}(\boldsymbol{\theta})\equiv
T_i\d\mu_{\m}(\boldsymbol{\theta})\coloneq L_{i,m_i+1}(\Exp{\ii\boldsymbol{\theta}})\d\mu_{\m}(\boldsymbol{\theta}).
\end{align*}
A simplified version of this lattice of measures can be gotten with
\begin{align}\label{eq:discrete-flows-laurent}
\d\mu_{\m}(\boldsymbol{\theta })\coloneq \prod_{i=1}^d\big( L_i(\Exp{\ii\boldsymbol{\theta}})\big)^{m_i}\d\mu(\boldsymbol{\theta}).
\end{align}
Notice that the introduction of $d$ different directions is, in principle, arbitrary, and that we coud even collapse all to one direction $d=1$ having the very same measure. This is connected with  the bijection existing between  $\Z_+^d$ and $\Z_+$. From the viewpoint of the  Gauss--Borel factorization the moment matrix
for \eqref{eq:discrete-flows-laurent} is
\begin{align*}
G(\m)= \prod_{i=1}^d\big( L_i(\boldsymbol{\Upsilon})\big)^{m_i}G=G \prod_{i=1}^d\big( L_i(\boldsymbol{\Upsilon})\big)^{m_i}.
\end{align*}

We now consider discrete flows associated with Darboux transformations generated with degree one Laurent polynomials. Our discrete deformations needs .
\begin{definition}
	Let $\n_a=(n_{a,-1},\dots, n_{a,-D},n_{a,D},\dots,n_{a,1})^\top\in\mathbb{C}^{2D}$ be	$2D$ linearly independent vectors   and let $q_a$, $a\in\{1,\dots,2D\}$ and $\m\in\Z_+^{2D}$ be 
	complex numbers  and integer multi-indices, respectively.   Then, the evolved measure is
	\begin{align*}
	\d \mu _{\m}(\z)\coloneq&L(\m,\z)\d \mu(\z), & L(\m,\z)\coloneq \prod_{a=1}^{2D}(L_{\n_a}(\z)-q_a )^{m_a},
	\end{align*}
	 the vacuum wave matrix  is 
	\begin{align*}
	W_0(\m)\coloneq\prod_{a=1}^{2D}(\centernot\n_a-q_a )^{m_a}
	\end{align*}
	and the deformed moment matrix is $G(\m)=W_0(\m) G =G W_0(\m)$.
\end{definition}

\begin{pro}\label{real}
	For degree one discrete flows  reality is ensured if
	\begin{align*}
	q_a&\in\R, &\boldsymbol{n}_a&=\hat \n_a, &\forall a\in\Z_D.
	\end{align*}
	This choice also ensures the nicety of the Laurent polynomial $L(\m,\z)$.
\end{pro}
Equivalently, 
\begin{align*}
\boldsymbol{n}_a&=\hat \n_a	=\PARENS{\begin{matrix}
	\bar{\boldsymbol{u}}_a\\\mathcal E_D\boldsymbol{u}_a
	\end{matrix}}, & \boldsymbol  u_a&\in\C^D.
\end{align*}

We see that for discrete transformations we have an one step real reduction, however  we could also consider as in  \cite{MOLPUC} two step reductions, in that case we request for \begin{align*}
\mathcal E_{2D}	\bar{\boldsymbol{n}}_a=\boldsymbol{n}_{D+1-a}&=\PARENS{\begin{matrix}\bar{\boldsymbol{u}}_a\\\mathcal E_D\bar{\boldsymbol{u}}_a\end{matrix}}, &
q_{D+1-a}&=\bar q_a, &
m_{D+1-a}&=m_a,
\end{align*}
which naturally appears when we  ask for $(L_{\n_a}(\z)-q_a )^{m_i} \cdot
(L_{\n_{D+1-a}}(\z)-q_{D+1-a} )^{m_{D+1-a}}$ to be real; involving therefore  two products of $W_0$. Now the action of only one  $T_a$ does not preserve the  Hermitian character,  but  the composition  of translations $T_a \cdot T_{D+1-a}$ does.

The definite positiveness for the case considered in Proposition \ref{real}  is achieved if
\begin{align*}
2\sum_{b=1}^D|u_{a,b}|\cos(\theta_b+\arg u_{a,b})&\geqslant q_a, & a\in\{1,\dots,2D\},
\end{align*}
which is fulfilled whenever
\begin{align*}
q_a+2\sum_{b=1}^D|u_{a,b}|&\leqslant 0.
\end{align*}	

Notice that
\begin{align}\label{TG}
T_aG=({\centernot \n}_a-q_a)G=G({\centernot \n}_a-q_a).
\end{align}
with corresponding degree one resolvents given by
\begin{align*}
\omega_a&\coloneq(T_aS) ({\centernot \n}_a-q_a) S^{-1}, &
\hat \omega_a^\dagger&\coloneq
\big(\hat{S}^{-1} \big)^{\dagger}({\centernot \n}_a-q_a)(T_a \hat{S} )^{\dagger}.
\end{align*}
For these adjoint resolvents we have
\begin{pro}\label{pro:resolvent}
	The following expressions for the  degree one resolvents hold
	\begin{align*}
	\omega_a&=(T_aH)
	\big(\hat S 	T_a\hat S^{-1}\big)^\dagger H^{-1}
	\\
	&=\PARENS{\begin{matrix}
		(T_aH _{[0]}) H_{[0]}^{-1}    & ({\centernot \n}_a)_{[0],[1]} &      0      & 0       &               &            &  \\
		0& (T_aH _{[1]}) H_{[1]}^{-1}    & ({\centernot \n}_a)_{[1],[2]} &    0    &               &            &  \\
		0&  0           & (T_aH _{[2]}) H_{[2]}^{-1}    &  ({\centernot \n}_a)_{[2],[3]}     &               &            &  \\
		0&0&0&(T_aH _{[3]}) H_{[3]}^{-1} &\\
		&             &            && \ddots &               &            & 
		\end{matrix}},\\
	\hat \omega_a^\dagger&=
	H^{-1} S (T_aS )^{-1} T_aH \\
	&=\PARENS{\begin{matrix}
		H_{[0]}^{-1} T_aH _{[0]}    &  0& 0           & 0       &               &            &  \\
		({\centernot \n}_a)_{[1],[0]} & H_{[1]}^{-1}T_aH _{[1]}     &0  &  0      &               &            &  \\
		0	& ({\centernot \n}_a)_{[2][,1]} & H_{[2]}^{-1}T_aH _{[2]}     &      0&               &            &  \\
		0&0&({\centernot \n}_a)_{[3],[2]} & H_{[3]}^{-1}T_aH _{[3]}  \\
		&             &            & &\ddots            &            &  \\
		\end{matrix}}.	\end{align*}
\end{pro}
\begin{proof}
	It is a direct consequence of \eqref{TG}, the first relations of $T_aG=({{\centernot \n}}_a-q_a)G$ and the second of $T_aG=G({{\centernot \n}}_a-q_a)$.
\end{proof}
It also implies the following relations
\begin{align*}
(T_aH _{[0]}) H_{[0]}^{-1}&=-({{\centernot \n}}_a)_{[0],[1]}\beta_{[1]}-q_a, \\
(T_aH_{[k]})  H_{[k]}^{-1}&= (T_a \beta_{[k]}   )({{\centernot \n}}_a)_{[k-1],[k]}-({{\centernot \n}}_a)_{[k],[k+1]}\beta_{[k+1]}-q_a,\\
(T_aH_{[k]}^{-1})({{\centernot \n}}_a)_{[k],[k+1]}H_{[k+1]}&=- \Delta_a\hat\beta_{[k+1]}^{\dagger} ,\\
H_{[0]}^{-1} (T_aH )_{[0]} &=-\hat\beta_{[1]}^{\dagger}({{\centernot \n}}_a)_{[1],[0]}-q_a, \\
H_{[k]}^{-1}T_aH_{[k]} &= ({{\centernot \n}}_a)_{[k],[k-1]}T_a\hat\beta _{[k]}^\dagger-\hat\beta_{[k+1]}^\dagger({{\centernot \n}}_a)_{[k+1],[k]}-q_a,\\
H_{[k]}({{\centernot \n}}_a)_{[k],[k-1]}T_aH_{[k-1]}^{-1} &=- \Delta_a\beta_{[k]},
\end{align*}
from where we derive
\begin{cor}
	Discrete Toda type equations are fulfilled
	\begin{multline*}
	\Delta_b\big((T_aH_{[k]})  H_{[k]}^{-1}
	\big)=({\centernot \n}_a)_{[k],[k+1]}H_{[k+1]}({\centernot \n}_b)_{[k+1],[k]}T_bH^{-1}_{[k]}-(T_a H_{[k]})({\centernot \n}_b)_{[k],[k-1]}(T_aT_bH^{-1}_{[k-1]})({\centernot \n}_a)_{[k-1],[k]},
	\end{multline*}
	or equivalently
	\begin{multline*}
	\Delta_b
	\big(H_{[k]}^{-1}
	T_aH_{[k]}\big)=
	(T_bH^{-1}_{[k]})({\centernot \n}_b)_{[k],[k+1]}H_{[k+1]}({\centernot \n}_a)_{[k+1],[k]}-({\centernot \n}_a)_{[k],[k-1]}(T_aT_bH_{[k-1]}^{-1})({\centernot \n}_b)_{[k-1],[k]}T_aH_{[k]},
	\end{multline*}
	for any couple of indexes $a,b\in\{1,\dots,D\}$.
\end{cor}

\subsection{Miwa type expressions for the MVOLPUT}

We introduce  the matrices
\begin{align*}
\alpha_{a,[k]}&\coloneq(T_aH_{[k]} )H^{-1}_{[k]}, &
\hat\alpha^\dagger_{a,[k]}&\coloneq H^{-1}_{[k]}T_aH_{[k]}.
\end{align*}
In terms of which
the nonzero  degree one resolvent block entries are
\begin{align*}
(\omega_a)_{[k],[k]}&=\alpha_{a,[k]}, &
(\hat \omega_a)_{[k],[k]}&=\hat \alpha_{a,[k]}
\end{align*}
on the main diagonal, and $({\centernot \n}_a)_+$ for $\omega_a$ and  $({\centernot {\hat \n}}_a)_+$ for $\hat \omega_a$,  on the first superdiagonal.
The adjoint  resolvent matrices are defined by
\begin{align*}
\hat{M}_a ^{\dagger}&\coloneq T_aH^{-1}\omega_a H=\big(\hat S(T_a\hat S)^{-1}\big)^\dagger, &	M_a&:=H \hat \omega_a^{\dagger} T_aH ^{-1}=S(T_aS)^{-1} ,
\end{align*}
and consequently can be written as follows
\begin{align*}
M_a&=\I+H({\centernot \n}_a)_-T_aH^{-1}, &
\hat{M}_a^\dagger&=\I+(T_aH)^{-1}({\centernot \n}_a)_+H.
\end{align*}
Hence, these block unitriangular matrices have only two block diagonals different from zero, the main diagonal which equals the identity and the first sub-diagonal with
\begin{align*}
(M_a)_{[k+1],[k]}&=\rho_{a,[k+1]}, & (\hat M_a)_{[k+1],[k]}&=\hat\rho_{a,[k+1]},
\end{align*}
where
\begin{align*}
\hat\rho^{\dagger}_{a,[k+1]}&\coloneq(T_aH_{[k]}^{-1} )({\centernot \n}_a)_{[k],[k+1]}H_{[k+1]},\\
\rho_{a,[k+1]}&:=H_{[k+1]}({\centernot \n}_a)_{[k+1],[k]}T_aH_{[k]}^{-1}.
\end{align*}

As we know, we have for the transformed and non-transformed  MVOLPUT  the following connection formul\ae
\begin{align*}
\omega_a \Phi(\z)&= (L_{\n_a}(\z)-q_a) T_a\Phi (\z),&
\hat \omega_a\hat\Phi(\z)&=
(L_{\hat \n_a}(\z)
-\bar q_a )
T_a\hat\Phi(\z),
\\
\hat M_aT_a \hat\Phi(\z)&=\hat\Phi(\z),&
M_a T_a\Phi(\z) &=\Phi(\z).
\end{align*}

Now we need two technical lemmata
\begin{lemma}
	The inequality $2D|[k]|\geqslant |[k+1]|$ holds true for all  $k=0,1,\dots$.
\end{lemma}
\begin{proof}
	Given a multi-index $\q=(\alpha_1,\dots,\alpha_D)\in[k]$ such that its $i$-th component satisfies $\alpha_i\gtrless 0$, then $\q\pm\boldsymbol{e_i}\in [k+1]$.
	Therefore, for each $\q\in[k]$ and each of its $D$ components we can get one ($\alpha_i\neq 0$) or two ($\alpha_i= 0$)
	multi-indices in $[k+1]$ out of the initial $\q\in[k]$.
	Hence, it is clear that given any multi-index $\q\in[k]$ we can get $n$ multi-indices in $[k+1]$ with $D \leqslant n\leqslant 2D$.
	Now, we observe that any $\q'\in[k+1]$ can be obtained by this procedure form at least one multi-index in $[k]$. Then $n|[k]|\geqslant|[k+1]|$and consequently  $2D|[k]|\geqslant|[k+1]|$.
\end{proof}
\begin{lemma}\label{lemma:fullrank}
	The matrix
	\begin{align*}
	\big(\chi_{[1]}(\boldsymbol{\Upsilon})\big)_{[k],[k+1]}\coloneq\PARENS{\begin{matrix}
		(\Upsilon_1^{-1})_{[k],[k+1]} \\
		(\Upsilon_2^{-1})_{[k],[k+1]} \\
		\vdots \\
		(\Upsilon_D^{-1})_{[k],[k+1]} \\
		(\Upsilon_D)_{[k],[k+1]} \\
		\vdots \\
		(\Upsilon_1)_{[k],[k+1]}
		\end{matrix} }\in \R^{2D|[k]|\times |[k+1]|}
	\end{align*}
	has full column rank  for all $ k=0,1\dots$.
\end{lemma}
\begin{proof}
	See Appendix \ref{proof8}
\end{proof}
\begin{definition} \label{def:Nqz}
	For the degree one discrete flows we consider
	\begin{enumerate}
		\item The matrix
		\begin{align*}
		N&\coloneq\PARENS{\begin{matrix}
			\n_1^\top \\ \vdots \\ \n_{2D}^\top
			\end{matrix}}
		\in\C^{2D\times 2D}, 
		\end{align*}
		which we assume to be invertible (as a consecuence $\hat N = \bar N\mathcal E_{2D}$ is also be invertible).
		\item For any  vector $\z\in\C^D$ the associated vector
		\begin{align*}
		\boldsymbol q&\coloneq N \chi_{[1]}(\z)\in\C^{2D}
		\end{align*}
		\item	We consider  two  rectangular matrices
		\begin{align*}
		{\centernot N}_{k}&\coloneq 
		\PARENS{\begin{matrix}
			({\centernot \n}_1 )_{[k],[k+1]} \\  \vdots \\	({\centernot \n}_{2D} )_{[k],[k+1]}
			\end{matrix}} \in \C^{2D|[k]|\times |[k+1]|}, \\
		[\boldsymbol{T} H]_k&:= \PARENS{\begin{matrix}
			T_1 H_{[k]} \\ \vdots \\   T_{2D}H_{[k]}
			\end{matrix}}\in \C^{2D|[k]|\times |[k]|}.
		\end{align*}
	\end{enumerate}
\end{definition}
We are ready for the following result expressing the MVOLPUT  as products of the quasi-tau matrices $H$  and its discrete flows
\begin{theorem}\label{theorem:tauMVOUT}
	The MVOLPUT can be expressed in terms of  quasi-tau matrices $H$  and its degree one discrete flows as follows
	\begin{align*}
	\phi_{[k]}(\z)=&(-1)^k
	{\centernot N}_{k-1}^{+}[\boldsymbol {T} H]_{k-1} H_{[k-1]}^{-1}{\centernot N}_{k-2}^{+}
	[\boldsymbol {T}H]_{k-2} H_{[k-2]}^{-1}\cdots{\centernot N}_{0}^{+}[\boldsymbol {T} H]_{0} H_{[0]}^{-1},\\
	\hat	\phi_{[k]}(\bar{\z}^{-1})=&(-1)^k
	{\centernot {\hat{N}}}_{k-1}^{+}[\boldsymbol {T} H]_{k-1} H_{[k-1]}^{-1}	{\centernot {\hat{N}}}_{k-2}^{+}
	[\boldsymbol {T}H]_{k-2} H_{[k-2]}^{-1}\cdots	{\centernot {\hat{N}}}_{0}^{+}[\boldsymbol {T} H]_{0} H_{[0]}^{-1}.
	\end{align*}
\end{theorem}
\begin{proof}
	Since $N$ is invertible the matrix ${\centernot N}_{k}$ still has full column rank. Hence,
	the correlation matrix ${\centernot N}_{k}^\top{\centernot N}_{k}\in\C^{|[k+1]|\times |[k+1]|}$ is invertible
	and the  left inverse of ${\centernot N}_{k}$ is
	\begin{align*}
{	\centernot N}_{k}^{+}\coloneq \big({\centernot N}_{k}^\top{\centernot N}_{k}\big)^{-1}{\centernot N}_{k}^\top.
	\end{align*}
	
	All the previous said, and as long as the two previously stated conditions hold we have
	\begin{align*}
	\alpha_{[k]}\phi_{[k]}(\z)+(\n_a\cdot \chi_{[1]}(\bUpsilon))_{[k],[k+1]}\phi_{[k+1]}(\z)&=0 ,
	\end{align*}
	whenever the transformation parameters $\{\n_a,q_a\}_{a=1}^{2D}$ are as described in Definition \ref{def:Nqz} and therefore
	%
	\begin{align*}
	{\centernot N}_{k-1}\phi_{[k]}(\boldsymbol z)&=
	-[\boldsymbol {T}H]_{k-1} H_{[k-1]}^{-1} \phi_{[k-1]}(\boldsymbol z)
	\end{align*}
	But, since ${\centernot N}_{k-1}$ has full column rank $|[k]|$ with
	left inverse ${\centernot N}_{k-1}^{+}$ we get
	\begin{align*}
	\phi_{[k]}(\boldsymbol z)=
	-{\centernot N}_{(k-1)}^+[\boldsymbol {T}H]_{k-1} H_{[k-1]}^{-1} \phi_{[k-1]}(\boldsymbol z).
	\end{align*}
	It is easy to see that a similar result can be obtained for $\hat{\phi}$ when evaluated  $\bar{\z}^{-1}$.
	Iteration leads to the desired result.
\end{proof}

\subsection{A  Toda type integrable hierarchy  }
For  the continuous deformation we introduce a covector  $t:=(t_{[0]},t_{[1]},\dots)$ where
$t_{[k]}=\big(t_{\q^{(k)}_1},\dots, t_{\q^{(k)}_{|[k]|}}\big)$ are complex covectors as well.

\begin{definition}
	The vacuum wave matrix  is
	\begin{align*}
	W_0(t)&\coloneq\Exp{t(\z)}, & t(\z)&\coloneq\sum_{\q\in\Z^D} t_\q \Upsilon_\q.
	\end{align*}
\end{definition}
Then
\begin{pro}
	The matrix $G(t)=W_0(t) G =G W_0(t)$ is the moment matrix of the deformed measure
	\begin{align*}
	\d \mu _{t}=\Exp{t(\z)}
	\d \mu, &
	\end{align*}
	and the Fourier series of the evolved measure is
	\begin{align*}
	\hat \mu_t=\Exp{t(\z)}\hat{\mu}.
	\end{align*}
\end{pro}

\begin{definition}
	We say that $t$ is an admissible set of times if $G(t)$ admits a block Cholesky factorization as in Proposition \ref{pro:gauss}.
	\end{definition}
Hereon  we assume that we are dealing with admissible times.

For the continuous flows we need $\exp(t(\z))$ to be real which is achieved when $t(\z)\in\R$, this last requirement holds for $t$
such that $\bar t=t\eta$; then, we request $\bar t_{[k]}=t_{[k]}\mathcal E_{|[k]|} $ or, equivalently,
\begin{pro}
	The continuous flows preserve the reality and definite positiveness whenever
	$t_{-\boldsymbol{\alpha}}=\bar t_{\boldsymbol{\alpha}}$. The real deformed measure in $\T^D$ has the following form
	\begin{align*}
	\d\mu_{t}(\boldsymbol\theta)=\Exp{2\sum_{k=0}^\infty \sum_{j=1}^{|[k]|/2} t_{\boldsymbol{\alpha}^{(k)}_j} \cos(
		\boldsymbol{\alpha}^{(k)}_j
		\cdot\boldsymbol{\theta})}\d\mu(\boldsymbol{\theta}).
	\end{align*}
\end{pro}

From the  Gauss--Borel factorization of the perturbed measure
\begin{align}\label{eq:borel.evol}
W_0(t)G=GW_0(t)=G(t)&= (S(t))^{-1}H(t)\big((\hat
S(t))^{-1}\big)^\dagger
\end{align}
we get $t$-dependent MVOLPUT.
\begin{definition}\label{def:wave}
	The wave matrices are defined as follows
	\begin{align*}
	W_1&:= S(t ) W_0(t), &  \hat{W}_1^{\dagger}& := W_0(t)\big( \hat{S}(t) \big)^{\dagger},\\
	W_2^{\dagger}&:= \big(S(t)\big)^{-1}H(t), &  \hat W_2&:= H(t) \Big(\big( \hat S (t)\big)^{-1} \Big)^{\dagger}.
	\end{align*}
\end{definition}

\begin{pro}
	We have
	\begin{align}\label{eq:GWW}
	G=W_1^{-1}\hat{W_2}= W_2^{\dagger} \left(\hat{W}^{-1}_1\right)^{\dagger}.
	\end{align}
\end{pro}
\begin{proof}
	Use the  Gauss--Borel factorization of
	$ G({t})=W_0({t})G=G W_0({t})$.
\end{proof}

The splitting as a direct sum of the linear  semi-infinite matrices in strictly block lower triangular matrices and upper block triangular
matrices is denoted by $M=M_<+M_\geq$, and the associated splitting in strictly upper triangular and lower triangular parts is denoted
by $M=M_\leq+M_>$.
\begin{lemma}\label{lemma:GD}
	We have the Gel'fand--Dickey expressions
	\begin{align*}
	\frac{\partial S}{\partial t_{\q}}S^{-1}+S\bUpsilon^\q S^{-1}&=\Big(S\bUpsilon^\q S^{-1}\Big)_\geq,&
	\frac{\partial S}{\partial t_{\q}}S^{-1}&=-\Big(S\bUpsilon^\q S^{-1}\Big)_<,\\
	\frac{\partial \hat S}{\partial \bar{t}_{-\q}}\hat S^{-1}+\hat S\bUpsilon^\q \hat S^{-1}&=\Big(\hat S\bUpsilon^{\q} \hat S^{-1}\Big)_\geq,&
	\frac{\partial \hat S}{\partial \bar{t}_{-\q}}\hat S^{-1}&=-\Big(\hat S\bUpsilon^{\q}\hat S^{-1}\Big)_<.
	\end{align*}
\end{lemma}
\begin{proof}
	See Appendix \ref{proof9}.
\end{proof}
Observe that   for Hermitian cases  the first and second lines in the previous lemma coincide ---due to the equality between  hatted and non hatted terms  and
$t_{\q}=\bar{t}_{-\q}$.

\begin{definition}
	We introduce the Zakharov--Shabat matrices
	\begin{align*}
	B_\q:= & (J_\q(t))_\geq, & 	\hat B_\q:= (\hat J_\q(t))_\geq.
	\end{align*}
\end{definition}
The next important result    collects the essential elements of the associated integrable theory
\begin{pro}\label{pro:integrable elements}
	\begin{enumerate}
		\item The wave matrices solve the linear systems
		\begin{align}\label{eq:linear}
		\frac{\partial W_1}{\partial t_{\q}}=&B_\q W_1,&   \frac{\partial \hat{W}_1}{\partial \bar{t}_{-\q}}=&\hat{B}_{\q}\hat{W}_1\\
		\frac{\partial \hat{W}_2}{\partial t_{\q}}=&B_{\q}\hat{W}_2, & \frac{\partial W_2}{\partial \bar{t}_{-\q}}=&\hat{B}_\q W_2
		\end{align}
		\item The Jacobi matrices are \emph{Lax} matrices; i.e,   they fulfill  the following Lax equations
		\begin{align*}
		\frac{\partial J_\q}{\partial t_{\q'}}=&[B_{\q'},J_\q], &
		\frac{\partial \hat J_\q}{\partial \bar{t}_{-\q'}}=&[\hat B_{\q'},\hat J_\q].
		\end{align*}
		\item The Zakharov--Shabat (or zero-curvature) equations are satisfied
		\begin{align*}
		\frac{\partial B_\q}{\partial t_{\q'}}-\frac{\partial B_{\q'}}{\partial t_{\q}}+\big[B_\q,B_{\q'}\big]=&0, &
		\frac{\partial \hat B_{\q}}{\partial \bar{t}_{-\q'}}-\frac{\partial \hat B_{\q'}}{\partial \bar{t}_{-\q}}+\big[\hat B_{\q},\hat B_{\q'}\big]=&0.
		\end{align*}
	\end{enumerate}
\end{pro}
\begin{proof}
	See Appendix \ref{proof10}.
\end{proof}
Notice that only the equations on the left or the right would be need in the Hermitian case.

We now introduce the Baker functions and their adjoints
\begin{definition}\label{def:baker}
	The Baker and adjoint Baker wave functions are
	\begin{align*}
	\Psi_1&:= W_1 \chi,   &  \Psi_1^*&:=\left(W_1^{-1}\right)^{\dagger} \chi,  &
	\hat{\Psi}_1^{\dagger}&:= \chi^{\dagger} \hat{W}_1^{\dagger},   &  \left(\hat{\Psi}_1^{*}\right)^{\dagger}&:= \chi^{\dagger} \hat{W}_1^{-1}, \\
	\hat{\Psi}_2&:= \hat{W}_2 \chi,   &  \hat{\Psi}_2^*&:= \left(\hat{W}_2^{-1}\right)^{\dagger} \chi,  &
	\Psi_2^{\dagger}&:= \chi^{\dagger} W_2^{\dagger},   &  \left(\Psi_2^{*}\right)^{\dagger}&:= \chi^{\dagger} W_2^{-1}.
	\end{align*}
\end{definition}

\begin{pro}\label{pro:baker}
	The Baker functions and adjoint Baker functions in Definition \ref{def:baker} can be expressed in terms of the MVOLPUT and the Fourier transform of the measure as follows
	\begin{align*}
	\Psi_1&=\Exp{t(\z)}\Phi(\z,t), & \z&\in(\C^*)^D,&  \Psi_1^*&=(2\pi)^D \bar{\hat{\mu}}( \z^{-1}) \Big(\big(H(t)\big)^{-1}\Big)^\dagger \hat\Phi(\z,t), & \z^{-1}&\in\mathcal D_\mu,\\
	\hat{\Psi}_1&=\Exp{\bar t(\z^{-1})}
	\hat\Phi(\z,t),&  \z&\in(\C^*)^D,   &  \hat{\Psi}_1^{*}&= (2\pi)^D\hat\mu(\z)\big(H(t)\big)^{-1}\Phi(\z,t), & \z&\in\mathcal D_\mu, \\
	\hat{\Psi}_2&=(2\pi)^D \hat\mu(\z)\Exp{t(\z)}\Phi(\z,t), & \z&\in\mathcal D_\mu,&  \hat{\Psi}_2^*&=\Big(\big(H(t)\big)^{-1}\Big)^\dagger \hat\Phi(\z,t), &  \z&\in(\C^*)^D,  \\
	\Psi_2&= (2\pi)^D\bar{\hat{\mu}}(\z^{-1})\Exp{\bar t(\z^{-1})}\hat\Phi(\z,t),&  \z^{-1}&\in\mathcal D_\mu,&  \Psi_2^{*}&= \big(H(t)\big)^{-1}\Phi(\z,t), & \z&\in(\C^*)^D.
	\end{align*}
\end{pro}
\begin{proof}
	We use Definitions  \ref{def:second}, \ref{def:wave} and \ref{def:baker} together with Proposition \ref{pro:secondMVOLPUT}
\end{proof}
\begin{pro}
	The Baker functions satisfy the following linear differential equations
	\begin{align*}
	\frac{\partial \Psi_1}{\partial t_{\q}}=&B_\q \Psi_1,&   \frac{\partial \hat{\Psi}_1}{\partial \bar{t}_{-\q}}=&\hat{B}_{\q}\hat{\Psi}_1\\
	\frac{\partial \hat{\Psi}_2}{\partial t_{\q}}=&B_\q \hat{\Psi}_2,&   \frac{\partial \Psi_2}{\partial \bar{t}_{-\q}}=&\hat{B}_{\q}\Psi_2
	\end{align*}
	while the adjoint Baker functions satisfy
	\begin{align*}
	\frac{\partial \hat{\Psi}_1^*}{\partial t_{\q}}=&-\left(\hat B_{-\q}\right)^{\dagger} \hat{\Psi}_1^*,&   \frac{\partial \Psi_1^*}{\partial \bar{t}_{-\q}}=&-\left(B_{-\q}\right)^{\dagger}\Psi_1^*,\\
	\frac{\partial \Psi_2^*}{\partial t_{\q}}=&-\left(\hat B_{-\q}\right)^{\dagger} \Psi_2^*,&   \frac{\partial \hat{\Psi}_2^*}{\partial \bar{t}_{-\q}}=&-\left(B_{-\q}\right)^{\dagger}\hat{\Psi}_2^*.
	\end{align*}
\end{pro}
\begin{proof}
	Taking derivatives on their definitions and using Lemma \ref{lemma:GD}  the result is straightforward.	
\end{proof}

We will be specially interested in the $t_{[1]}$ part of the covector of the continuous deformation parameters
$$t_{[1]}=(t_{-1},\dots,t_{-D},t_{D},\dots,t_{1}).$$
\begin{pro}
	The following differential equations are satisfied by the square matrices $H_{[k]}$ and the rectangular matrices $\beta_{[k]}$.
	\begin{align*}
	\frac{\partial \beta_{[k]}}{\partial t_{a}}&=-H_{[k]} (\Upsilon_{a})_{[k],[k-1]}H_{[k-1]}^{-1}   &
	\frac{\partial \hat{\beta}_{[k]}}{\partial \bar{t}_{-a}}&=-H_{[k]}^{\dagger} (\Upsilon_{a})_{[k],[k-1]}\left(H_{[k-1]}^{-1}\right)^{\dagger}\\
	\frac{\partial H_{[k]}}{\partial t_{a}}H_{[k]}^{-1}&=\beta_{[k]}(\Upsilon_{a})_{[k-1],[k]}-(\Upsilon_{a})_{[k],[k-1]}\beta_{[k+1]}&
	\frac{\partial H_{[0]}}{\partial t_{a}}H_{[0]}^{-1}&=-(\Upsilon_{a})_{[0]}{[1]}\beta_{[1]}
	\end{align*}
\end{pro}
\begin{proof}
	These relations are just identifications if we use in Lemma \ref{lemma:GD} the explicit formul{\ae} \eqref{eq:jacobi}.
\end{proof}
Taking second partial derivatives in the previous expressions leads to the following
\begin{theorem}
	The $H_{[k]}$ matrices are subject to the following Toda lattice type equations
	\begin{align}
	\frac{\partial }{\partial t_{b}}\left( \frac{\partial H_{[k]}}{\partial t_{a}}H_{[k]}^{-1}\right)&=
	(\Upsilon_{a})_{[k],[k+1]}H_{[k]}(\Upsilon_{b})_{[k+1],[k]}H_{[k]}^{-1}-
	H_{[k]}(\Upsilon_{b})_{[k],[k-1]}H_{[k-1]}^{-1}(\Upsilon_{a})_{[k-1],[k]}
	\end{align}
	which can be rewritten for the rectangular matrices $\beta_{[k]}$ as
	\begin{align*}
	\frac{\partial^2 \beta_{[k]}}{\partial t_a \partial t_b}&=
	\frac{\partial }{\partial t_{a}}\left(\beta_{[k]}(\Upsilon_{b})_{[k-1],[k]}\beta_{[k]} \right)-
	\frac{\partial \beta_{[k]}}{\partial t_{a}}\beta_{[k-1]}(\Upsilon_{b})_{[k-2],[k-1]}-
	(\Upsilon_{b})_{[k],[k+1]}\beta_{[k+1]}\frac{\partial \beta_{[k]}}{\partial t_{a}}
	\end{align*}
	and for the $\hat{\beta}_{[k]}$ as
	\begin{align*}
	\frac{\partial^2 \hat{\beta}_{[k]}}{\partial \bar{t}_{-a} \partial \bar{t}_{-b}}&=
	\frac{\partial }{\partial \bar{t}_{-a}}\left(\hat{\beta}_{[k]}(\Upsilon_{b})_{[k-1],[k]}\hat{\beta}_{[k]} \right)-
	\frac{\partial \hat{\beta}_{[k]}}{\partial \bar{t}_{-a}}\hat{\beta}_{[k-1]}(\Upsilon_{b})_{[k-2],[k-1]}-
	(\Upsilon_{b})_{[k],[k+1]}\hat{\beta}_{[k+1]}\frac{\partial \hat{\beta}_{[k]}}{\partial \bar{t}_{-a}}.
	\end{align*}
\end{theorem}

\subsection{Bilinear equations}
We begin with the following observation
\begin{pro}\label{t,t'}
	Wave matrices evaluated at different admissible times $t,t'$ are related as it follows
	{\small\begin{align*}
	W_{1}(t)\left( W_1(t') \right)^{-1}&=\hat{W}_2(t) \left(\hat{W}_2(t')\right)^{-1},&
	\hat{W}_1(t) \left(\hat{W}_1(t')\right)^{-1}&=W_{2}(t)\left(W_2(t') \right)^{-1}, &
	W_{1}(t)\left(W_2(t') \right)^{\dagger}&=\hat{W}_2(t) \left(\hat{W}_1(t')\right)^{\dagger}.
	\end{align*}}
\end{pro}
\begin{proof}
	Just remember \ref{eq:GWW} wich is valid when evaluated at any time, therefore
	\begin{align*}
	G=\left(W_1(t)\right)^{-1}\hat{W}_2(t)= \left(W_2(t)\right)^{\dagger} \left([\hat{W}_1(t)]^{-1}\right)^{\dagger}=
	\left(W_1(t')\right)^{-1}\hat{W}_2(t')= \left(W_2(t')\right)^{\dagger} \left([\hat{W}_1(t')]^{-1}\right)^{\dagger}
	\end{align*}
	from where each equality follows easily.
\end{proof}
Notice that  the first two relations would be equivalent in the Hermitian case.
\begin{lemma}\label{lemma:identity}
For any polyradius $\boldsymbol r\in\R_>^D$ we have
	\begin{align*}
	\mathbb{I}&
	=\frac{1}{\big( 2 \pi\operatorname{i}\big)^{D}}\oint_{\T^D(\boldsymbol r)} \frac{\chi (\z) \cdot \left(\chi(\z^{-1})\right)^{\top}}{\prod_{a=1}^{D}z_a} \d^D\z
	=
	\frac{1}{\big( 2 \pi\operatorname{i}\big)^{D}}\oint_{\T^D(\boldsymbol r)} \frac{\chi (\z^{-1}) \cdot \left(\chi(\z)\right)^{\top}}{\prod_{a=1}^{D}z_a} \d^D\z,
	\end{align*}
	here $\d\z\coloneq\d z_1 \d z_2 \dots \d z_D$ denotes de Lebesgue measure in $\C^D$.
\end{lemma}
\begin{theorem}
	Let
$\boldsymbol c, \hat{\boldsymbol c}\in\mathcal D_\mu$,
$i\in\{1,2\}$
 be two points in the Reinhardt domain of convergence of the Laurent series $\hat\mu(\z)$ of the measure $\d\mu$ and  let $\boldsymbol r_{\boldsymbol c},\boldsymbol r_{\hat{\boldsymbol c}}\in\R_>^D$ be the corresponding
 poly-radii. Then, the Baker functions satisfy the following bilinear equations for every couple of admissible times $(t,t')$
	\begin{align*}
	\oint_{\T^D(\boldsymbol r_{\boldsymbol c})}\frac{\Psi_1(\z,t)\big(\Psi_1^*(\bar \z^{-1},t') \big)^{\dagger}}{\prod_{a=1}^{D}z_a} \d^D\z&=
\oint_{\T^D(\boldsymbol r_{\hat{\boldsymbol c}}^{-1})}\frac{\hat{\Psi}_2(\z^{-1},t)\big(\hat{\Psi}_2^*(\bar \z,t') \big)^{\dagger}}{\prod_{a=1}^{D}z_a}\d^D\z, \\
\oint_{\T^D(\boldsymbol r_{\boldsymbol c})}\frac{\Psi_2(\z^{-1},t)\big(\Psi_2^*(\bar \z,t') \big)^{\dagger}}{\prod_{a=1}^{D}z_a}\d^D\z&=\
\oint_{\T^D(\boldsymbol r_{\hat{\boldsymbol c}}^{-1})}\frac{\hat{\Psi}_1(\z,t)\big(\hat{\Psi}_1^*(\bar \z^{-1},t') \big)^{\dagger}}{\prod_{a=1}^{D}z_a}\d^D\z,
\\
	\oint_{\T^D(\boldsymbol r_{\boldsymbol c})}\frac{\Psi_1(\z,t)\big(\Psi_2(\bar \z^{-1},t') \big)^{\dagger}}{\prod_{a=1}^{D}z_a} \d^D\z &=
\oint_{\T^D(\boldsymbol r_{\hat{\boldsymbol c}}^{-1})}\frac{\hat{\Psi}_2(\z^{-1},t)\big(\hat{\Psi}_1(\bar \z,t') \big)^{\dagger}}{\prod_{a=1}^{D}z_a}  \d^D\z, \\
\oint_{\T^D(\boldsymbol r_{\boldsymbol c})}\frac{\Psi_2^*(\z,t)\big(\Psi_1^*(\bar \z^{-1},t') \big)^{\dagger}}{\prod_{a=1}^{D}z_a}   \d^D\z &=
\oint_{\T^D(\boldsymbol r_{\hat{\boldsymbol c}}^{-1})}\frac{\hat{\Psi}_1^*(\z^{-1},t)\big(\hat{\Psi}_2^*(\bar \z,t') \big)^{\dagger}}{\prod_{a=1}^{D}z_a}   \d^D\z.
	\end{align*}
The MVOLPUT evaluated at different times satisfy
	{\begin{align*}
	\oint_{\T^D(\boldsymbol r_{\boldsymbol c})}  \Phi(\z,t)\big(\hat{\Phi}(\bar \z^{-1},t')\big)^{\dagger}\frac{\Exp{t(\z)}\hat{\mu}(\z)}{\prod_{a=1}^{D}z_a}
 \d^D \z &=
\oint_{\T^D(\boldsymbol r_{\hat{\boldsymbol c}}^{-1})}  \Phi(\z^{-1},t)\big(\hat{\Phi}(\bar \z,t')\big)^{\dagger}\frac{\Exp{t(\z^{-1})}\hat{\mu}(\z^{-1})}{\prod_{a=1}^{D}z_a} \d^D\z,\\
	\oint_{\T^D(\boldsymbol r_{\boldsymbol c})}  \hat{\Phi}(\z^{-1},t)\big(\Phi(\bar \z,t')\big)^{\dagger}
	\frac{\Exp{\bar{t}(\z)} \bar{\hat{\mu}}(\z) }{\prod_{a=1}^{D}z_a} \d^D\z &=
	\oint_{\T^D(\boldsymbol r_{\hat{\boldsymbol c}}^{-1})} \hat{\Phi}(\z,t)\big(\Phi(\bar \z^{-1},t')\big)^{\dagger}
	\frac{\Exp{\bar{t}(\z^{-1})} \bar{\hat{\mu}}(\z^{-1}) }{\prod_{a=1}^{D}z_a} \d^D\z, \\
	\oint_{\T^D(\boldsymbol r_{\boldsymbol c})}   \Phi(\z,t)\big(\hat{\Phi}(\bar \z^{-1},t')\big)^{\dagger}
	\frac{\Exp{(t+t')(\z)} \hat{\mu}(\z) }{\prod_{a=1}^{D}z_a} \d^D\z &=
	\oint_{\T^D(\boldsymbol r_{\hat{\boldsymbol c}}^{-1})} \Phi(\z^{-1},t)\big(\hat{\Phi}(\bar \z,t')\big)^{\dagger}
	\frac{\Exp{(t+t')(\z^{-1})} \hat{\mu}(\z^{-1}) }{\prod_{a=1}^{D}z_a}\d^D\z, \\
	\oint_{\T^D(\boldsymbol r_{\boldsymbol c})}   \Phi(\z,t)\big(\hat{\Phi}(\bar \z^{-1},t')\big)^{\dagger}
	\frac{\hat{\mu}(\z) }{\prod_{a=1}^{D}z_a}\d^D\z &=
	\oint_{\T^D(\boldsymbol r_{\hat{\boldsymbol c}}^{-1})} \Phi(\z^{-1},t)\big(\hat{\Phi}(\bar \z,t')\big)^{\dagger}
	\frac{\hat{\mu}(\z^{-1}) }{\prod_{a=1}^{D}z_a}\d^D\z.
	\end{align*}}
\end{theorem}
\begin{proof}
 For the first set of equations  involving Baker functions just insert the identity according to Lemma \ref{lemma:identity} matrix in each equality in \ref{t,t'} and remember  Definition
\ref{def:baker}. For the second set of equations  involving MVOLPUT  use Proposition \ref{pro:baker}.
\end{proof}

\subsection{Miwa shifts and vertex operators}

For $D=1$, see \cite{carlos},  Miwa shifts generate discrete flows from continuous flows, there you have coherent shifts of the form $t\to t+[w]$ or $t_j\to t_j+w^j/j$, $j\in\Z_>$, where we request $w\in\D,z\in\bar{\D}$ or $w\in \C\setminus \bar{\D}, z\in \C\setminus {\D}$; recall that when considered as a perturbation of the measure we evaluate $z\in\T$. This suggests  to consider for each orthant different Miwa shifts.
We need to use a slightly  modified version of Definitions \ref{orthant} and \ref{def:cauchy kernel} and also introduce orthant coherent Miwa shifts
\begin{definition}
		For each subset $\sigma\in 2^{\Z_D}$we introduce
			\begin{enumerate}
		\item		The corresponding   complete integer orthants are given by
	\begin{align*}
	(\Z^D)^\sigma\coloneq &\bigtimes_{i=1}^D Z^i, &
	Z^i\coloneq&\ccases{
		\Z_-, & i\in \sigma,\\
		\Z_+, & i\in\complement \sigma,
	}
	\end{align*}

	\item For any $\boldsymbol w\in (\D^D)_\sigma$ the corresponding orthant Miwa coherent shift $[\boldsymbol w]_\sigma$ has as its entries
	\begin{align*}
	([\boldsymbol w]_\sigma)_\q\coloneq\ccases{
		0, & \q\in\Z^D\setminus(\Z^D)^\sigma,\\
		\frac{\boldsymbol w^\q}{|\q|}, & \q\in (\Z^D)^\sigma.
	}
	\end{align*}
			\end{enumerate}
\end{definition}
Observe that $(\Z^D)^\sigma=\Z^D\cap\overline{(\R^D)_\sigma}$
\begin{lemma}
	For each   $\sigma\subset \Z_D$ and $\boldsymbol w\in (\D^D)_\sigma$  we have
\begin{align*}
-\sum_{\q\in\Z^D}([\boldsymbol w]_\sigma)_\q\z^\q=\log\big(1-\sum_{i\in\sigma}(w_iz_i)^{-1}-\sum_{i\in\complement \sigma } w_iz_i\big)
\end{align*}
uniformly for  $\z\in \overline{ (\D^D)_\sigma}$.
\end{lemma}
\begin{proof}
In the one hand we have
	\begin{align}\label{eq:log}
	-\sum_{\q\in\Z^D}([\boldsymbol w]_\sigma)_\q\z^\q=-\sum_{\q\in(\Z^D)^\sigma}\frac{\boldsymbol w^\q\z^\q}{|\q|}
	\end{align}
	and on the other the expansion
	\begin{align*}
	\log\big(1-\sum_{i\in\sigma}(w_iz_i)^{-1}-\sum_{i\in\complement \sigma } w_iz_i\big)=-\sum_{k=1}^\infty\frac{1}{k}
	\Big(\sum_{i\in\sigma}(w_iz_i)^{-1}+\sum_{i\in\complement \sigma } w_iz_i\Big)^k
	\end{align*}
	converges uniformly for  $\z\in \overline{ (\D^D)_\sigma}$ and coincides with the RHS of \eqref{eq:log}.
\end{proof}

\begin{pro}
	For each $\sigma\in 2^{\Z_D}$ and each $\boldsymbol w\in(\D^D)_\sigma$ the orthant Miwa coherent shifts in the continuos flows
	                \begin{align*}
	 t\mapsto t -[\boldsymbol w]_\sigma
	                \end{align*}
induce, in terms of the Laurent polynomial $L_{\boldsymbol w}\coloneq 1-\sum\limits_{i\in\sigma}w_i^{-1}z_i^{-1}-\sum\limits_{i\in\complement \sigma } w_iz_i$ of longitude $\ell(L)=1$
 the following Darboux perturbation of the measure
	\begin{align*}
\d\mu_t(\boldsymbol\theta) \rightarrow L_{\boldsymbol w}(\Exp{\operatorname{i}\boldsymbol{\theta}})	\d\mu_t(\boldsymbol\theta).
	\end{align*}
\end{pro}
\begin{proof}
	From the previous Lemma we have
		\begin{align*}
		\exp(t(\z)-[\boldsymbol w]_\sigma(\z))\d\mu(\z) =\big(1-\sum_{i\in\sigma}(w_iz_i)^{-1}-\sum_{i\in\complement \sigma } w_iz_i\big)\exp(t(\z))\d\mu(\z)
		\end{align*}
		and the result follows.
\end{proof}

This Miwa shift can be written as a vertex type operator $\Gamma_{\boldsymbol w}^\sigma$ for each orthant $(\Z^D)^\sigma$ and each vector $\boldsymbol w\in(\D^D)_\sigma$
\begin{align*}
\Gamma_{\boldsymbol w}^\sigma\coloneq \exp\Big(-\sum_{\q\in (\Z^D)^\sigma}\frac{\boldsymbol w^\q}{|\q|}\frac{\partial}{\partial t_\q}\Big).
\end{align*}
Observe that these vertex operators generate discrete flows or Darboux transformations. Consequently, they are useful in results like the Miwa expressions of the MVOLPUT, see Theorem \ref{theorem:tauMVOUT}
\subsection{Asymptotic modules}
We will proceed  to study families of nonlinear partial differential-difference equations
involving a single fixed site, say the $k$-th position, in the lattice and therefore not mixing several sites as   the Toda type equations do mix, involving near neighbors $k-1$, $k$ and $k+1$. .

Since we are interested in partial differential-difference equations in this section we have to consider the full deformation matrix,
this is, the  product of both the continuous and discrete parts that have been introduced in previous sections.

\begin{definition} For the perturbation parameters $t=\{t_\q\}_{\q\in\Z^D}$, $t_\q\in\C$ and $\m\in\Z^{2D}$ we introduce the following \emph{complete perturbation} 
	\begin{align*}
	W_0(t,\m)\coloneq\exp\Big(\sum_{\q\in\Z^D} t_{\q} \Upsilon_{{\q}}\Big)
	\prod_{a=1}^{2D}\left({\centernot \n}_a-q_a \right)^{m_a}.
	\end{align*}
\end{definition}

Observe that
\begin{align*}
(W_0(t,\m))^\dagger=\exp(
\sum_{\q\in\Z^D} 
\bar t_{-\q} 
\Upsilon_{\q}
)
\prod_{a=1}^{2D}({\centernot{\hat{\n}}}_a-\bar{q}_a )^{m_a}
\end{align*}
and
\begin{align*}
\frac{\partial W_0^{\dagger}}{\partial \bar{t}_{a}}&=\left(\frac{\partial W_0}{\partial t_{a}}\right)^{\dagger} &
T_{a}W_0^\dagger&=\left( T_a W_0\right)^{\dagger}.
\end{align*}

\begin{definition}\label{def:asymptotic-module}
	Given two semi-infinite matrices $R_1(t,\m)$ and $R_2(t,\m)$ we say that
	\begin{itemize}
		\item  $R_1(t)\in\mathfrak{l}W_0$ if $R_1(t)\big(W_0(t,\m)\big)^{-1}$ is a block strictly lower triangular matrix.
		\item  $R_1(t)\in(\mathfrak{l}W_0)^{\dagger}$ if $R_1(t)\big((W_0(t,\m))^{\dagger}\big)^{-1}$ is a block strictly lower triangular matrix.
		\item  $R_2(t)\in\mathfrak{u}$ if it is a block upper triangular matrix.
	\end{itemize}
	The sets $\mathfrak{l}W_0$ and $\mathfrak u$ are known as \emph{asymptotic modules}.
\end{definition}

Then, we can state the following \emph{congruences} \cite{mm} or \emph{asymptotic module} \cite{jaulent} style results
\begin{pro}[Asymptotic modules]\label{pro:asymptotic-module}
	\begin{enumerate}
	\item Given two semi-infinite matrices $R_1(t,\m)$ and $R_2(t,\m)$ such that
	\begin{itemize}
		\item  $R_1(t,\m)\in\mathfrak{l}W_0(t,\m)$,
		\item  $R_2(t,\m)\in\mathfrak{u}$,
		\item $R_1(t,\m)G=R_2(t,\m)$.
	\end{itemize}
	Then, $R_1(t,\m)=0$ and $R_2(t,\m)=0$.
\item Given two semi-infinite matrices $R_1(t,\m)$ and $R_2(t,\m)$ such that
	\begin{itemize}
		\item  $R_1(t,\m)\in(\mathfrak{l}W_0(t,\m))^{\dagger}$,
		\item  $R_2(t,\m)\in\mathfrak{u}$,
		\item $R_1(t,\m)G^{\dagger}=R_2(t,\m)$.
	\end{itemize}
	Then, $	R_1(t,\m)=0$ and $R_2(t,\m)=0$.
	\end{enumerate}
\end{pro}
\begin{proof}
	Observe that
	\begin{gather*}
	R_2(t,\m)=R_1(t,\m)G=R_1(t,\m)\big(W_1(t,\m)\big)^{-1}W_1(t,\m)G
	=R_1(t,\m)\big(W_1(t,\m)\big)^{-1}\hat{W}_2(t,\m),\\
	R_2(t,\m)\Big(H(t,\m)\big((\hat{S}(t,\m))^{-1}\big)^\dagger\Big)^{-1}= R_1(t,\m)\big(W_0(t,\m)\big)^{-1}\big(S(t,\m)\big)^{-1},
	\end{gather*}
	and also that
	\begin{gather*}
	R_2(t,\m)=R_1(t,\m)G^{\dagger}=R_1(t,\m)\big(\hat{W}_1(t,\m)\big)^{-1}\hat{W}_1(t,\m)G^\dagger=R_1(t,\m)\big(\hat{W}_1(t,\m)\big)^{-1}W_2(t,\m),\\
	R_2(t,\m)\Big(H(t,\m)^{-1}S(t,\m)\Big)^{\dagger}=R_1(t,\m)\big(W_0(t,\m)^{\dagger}\big)^{-1}\big(\hat{S}(t,\m)\big)^{-1},
	\end{gather*}
	and, as in every RHS we have a strictly lower triangular matrix while in every LHS we have an upper triangular matrix,
	both sides must vanish and the result follows.
\end{proof}
We use the congruence notation
\begin{definition}
	When $A-B\in\mathfrak{l}W_0$ $\left(\in(\mathfrak{l}W_0)^{\dagger}\right)$
	we write $A=B+\mathfrak{l}W_0$ $\left(A=B+(\mathfrak{l}W_0)^\dagger \right)$ and if $A-B\in\mathfrak{u}$ we write $A=B+\mathfrak{u}$.
\end{definition}
\subsection{KP flows}

We now study, using the asymptotic module technique, the first and second order KP flows and their compatibility conditions resulting in nonlinear partial differential-difference equations for the coefficients of the MVOLPUT. These are onsite equations not mixing different values of $k$. Finally, we consider the third order flows.

\begin{definition}
	\begin{itemize}
	\item  
	The longitude one perturbation parameters are
	\begin{align*}
	t_{[1]}\coloneq & (t_{-1},\dots,t_{-D},t_{D},\dots,t_{1}), 
	\end{align*}
	where 
	\begin{align*}
	t_a&\coloneq t_{\operatorname{sgn}(a)\ee_{|a|}}, & a&\in\{1,\dots,D\}.
%
	\end{align*}
	\item We also consider the following derivatives
	\begin{align*}
	\partial_{a}&\coloneq\frac{\partial }{\partial t_a} &
	\bar{\partial}_{a}&\coloneq\frac{\partial }{\partial \bar{t}_a} , &
	\partial_{\n_a}&\coloneq\sum_{b=-D}^D {n}_{a,b}  \partial_b, &
	\bar \partial_{\n_a}&\coloneq\sum_{b=-D}^D {n}_{a,b}  \bar{\partial}_b,
	\end{align*}
	where $\n_a\in\C^{2D}$, $a\in\{\pm 1,\dots,\pm D\}$.
\end{itemize}
\end{definition}
\begin{pro}\label{pro:wave}
	The following relations hold true
\begin{align*}
\partial_{\n_a}W_1=&\big[(\mathbb{I}+\beta)({\centernot\n}_a)_{\geq}\big]W_0+\mathfrak{l}W_0,&
T_a W_1=&\big[(\mathbb{I}+T_a\beta)({\centernot \n}_a)_{\geq}-q_a\big]W_0+\mathfrak{l}W_0,\\
\bar{\partial}_{\bar{\n}_a}\hat{W}_1=&\big[(\mathbb{I}+\hat{\beta})({\centernot{\hat{\n}}}_a)_{\geq}\big]W_0^{\dagger}+(\mathfrak{l}W_0)^\dagger,&
T_a \hat{W}_1=&\big[(\mathbb{I}+T_a\hat{\beta})
({\centernot{\hat{\n}}}_a)_{\geq}-\bar{q}_a\big]W_0^{\dagger}+(\mathfrak{l}W_0)^{\dagger}.
\end{align*}
Coinciding both equations  in the Hermitian case.
\end{pro}
\begin{proof}
	Firstly, we realize that
	\begin{align*}
	\partial_{\n_a}W_0&={\centernot\n}_a W_0, &
	\bar{\partial}_{\bar{\n}_a }W_0^{\dagger}&={\centernot{\hat{\n}}}_aW_0^{\dagger}.
	\end{align*}
	
	Secondly, we remember the definitions of $W_1$ and $\hat{W}_1$ and work there using the congruence techniques
	\begin{align*}
	\partial_{\n_a}W_1=&[\partial_{\n_a} S+ S{\centernot\n}_a]W_0\\=&\big[(\mathbb{I}+\beta)({\centernot\n}_a)_{\geq}\big]W_0+\mathfrak{l}W_0,\\
	T_aW_1=&(T_a S)({\centernot\n}_a-q_a)W_0\\
	=&\big[(\mathbb{I}+T_a\beta)({\centernot\n}_a)_{\geq}-q_a\big]W_0+\mathfrak{l}W_0,\\
	\bar{\partial}_{\bar{\n}_a}\hat{W}_1= & [\bar{\partial}_{\bar{\n}_a}\hat{S}+\hat{S}{\centernot{\hat{\n}}}_a]W_0^{\dagger}\\
	= & \big[(\mathbb{I}+\hat{\beta})({\centernot{\hat{\n}}}_a)_{\geq}\big]W_0^{\dagger}+(\mathfrak{l}W_0)^\dagger,\\
	T_a \hat{W}_1= & (T_a \hat{S})({\centernot{\hat{\n}}}_a-\bar{q}_a \big)W_0^\dagger \\
	= & \big[(\mathbb{I}+T_a\hat{\beta})({\centernot{\hat{\n}}}_a)_{\geq}-\bar{q}_a\big]W_0^{\dagger}+(\mathfrak{l}W_0)^{\dagger}.
	\end{align*}
\end{proof}

\begin{pro}
	Both  Baker functions $\{\Psi_1,\hat{\Psi}_2\}$ are solutions of the following partial difference-differential linear systems
	\begin{align*}
	\big(\partial_{\n_a}-T_a\big)\Psi=[q_a-\Delta_a \beta ({\centernot\n}_a )_{\geq}]\Psi.
	\end{align*}
	While both Baker functions  $\{\hat{\Psi}_1,\Psi_2\}$ do solve
	\begin{align*}
	\big(\bar{\partial}_{\bar{\n}_a}-T_a\big)\Psi=[\bar{q}_a-\Delta_a \hat{\beta}({\centernot{\hat{\n}}}_a)_{\geq}]\Psi.
	\end{align*}
\end{pro}
Only one set of these equations is relevant in the Hermitian case.

\begin{proof}
	The results in Proposition \ref{pro:wave} can be rewritten as it follows
	\begin{align*}
	\big(\partial_{\n_a}-T_a-[q_a-\Delta_a \beta ({\centernot\n}_a )_{\geq}]\big)W_1 &\in \mathfrak{l}W_0, &
	\big(\bar{\partial}_{\bar{\n}_a}-T_a-[\bar{q}_a-\Delta_a \hat{\beta}({\centernot{\hat{\n}}}_a)_{\geq}]\big)\hat{W}_1 &\in (\mathfrak{l}W_0)^{\dagger}.
	\end{align*}
	Then, it is straightforward to see that
	\begin{align*}
	\big(\partial_{\n_a}-T_a-[q_a-\Delta_a \beta ({\centernot\n}_a )_{\geq}]\big)\hat{W}_2 &\in \mathfrak{u}, &
	\big(\bar{\partial}_{\bar{\n}_a}-T_a-[\bar{q}_a-\Delta_a \hat{\beta}({\centernot{\hat{\n}}}_a)_{\geq}]\big)W_2 &\in \mathfrak{u}.
	\end{align*}
Finally, just recall  the definitions of the Baker functions and apply Proposition \ref{pro:asymptotic-module}.
\end{proof}
The compatibility of the linear systems satisfied by the Baker functions imply
\begin{theorem}\label{teo3}
	The following nonlinear discrete equations for  the first subdiagonal matrices $\beta$ and $\hat{\beta}$ hold
	\begin{align}\label{difference-differential-beta}
	\Delta_b\left[\partial_{\n_a} \beta +(\Delta_a \beta)\big(q_a+({\centernot\n}_a )_{\geq} \beta \big) \right]({\centernot\n}_b )_{\geq}&=
	\Delta_a\left[\partial_{\n_b} \beta+(\Delta_b \beta)\big(q_b+({\centernot\n}_b)_{\geq} \beta \big) \right]({\centernot\n}_a )_{\geq},\\
	\Delta_b\left[\bar{\partial}_{\bar{\n}_a} \hat{\beta} +(\Delta_a \hat\beta)\big(\bar{q}_a+({\centernot{\hat{\n}}}_a)_{\geq} \hat{\beta} \big) \right]({\centernot{\hat{\n}}}_b)_{\geq}&=
	\Delta_a\left[\bar{\partial}_{\bar{\n}_b} \hat{\beta}+(\Delta_b \hat\beta)\big(\bar{q}_b+({\centernot{\hat{\n}}}_b)_{\geq} \hat{\beta} \big) \right]({\centernot{\hat{\n}}}_a)_{\geq}.
	\end{align}
\end{theorem}


\begin{definition}
\begin{enumerate}
 \item Longitude 2 times are couples $(a,b)\subset\{-1,\dots,D,D,\dots,1\}$ where we assume $a+b\neq 0$
\begin{align*}
t_{(a,b)}&\coloneq t_{\operatorname{sgn}(a)\ee_{|a|}+\operatorname{sgn}(b)\ee_{|b|}},
\end{align*}
with corresponding derivatives denoted by
\begin{align*}
\partial_{(a,b)}&:=\frac{\partial }{\partial t_{(a,b)}}.
\end{align*}

\item We introduce the diagonal matrices 
%
%
  \begin{align}\label{eta}
   V_{a,b}:= & \partial_a \beta \left(\Upsilon_b\right)_{>},&
      (V_{a,b})_{[k]}= &\partial_a \beta_{[k]} \left(\Upsilon_b\right)_{[k-1],[k]}, &
      U_{a,b}:= & -V_{a,b}-V_{b,a},\\
   \hat{V}_{a,b}:= & \bar{\partial}_{-a} \hat{\beta} \left(\Upsilon_b\right)_{>},&
      (\hat{V}_{a,b})_{[k]}= &\bar{\partial}_{-a} \hat{\beta}_{[k]} \left(\Upsilon_b\right)_{[k-1],[k]}, &
      \hat{U}_{a,b}:= & -\hat{V}_{a,b}-\hat{V}_{b,a}.     
    \end{align}
\end{enumerate}
\end{definition}
We have excluded the case $a+b=0$ because in that situation what we obtain is $t_{(a,-a)}=t_0$. Observe that 
\begin{align*}
\frac{\partial \Psi_i}{\partial t_0}&=\Psi_i, & \frac{\partial \hat \Psi_i}{\partial t_0}&=\hat \Psi_i,
\end{align*}
with $i=1,2$.

\begin{pro}\label{pro:kp-schrodinger}
  Both Baker functions $\Psi_1$ and $\hat{\Psi}_2$ are solutions of
  \begin{align}\label{eq: linear.wave}
\begin{aligned}
 \frac{\partial \Psi}{\partial t_{(a,b)}}&=\frac{\partial^2\Psi}{\partial t_a\partial t_b}+U_{a,b}\Psi, & a+b\neq0,\\
 \Psi&=\frac{\partial^2\Psi}{\partial t_a\partial t_{-a}}+U_{a,-a}\Psi, 
\end{aligned}
  \end{align}
  while Baker functions $\hat{\Psi}_1$ and $\Psi_2$ solve
  \begin{align}\label{eq: linear.waves}
 \begin{aligned}
 \frac{\partial \Psi}{\partial \bar{t}_{(-a,-b)}}&=\frac{\partial^2\Psi}{\partial \bar{t}_{-a}\partial \bar{t}_{-b}}+\hat{U}_{a,b}\Psi, & a+b\neq 0,\\
 \Psi&=\frac{\partial^2\Psi}{\partial \bar{t}_{-a}\partial \bar{t}_{a}}+\hat{U}_{a,-a}\Psi.
 \end{aligned}  \end{align}
\end{pro}
\begin{proof}
	See Appendix \ref{proof11}.
\end{proof}

  Observe that  for $a=b$ \eqref{eq: linear.wave} reads
  \begin{align*}
 \frac{\partial \Psi_{[k]}}{\partial t^{(2)}_a}&=\frac{\partial^2\Psi_{[k]}}{\partial t_a^2}+(U_a)_{[k]}\Psi_{[k]}, 
 & t^{(2)}_a:= &t_{(a,a)},&
 U_a=&-2V_{a,a}
  \end{align*}
  and \eqref{eq: linear.waves}
  \begin{align*}
 \frac{\partial \Psi_{[k]}}{\partial \bar{t}^{(2)}_{-a}}&=\frac{\partial^2\Psi_{[k]}}{\partial \bar{t}_{-a}^2}+(\hat U_a)_{[k]}\Psi_{[k]}, 
 & \bar{t}^{(2)}_{-a}:= &\bar{t}_{(-a,-a)},&
 \hat U_a=&-2\hat V_{a,a}
  \end{align*}
  which are time dependent one-dimensional Schr\"{o}dinger  type equations for the square matrices $\Psi_{[k]}$, the wave functions, and
  potential the square matrices $(U_a)_{[k]},(\hat{U}_a)_{[k]}$. Moreover, multidimensional matrix Schr\"{o}dinger 
  equations appear if we look to other directions, thus given  $(a_1,\dots,a_d)\subset\{\pm 1,\dots,\pm D\}$,  
  we can look at the second order time flows generated by
    $\frac{\partial}{\partial t}:=\frac{\partial}{\partial t_{(a_1,a_1)}}+\dots+\frac{\partial}{\partial t_{(a_d,a_d)}}$, 
    $\frac{\partial}{\partial \bar{t}}:=\frac{\partial}{\partial \bar{t}_{(-a_1,-a_1)}}+\dots+\frac{\partial}{\partial \bar{t}_{(-a_d,-a_d)}}$
  to get in terms of the following $d$-dimensional two possible nabla operators 
  $\nabla:= (\frac{\partial}{\partial t_{a_1}},\dots,\frac{\partial}{\partial t_{a_d}})^\top$,
  $\bar{\nabla}:= (\frac{\partial}{\partial \bar{t}_{-a_1}},\dots,\frac{\partial}{\partial \bar{t}_{-a_d}})^\top$
  Laplacians 
  $\Delta:=\nabla^2=\frac{\partial^2}{\partial t_{a_1}^2}+\dots+\frac{\partial^2}{\partial t_{a_d}^2}$,
  $\bar{\Delta}:=\bar{\nabla}^2=\frac{\partial^2}{\partial \bar{t}_{-a_1}^2}+\dots+\frac{\partial^2}{\partial \bar{t}_{-a_d}^2}$
  and matrix potentials  
  $U\coloneq U_{a_1,a_1}+\dots+U_{a_d,a_d}=2\nabla(\beta)\cdot\boldsymbol\Upsilon$, 
    \begin{align*}
 \frac{\partial \Psi_{[k]}}{\partial t}&=\Delta \Psi_{[k]}+U_{[k]}\Psi_{[k]} &
 \frac{\partial \Psi_{[k]}}{\partial \bar{t}}&=\bar{\Delta} \Psi_{[k]}+\hat{U}_{[k]}\Psi_{[k]}.
  \end{align*}
  Where the equation on the LHS suits $\Psi_1$ and $\hat{\Psi}_2$ while the one in the RHS suits $\hat{\Psi}_1$ and $\Psi_2$. 

  We see that again only the $k$-th  site of the lattice is involved in these linear equations and, consequently, 
  its compatibility will lead to equations for the coefficients evaluated at that site.
These nonlinear equations for which $\beta_{[k]}$ or $\hat{\beta}_{[k]}$ are a solution are
  \begin{theorem}\label{teo4}
The following nonlinear partial differential equations
    \begin{align}\label{eq:beta}\begin{aligned}
     & \partial_{(c,d)}\big(\partial_a\beta(\Upsilon_b)_{>}+\partial_b\beta(\Upsilon_a)_{>}\big)-
      \partial_{(a,b)}\big(\partial_c\beta(\Upsilon_d)_{>}+\partial_d\beta(\Upsilon_c)_{>}\big)  \\
   &= \partial_a\partial_b\big(\partial_c\beta(\Upsilon_d)_{>}+\partial_d\beta(\Upsilon_c)_{>}\big)-
   \partial_c\partial_d\big(\partial_a\beta(\Upsilon_b)_{>}+\partial_b\beta(\Upsilon_a)_{>}\big)\\
    &+    (\partial_b\partial_c\beta)\big((\Upsilon_d)_{>}\beta(\Upsilon_a)_{>}-(\Upsilon_a)_{>}\beta(\Upsilon_d)_{>}\big)
    +(\partial_b\partial_d\beta)\big((\Upsilon_c)_{>}\beta(\Upsilon_a)_{>}-(\Upsilon_a)_{>}\beta(\Upsilon_c)_{>}\big)
   \\& +(\partial_a\partial_c\beta)\big((\Upsilon_d)_{>}\beta(\Upsilon_b)_{>}-(\Upsilon_b)_{>}\beta(\Upsilon_d)_{>}\big)
    +(\partial_a\partial_d\beta)\big((\Upsilon_c)_{>}\beta(\Upsilon_b)_{>}-(\Upsilon_b)_{>}\beta(\Upsilon_c)_{>}\big)\\
&+ \big[ \partial_a\beta(\Upsilon_b)_{>}+ \partial_b\beta(\Upsilon_a)_{>},\partial_c\beta(\Upsilon_d)_{>}+\partial_d\beta(\Upsilon_c)_{>}\big]
    \end{aligned}
    \end{align}
    \begin{align}\label{eq:betas}\begin{aligned}
     & \bar\partial_{(-c,-d)}\big(\bar\partial_{-a}\hat\beta(\Upsilon_{b})_{>}+\bar\partial_{-b}\hat\beta(\Upsilon_a)_{>}\big)-
      \bar\partial_{(-a,-b)}\big(\bar\partial_{-c}\hat\beta(\Upsilon_d)_{>}+\bar\partial_{-d}\hat\beta(\Upsilon_c)_{>}\big)  \\
   &= \bar\partial_{-a}\bar\partial_{-b}\big(\bar\partial_{-c}\hat\beta(\Upsilon_d)_{>}+\bar\partial_{-d}\hat\beta(\Upsilon_c)_{>}\big)-
   \bar\partial_{-c}\bar\partial_{-d}\big(\bar\partial_{-a}\hat\beta(\Upsilon_b)_{>}+\bar\partial_{-b}\hat\beta(\Upsilon_a)_{>}\big)\\
    &+    (\bar\partial_{-b}\bar\partial_{-c}\hat\beta)\big((\Upsilon_d)_{>}\hat\beta(\Upsilon_a)_{>}-(\Upsilon_a)_{>}\hat\beta(\Upsilon_d)_{>}\big)
    +(\bar\partial_{-b}\bar\partial_{-d}\hat\beta)\big((\Upsilon_c)_{>}\hat\beta(\Upsilon_a)_{>}-(\Upsilon_a)_{>}\hat\beta(\Upsilon_c)_{>}\big)
   \\& +(\bar\partial_{-a}\bar\partial_{-c}\hat\beta)\big((\Upsilon_d)_{>}\hat\beta(\Upsilon_b)_{>}-(\Upsilon_b)_{>}\hat\beta(\Upsilon_d)_{>}\big)
    +(\bar\partial_{-a}\bar\partial_{-d}\hat\beta)\big((\Upsilon_c)_{>}\hat\beta(\Upsilon_b)_{>}-(\Upsilon_b)_{>}\hat\beta(\Upsilon_c)_{>}\big)\\
&+ \big[ \bar\partial_{-a}\hat\beta(\Upsilon_b)_{>}+ \bar\partial_{-b}\hat\beta(\Upsilon_a)_{>},\bar\partial_{-c}\hat\beta(\Upsilon_d)_{>}+\bar\partial_{-d}\hat\beta(\Upsilon_c)_{>}\big]
    \end{aligned}
    \end{align}
    are satisfied for $a,b,c,d\in\{\pm 1,\dots,\pm D\}$.
  \end{theorem}
Observe that this equations decouple giving for each $k$  the same equation \eqref{eq:beta},\eqref{eq:betas} up to the replacements 
  $\beta\to\beta_{[k]}$,$\hat\beta\to\hat\beta_{[k]}$ and $\Upsilon_A\to(\Upsilon_A)_{[k-1],[k]}$, $k=1,2,\dots$ and $A=a,b,c,d$.

\begin{proof}
 Let us start with the first equation. If we denote
\begin{align}\label{L2}
L_{a,b}&:= \partial_a\partial_b +U_{a,b}, & a,b&\in\{\pm 1,\dots,\pm D\},
\end{align}
 \eqref{eq: linear.wave} reads $ \partial_{a,b}(W_1)=L_{a,b}(W_1)$.
The compatibility conditions for this linear system are
\begin{align*}
\Big(  \partial_{(a,b)}(L_{c,d})-\partial_{(c,d)}(L_{a,b})+[L_{c,d},L_{a,b}]\Big)(W_1)&=0
\end{align*}
and consequently
    \begin{align}\label{curvature}
R_{a,b,c,d}(W_1)&=0, & a,b,c,d\in\{\pm 1,\dots,\pm D\},
    \end{align}
where
  \begin{multline*}
    R_{a,b,c,d}:= (\partial_bU_{c,d})\partial_a+(\partial_aU_{c,d})\partial_b-
    (\partial_dU_{a,b})\partial_c-(\partial_cU_{a,b})\partial_d-
    \partial_{(a,b)}(U_{c,d})+\partial_{(c,d)}(U_{a,b})\\+[U_{c,d},U_{a,b}]-\partial_c\partial_d U_{a,b}+
\partial_a\partial_bU_{c,d}.
  \end{multline*}
Therefore, $W_1$ satisfies an equation of the form
      \begin{align*}
\big(    \sum_{j\in I} B_j\partial_j+ A\big)(W_1)&=0.
  \end{align*}
  for a given subset of indexes $I\subset
  \{\pm 1,\dots, \pm D\}$.
Now, recalling that 
\begin{align*}
\partial_a W_1=\big((\Upsilon_a)_> +\beta (\Upsilon_a)_> \big) W_0+\mathfrak{l}W_0
\end{align*}
leads to
 \begin{align*}
  0&= \big (\sum_{j\in I} B_j\partial_j+ A\big)(W_1)=
  \Big( \sum_{j\in I} B_j(\Upsilon_j)_>+
  \sum_{j\in I} D B_j\beta(\Upsilon_j)_>+A \Big)W_0+\mathfrak{l}W_0
  \end{align*}
Observe   that $   \sum\limits_{j\in I} B_j\beta(\Upsilon_j)_>+A$ is a diagonal matrix. The sum $\sum\limits_{j\in I}  B_j(\Upsilon_j)_>$ 
is a linear combination of terms like $\partial_a\beta (\Upsilon_a)_>(\Upsilon_b)_>$, which   either are zero everywhere but possibly  on the first superdiagonal
$a+b\neq 0$ or are identically zero for $a+b=0$, since $ (\Upsilon_a)_>(\Upsilon_{-a})_>=0$. Therefore, each of them can be set equal to zero 
and from \eqref{curvature} we find, in the first place, that
   \begin{align*}
      (\partial_bU_{c,d})(\Upsilon_a)_>+(\partial_aU_{c,d})(\Upsilon_b)_>-
    (\partial_dU_{a,b})(\Upsilon_c)_>-(\partial_cU_{a,b})(\Upsilon_d)_>&=0,
    \end{align*}
which is identically satisfied because of \eqref{eta}. In the second place, we get the
following  nonlinear equation
    \begin{multline*}
          (\partial_bU_{c,d})\beta (\Upsilon_a)_>+(\partial_aU_{c,d})\beta (\Upsilon_b)_>-
    (\partial_dU_{a,b})\beta (\Upsilon_c)_>-(\partial_cU_{a,b})\beta (\Upsilon_d)_>+    \partial_{(a,b)}U_{c,d}-
    \partial_{(c,d)}U_{a,b}\\+[U_{c,d},U_{a,b}]-\partial_c\partial_d U_{a,b}+
\partial_a\partial_bU_{c,d}=0,
    \end{multline*}
   and recalling \eqref{eta} we get the desired result.\\
The proof for the second equation is exactly the same with the repacements 
$\partial_a \to \bar \partial_{-a}$,
$\beta \to \hat \beta$,
$U_{a,b} \to \hat{U}_{a,b}$,
$\mathfrak{l}W_0 \to (\mathfrak{l}W_0)^{\dagger}$.
\end{proof}
  

\begin{definition}
\begin{enumerate}
 \item Longitude 3 times are triples $(a,b,c)\subset\{\pm 1,\dots,\pm D\}$ where we assume $a+b\neq 0,a+c\neq 0,b+c\neq 0 $
\begin{align*}
t_{(a,b,c)}&\coloneq t_{\operatorname{sgn}(a)\ee_{|a|}+\operatorname{sgn}(b)\ee_{|b|}+\operatorname{sgn}(c)\ee_{|c|}},
\end{align*}
with corresponding derivatives denoted by
\begin{align*}
\partial_{(a,b,c)}&:=\frac{\partial }{\partial t_{(a,b,c)}}.
\end{align*}

\item We introduce the diagonal matrices 
 \begin{align*}
   V_{a,b,c}\coloneq\diag((V_{a,b,c})_{[0]},(V_{a,b,c})_{[1]},(V_{a,b,c})_{[2]},\dots)
 \end{align*}
   with
   \begin{align*}
     V_{a,b,c}:= &\partial_a\beta [\beta,(\Upsilon_b)_>](\Upsilon_c)_>,\\
      ( V_{a,b,c})_{[k]} 
          =&\partial_a\beta_{[k]}\Big(\beta_{[k-1]}\big(\Upsilon_b\big)_{[k-2],[k-1]}
          -\big(\Upsilon_b\big)_{[k-1],[k]}\beta_{[k]}\Big)\big(\Upsilon_c\big)_{[k-1],[k]}
   \end{align*}
\end{enumerate}
\end{definition}
We have excluded the case $a+b=0,a+c\neq 0, b+c\neq 0$ because in those cases what we get is $t_{(a,b,c)}=t_c,t_b,t_a$, respectively. 
\begin{pro}
Observe that
  \begin{align*}
     V_{a,b,c} 
     =& \partial_a\beta^{(2)}(\Upsilon_b)_>(\Upsilon_c)_{>}
     -\partial_a\beta (\Upsilon_b)_>\beta(\Upsilon_c)_>,\\ 
      ( V_{a,b,c})_{[k]} 
          =&\partial_a\beta^{(2)}_{[k]}
     \big(\Upsilon_b\big)_{[k-2],[k-1]}-\partial_a\beta_{[k]}\big(\Upsilon_b\big)_{[k-1],[k]}\beta_{[k]}\big(\Upsilon_c\big)_{[k-1],[k]}.
   \end{align*}
\end{pro}

  Notice that there is a similar definition  $\hat V_{a,b,c}$ given by 
  the modifications $\beta \to \hat \beta$ and $\partial_a \to \bar \partial_{-a}$; for the sake of simplicity, we will just mention it after every result instead.
  Observe also  that to obtain the second equalities we have used $\big((\partial_a S) S^{-1}\big)_{[k][k-2]}=0$, which follows from Lemma \ref{lemma:GD},  that leads to
  \begin{align*}
 \partial_a \beta^{(2)}_{[k]}&=\partial_a \beta^{(1)}_{[k]} \beta_{[k-1]}, & \beta^{(2)}_{[k]}&\coloneq S_{[k][k-2]}.
  \end{align*}

  We remark  that $( V_{a,b,c})_{[k]}$ depends on $\beta_{[k]}$ and $\beta^{(2)}_{[k]}$ only, 
  coefficients of the MVOLPUT for the second and third higher lenght monomials, $
  \Phi_{[k]}(\z)=
  \chi_{[k]}(\z)+\beta_{[k]}\chi_{[k-1]}(\z)+\beta_{[k]}^{(2)}\chi_{[k-2]}(\z)+\cdots+\beta_{[k]}^{(k)}
  $.
  If we insist in using only the second higher longitude  coefficient and not the third one  
  there is a price that must be paid, two consecutive Laurent polynomials $\Phi_{[k]}$ and $\Phi_{[k-1]}$ will be involved
  ---as we require both of $\beta_{[k]}$ and $\beta_{[k-1]}$.
  Then
  \begin{pro}\label{latriple}
    The Baker functions $\Psi_1$ and $\hat{\Psi}_2$ are both solutions of the third order linear differential equations
    \begin{align*}
      \partial_{(a,b,c)}\Psi=&\partial_a \partial_b \partial_c \Psi
      -(1-\delta_{0,b+c})V_{a,b}\partial_c \Psi 
      -(1-\delta_{0,a+c})V_{b,c}\partial_a \Psi
      -(1-\delta_{0,a+b})V_{c,a}\partial_b \Psi \\
            &-\big(\partial_a V_{b,c}+\partial_c V_{a,b}+\partial_b V_{c,a}+
          (1-\delta_{0,b+c})V_{a,b,c}+
            (1-\delta_{0,a+c})V_{b,c,a}+ 
            (1-\delta_{0,a+b})V_{c,a,b}\big)\Psi.
    \end{align*}
    and a similar equation holds for $\hat{\Psi}_1$ and $\Psi_2$ after the replacements 
    $V \to \hat V$, $\beta \to \hat \beta$ and $\partial_a \to \bar \partial_{-a}$.
  \end{pro}
  \begin{proof}
  	See Appendix \ref{proof12}.
  \end{proof}

 Notice that as we anticipated in the study of the second order flows there is a special feature to be considered with 
 rescpect to what we would expect from the real case of MVOPR \cite{MVOPR}. As one can see, the equation takes different shapes depending on the 
 values of $a,b,c$.
 \begin{itemize}
  \item If $|\boldsymbol{e}_a+\boldsymbol{e}_b+\boldsymbol{e}_c|=3$ then none of the $\delta$'s is present and the result 
  coincides with the one we had for MVOPR.
  \item If $|\boldsymbol{e}_a+\boldsymbol{e}_b+\boldsymbol{e}_c|=1$ then 
  observe that in the LHS of the equation we would be deriving with respect to a first order flow (instead of a third order flow).
  Also at least one, and at most two of the following three equations must hold
  $b+c=0$,  $a+c=0$ and   $a+b=0$. As a result at least one and at most two of the $\delta$'s give a nonzero contribution and two of 
  four of the terms in the equation cancel. 
 \end{itemize}

\begin{appendices}
\section{OLPUC}
\settocdepth{section}
In \cite{carlos}, following a CMV approach, orthogonal Laurent polynomials in the unit circle (OLPUC) where discussed in detail from the same viewpoint as in this paper. Let us investigate how OLPUC are recovered as MVOLPUT when $D=1$ (MVOLPUT$_{D=1}$). For that aim we recall now  some  basic facts about the construction presented in \cite{carlos}.
The moment matrix considered in \cite{carlos} can be expressed as
\begin{align*}
 G&\coloneq \int_{\mathbb{T}} \chi \d \mu \chi^{\dagger}, \\
  \chi(z)&\coloneq \PARENS{\begin{matrix} 1 \\ z^{-1} \\ z \\ z^{-2} \\ z^2 \\ \vdots \end{matrix}},  &
  G^{[k]}\coloneq \PARENS{\begin{matrix}
            G_{0,0} & G_{0,1} & \dots & G_{0,(k-1)} \\
            G_{1,0} & G_{1,1} & \dots & G_{1,(k-1)} \\
            \vdots  & \vdots &        & \vdots \\
            G_{(k-1),0} & G_{(k-1),1} & \dots & G_{(k-1),(k-1)}
           \end{matrix}}.
\end{align*}
Where $\chi$ is the vector of monomials ordered  according to the CMV proposal.
Notice that as we assume here that the measure is real, $\d \mu=\d \bar\mu$, we have that the moment matrix is Hermitian,  $G=G^\dagger$; moreover, we also assume that we are dealing with a positive definite Borel measure, i.e.,  every minor of the
associated moment matrix is positive, this is $\det G^{[k]}> 0 $ for all $k\in\{0,1,2,\dots\}$. For such cases the   Gauss--Borel factorization
of the moment matrix leads to
\begin{align}
 G&\coloneq s^{-1} h (s^{-1})^{\dagger},  \label{gauss-olpuc}\\
 s&\coloneq \PARENS{\begin{matrix}
      1               &     0     & 0                     &                     &               &         & \\
       \bar{\alpha}_1 & 1        &0                      &                     &               &         & \\
            *         & \alpha_2 & 1                    &                     &               &         & \\
            *         &  *       &\ddots                & \ddots              &               &         &  \\
                      &          &                      &                     &               &         & \\
                      &          &                      &\bar{\alpha}_{2k+1}  & 1             &    0      &  \\
                      &          &                      &          *           & \alpha_{2k+2} & 1        &  \\
                      &          &                      &                     &                          &  \\
                      &          &                      &                     &               &\ddots    & \ddots
     \end{matrix}}, &
 h&\coloneq  \diag( h_0, h_1, \dots,  h_{2k+1,} h_{2k+2},\dots)\notag
    \end{align}
    with the OLPUC  given by
    \begin{align*}
     \varphi(z)&\coloneq s \chi(z),&
\varphi^{(2k+1)}&=z^{-(k+1)}+\dots+\bar{\alpha}_{2k+1}z^{k}, &
\varphi^{(2k+2)}&=\alpha_{2k+2}z^{-(k+1)}+\dots+z^{k+1},
\end{align*}
and  corresponding Szeg\H{o} polynomials expressed as follows
\begin{align*}
z^{k+1}\varphi^{(2k+1)}(z)&=P_{2k+1}^*(z), &
z^{k+1}\varphi^{(2k+2)}(z)&=P_{2k+2}(z),
\end{align*}
with reciprocal polynomials given by  $P_{k}^*(z)=z^k \bar{P}_k(z^{-1})$ and Verblunsky coefficients defined as
$\alpha_k\coloneq P_{k}(0)$. In the theory we have also $\rho_k\coloneq \frac{h_{k+1}}{h_k}$.

However, the MVOLPUT$_{D=1}$  is not exactly the one just described for OLPUC  \cite{carlos}. Despite we are dealing with the very same moment matrix $G$, now the Gauss--Borel factorization is a $2\times 2$ block (but for the first left upper corner) factorization:
\begin{align}\label{gauss-mvolput}
G&=S^{-1}H (S^{-1})^{\dagger},
\end{align}
with
\begin{align*}
 S&\coloneq \PARENS{\begin{matrix}
      1           &         0    &  0          &        &               &            &  \\
      \beta_{[1]} & \mathbb{I}_2  & 0           &        &               &            &  \\
            *      & \beta_{[2]} & \mathbb{I}_2 &        &               &            &  \\
                  &             & \ddots     & \ddots &               &            &  \\
                  &             &            &        &               &            &  \\
                  &             &            &        & \beta_{[k+1]} & \mathbb{I}_2 &  \\
                  &             &            &        &               & \ddots     & \ddots
     \end{matrix}},   &
 H&\coloneq \PARENS{\begin{matrix}
      H_{[0]}     &  0           & 0           &        &               &            &  \\
          0        & H_{[1]}     &     0       &        &               &            &  \\
          0        &     0        & H_{[2]}    &        &               &            &  \\
                  &             &            & \ddots &               &            &  \\
                  &             &            &        &               &            &  \\
                  &             &            &        &               & H_{[k+1]} &  \\
                  &             &            &        &               &            & \ddots
     \end{matrix}},
\end{align*}
being all the entries of these matrices, except from $S_{[0][0]}=1, H_{[0]}$ which are real numbers,  $2\times 2$ complex matrices. The Laurent orthogonal polynomials are encoded in the semi-infinite vector
\begin{align*}
 \phi(z)\coloneq S \chi(z),
\end{align*}
where we can write
\begin{align*}
 \chi(z)&\coloneq \PARENS{\begin{matrix} 1 \\ \chi_{[1]}(z) \\ \chi_{[2]}(z) \\ \vdots \\ \chi_{[k+1]}(z) \\ \vdots \end{matrix}}, &
 \chi_{[k+1]}\coloneq \PARENS{ \begin{matrix} z^{-(k+1)} \\ z^{k+1} \end{matrix}}
\end{align*}
and
\begin{align*}
 \phi_{[k+1]}(z)\coloneq \sum_{ l=0}^{k+1}S_{[k+1][l]}\chi_{[ l]}(z)=\PARENS{\begin{matrix} \phi^{(2k+1)}(z) \\ \phi^{(2k+2)}(z) \end{matrix}}=
 \PARENS{\begin{matrix}  z^{-(k+1)}+\dots+*z^{k} \\ *z^{k}+\dots+z^{k+1}\end{matrix}}.
\end{align*}
To properly connect both scenarios we need a further object:
\begin{definition}
The derivative reciprocal Verblusnky coefficients are defined by
\begin{align*}
 \bar{\lambda}_{2k+1}&\coloneq\frac{\d P^*_{2k+1}(0)}{\d z}.
\end{align*}
\end{definition}
Then, we have the following
\begin{pro}
When the Gauss--Borel factorization \eqref{gauss-olpuc} can be performed then \eqref{gauss-mvolput} also exists with
 the block entries of $S,H$  expressed in terms of the entries of $s,h$ as follows
\begin{align*}
 H_{[k+1]}&=-h_{2k+1}\PARENS{\begin{matrix}
 1              &  \bar{\alpha}_{2k+2} \\
 \alpha_{2k+2} &  1
 \end{matrix}}, &
 \beta_{[k+1]}&= \PARENS{\begin{matrix}
                   \bar{\lambda}_{2k+1}             & \bar{\alpha}_{2k+1} \\
                   \alpha_{2k+1} &         \lambda_{2k+1}
                  \end{matrix}},
 \end{align*}
 where in the last equation we take $k\geqslant 1$ and for $k =1$ we have $\beta_{[1]}=\PARENS{\begin{smallmatrix} \bar{\alpha}_1 \\ \lambda_1 \end{smallmatrix}}$.
Moreover,
 \begin{align*}
 \lambda_1&=\alpha_1 &
\lambda_{2k+3}-\lambda_{2k+1}&=\bar{\alpha}_{2k+1} \alpha_{2k+2}+\bar{\alpha}_{2k+2} \alpha_{2k+3},
\end{align*}
it also holds that
\begin{align*}
  \rho_k=1-\bar\alpha_k\alpha_k.
\end{align*}The  MVOLPUT$_{D=1}$,  $\{\phi(z)\}_{k=0}^\infty$, the OLPUC,  $\{\varphi_{k}(z)\}_{k=0}^\infty$, and the Szeg\H{o} polynomials and  reciprocal polynomials, $\{P_k(z),P_k^*(z)\}_{k=0}^\infty$,  are related as follows
\begin{align*}
 \phi^{(2k+1)}(z)&=\varphi^{(2k+1)}(z), &
 \phi^{(2k+2)}(z)&=-\alpha_{2k+2}\varphi^{(2k+1)}(z)+\varphi^{(2k+2)}(z),\\
\PARENS{\begin{matrix}
 z^{k+1} &  0 \\
 0       &  z^k
\end{matrix}} \phi_{[k+1]}(z)& = \PARENS{\begin{matrix}P_{2k+1}^*(z) \\ P_{2k+1}(z)\end{matrix}}, &z^{k+1} H_{[k+1]}^{-1} h_{2k+2} \phi_{[k+1]}(z)&= \PARENS{\begin{matrix}P_{2k+2}^*(z) \\ P_{2k+2}(z)\end{matrix}}.
\end{align*}
\end{pro}

\begin{proof}

To prove the results we must start removing from $s$ its block diagonal.    We introduce
\begin{align*}
 n\coloneq\PARENS{\begin{matrix}
     1  &  0        & 0  &     0     &  0 &        &        &              &   & \\
     0   & 1        & 0 &       0   &  0  &        &        &              &   & \\
     0   & \alpha_2 & 1 &    0      & 0   &        &        &              &   & \\
        0& 0         & 0  & 1        & 0 &        &        &              &   & \\
        0&     0     &  0 & \alpha_4 & 1 &        &        &              &   & \\
        &          &   &          &   & \ddots & \ddots &              &   & \\
        &          &   &          &   & \ddots & \ddots &              &   & \\
        &          &   &          &   &        &        & 1            & 0 & \\
        &          &   &          &   &        &        & \alpha_{2k+2}& 1 & \\
        &          &   &          &   &        &        &              &   & \ddots
    \end{matrix}}.
\end{align*}

If we do so in the factorization we end up with
\begin{align*}
G=s^{-1}n n^{-1}h (n^{-1})^{\dagger} n^{\dagger} (s^{-1})^{\dagger}=(n^{-1}s)^{-1} (n^{-1}h (n^{-1})^{\dagger}) [(n^{-1}s)^{-1}]^{\dagger},
\end{align*}
so that
\begin{align*}
 S&=n^{-1} s,   & H&=n^{-1}h(n^{-1})^{\dagger}.
\end{align*}
From the first expression the relations between $\phi$ and $\varphi$ follow and from the second one we get
\begin{align*}
 H_{[k+1]}&=\PARENS{\begin{matrix}
             h_{2k+1}              &  h_{2k+1}\bar{\alpha}_{2k+2} \\
             \alpha_{2k+2}h_{2k+1} &  h_{2k+2}+\alpha_{2k+2}h_{2k+1}\bar{\alpha}_{2k+2}
            \end{matrix}}.
\end{align*}
But due to the persymmetry of $H_{[k+1]}$ the elements in the main diagonal must coincide and therefore
\begin{align*}
 h_{2k+1} &= h_{2k+2}+\alpha_{2k+2}h_{2k+1}\bar{\alpha}_{2k+2},
\end{align*}
from where we deduce $\rho_{2k+2}\coloneq\frac{h_{2k+2}}{h_{2k+1}}=1-\alpha_{2k+2}\bar{\alpha}_{2k+2}$.
In order to get the expressions relating the $\phi$ to the Szeg\H{o} polynomials we only have to take the
relations between $\phi$ and $\varphi$ and remember both the expressions relating the $\varphi$ to the Szeg\H{o} polynomials and also the
recursion relation for the last ones
\begin{align*}
 P_k(z)=zP_{k-1}(z)+\alpha_k P_{k-1}^*(z).
\end{align*}
These relations  lead to expressions for  the entries of $\beta_{[2k+1]}$.
From the two alternatives of writing  the Jacobi operator we get
the relation between $\alpha$ and $\lambda$
\begin{align*}
 J_{[k+1],[k+1]}&=\begin{pmatrix}
                 0              & \bar \alpha_{2k+1} \\
                 -\alpha_{2k+3} & \lambda_{2k+1}-\lambda_{2k+3}
                 \end{pmatrix}\\
 &=\frac{h_{2k+1}}{h_{2k+2}} \PARENS{\begin{smallmatrix}
                               \lambda_{2k+1}-\lambda_{2k+3}-\bar\alpha_{2k+1}\alpha_{2k+2}-\bar\alpha_{2k+2}\alpha_{2k+3}   & -\bar\alpha_{2k+2}(\lambda_{2k+1}-\lambda_{2k+3}-\bar\alpha_{2k+2}\alpha_{2k+3}) +\bar\alpha_{2k+1} \\
                               [\alpha_{2k+2}(\lambda_{2k+1}-\lambda_{2k+3})-\alpha_{2k+3}]-\alpha_{2k+2}\bar\alpha_{2k+1}\alpha_{2k+2} & [\alpha_{2k+2}(\lambda_{2k+1}-\lambda_{2k+3})-\alpha_{2k+3}]\bar\alpha_{2k+2}+\alpha_{2k+2}\bar\alpha_{2k+1}
                            \end{smallmatrix}}.
\end{align*}
While from
the recursion relation for $\phi$, which can be expressed as follows
\begin{align*}
 z \phi_{[k+1]}&=\frac{h_{2k+1}}{h_{2k}}\begin{pmatrix}
                                        1             & -\bar\alpha_{2k} \\
                                        \alpha_{2k+2} & \alpha_{2k+2}\bar\alpha_{2k}
                                      \end{pmatrix}\phi_{[k]}+
                                      \begin{pmatrix}
                                        0             & \bar\alpha_{2k+1} \\
                                        -\alpha_{2k+3} & \bar\alpha_{2k+1}\alpha_{2k+2}+\bar\alpha_{2k+2}\alpha_{2k+3}
                                      \end{pmatrix}\phi_{[k+1]}
                                      \begin{pmatrix}
                                        0 & 0 \\
                                        0 & 1
                                      \end{pmatrix}\phi_{[k+2]},
\end{align*}
 we get
\begin{align*}
 z \phi_{2k+1}(z)=\frac{h_{2k+1}}{h_{2k}}[\phi_{2k-1}-\bar\alpha_{2k}\phi_{2k}]+\bar\alpha_{2k+1}\phi_{2k+2}.
\end{align*}
If we compare the coefficients in $z^{-k}$ we get
\begin{align*}
 1&=\frac{h_{2k+1}}{h_{2k}}+\bar\alpha_{2k+1}\alpha_{2k+1}
\end{align*}
which, recalling that  $\rho_{2k+1}\coloneq\frac{h_{2k+1}}{h_{2k}}$ gives the expressions for the odd $\rho$.
\end{proof}

Notice that for the previous result to hold we have assumed  that the Gauss--Borel factorization \eqref{gauss-olpuc} can be performed,  which implies the block Gauss--Borel factorization \eqref{gauss-mvolput}. The OLPUC  factorization \eqref{gauss-olpuc}  of \cite{carlos} was ensured whenever all the standard principal minors do not vanish, $\det G^{[k]}\neq 0$ for all $k\in\{0,1,2,\dots\}$; however, for the   MVOLPUT$_{D=1}$ construction we request a less demanding condition, now we require only the block principal minors not to vanish, hence it only requests $\det G^{[2k+1]}\neq 0$ for all $k\in\{0,1,2,\dots\}$, and nothing to  the even minors, which in fact could vanish. Whenever all the odd principal minors are non zero and some even principal minors do vanish we have  MVOLPUT$_{D=1}$  but not OLPUC. Let us explore this situation

\begin{pro}
Let us consider a measure such that $\det G^{[2k+1]}\neq 0$, $k\in\{1,2,\dots\}$, $\det G^{[2k+2]}\neq 0\,\,\, \forall k\neq  l$ and $\det G^{[2 l+2]}=0$.
Then
\begin{align*}
 H_{[k+1]}&=\begin{pmatrix}
             h_{2k+1}              &  h_{2k+1}\bar{\alpha}_{2k+2} \\
             \alpha_{2k+2}h_{2k+1} &  h_{2k+1}
            \end{pmatrix} & \forall k&\neq  l,
\end{align*}
and
\begin{align*}
 H_{[ l+1]}&=\PARENS{\begin{matrix}
             H_{2 l+1,2 l+1} &  H_{2 l+1,2 l+2} \\
             H_{2 l+2,2 l+1} &  H_{2 l+2,2 l+2}
            \end{matrix}}=
            \PARENS{\begin{matrix}
             0              &  H_{2 l+1} \\
             \bar{H}_{2 l+1} &  0
            \end{matrix}}  &  H_{2 l+1}\coloneq\langle \phi^{(2 l+1)},\phi^{(2 l+2)} \rangle.
\end{align*}
\end{pro}

\begin{proof}
From the imposed conditions on the moment matrix we see that $\forall k<2 l+2$ we have that $\det g^{[k]}> 0$ and therefore
$\det H_{[k]}>0$. But $\det g^{[2 l+2]}=0$ and therefore
\begin{align*}
 \begin{vmatrix}
      H_{[0]}     &        0     & 0           &        &               &            &  \\
             0     & H_{[1]}     &     0       &        &               &            &  \\
           0       &      0       & H_{[2]}    &        &               &            &  \\
                  &             &            & \ddots &               &            &  \\
                  &             &            &        &               &            &  \\
                  &             &            &        &               & H_{[ l]}    &  \\
                  &             &            &        &               &            & H_{2 l+1,2 l+1}
     \end{vmatrix}=\Big(\prod_{k=0}^{ l} \det H_{[k]}\Big)  H_{2 l+1,2 l+1}=0.
\end{align*}
Since $\prod_{k=0}^{ l} \det H_{[k]}=\det g^{[2 l+1]}>0$ the only possibility is that $H_{2 l+1,2 l+1}=0$. Then using the persymmetry of the
blocks of $H$ it is straightforward to get that $H_{2 l+2,2 l+2}=H_{2 l+1,2 l+1}=0$ and that $H_{2 l+1,2 l+2}=\bar{H}_{2 l+2,2 l+1}\coloneq H_{2 l+1}$. Notice
that since again $\det G^{[2 l+3]}> 0$ we have that $\det H_{[ l+1]}\neq 0$ and therefore $H_{2 l+1} \neq 0$.
\end{proof}

We have
\begin{align*}
 H_{[k+1]}&=\PARENS{\begin{matrix}
             h_{2k+1}              &  h_{2k+1}\bar{\alpha}_{2k+2} \\
             \alpha_{2k+2}h_{2k+1} &  h_{2k+1}
            \end{matrix}}=\PARENS{\begin{matrix}1 &0\\
            \alpha_{2k+2}&1
                          \end{matrix}} \PARENS{\begin{matrix} h_{2k+1}&0\\
                          0&h_{2k+1}
                          \end{matrix}}\PARENS{\begin{matrix} 1&0\\
                          \alpha_{2k+2}&1
                          \end{matrix}}^\dagger& \forall k&\neq  l,
\end{align*}

\section{Persymmetric and spectral matrices for $D=1,2$}\label{examples}
For $D=1$ the persymmetry matrix is
For $D=1$ we have
\begin{align*}
\eta_1=\PARENS{
	\begin{array}{c|cc|cc|cc|cc|c}
	1&0&0&0&0&0&0&0&0&\dots\\
	\hline
	0&0&\textbf{1}&0&0&0&0&0&0&\dots\\
	0&\textbf{1}&0&0&0&0&0&0&0&\dots\\
	\hline
	0&0&0&0&\textbf{1}&0&0&0&0&\dots\\
	0&0&0&\textbf{1}&0&0&0&0&0&\dots\\
	\hline
	0&0&0&0&0&0&\textbf{1}&0&0&\dots\\
	0&0&0&0&0&\textbf{1}&0&0&\textbf{1}&\\
	\hline
	0&0&0&0&0&0&0&0&\textbf{1}&\dots\\
	0&0&0&0&0&0&0&\textbf{1}&0&\dots\\
	\hline
	\vdots&\vdots &\vdots&\vdots& \vdots& \vdots&\vdots&\vdots&\vdots&  \end{array}},
\end{align*}
and  $D=2$ the partial persymmetries matrices are
{\small\begin{align*}
	\eta_1&=\PARENS{
		\begin{array}{c|cccc|cccccccc|c}
		1&0&0&0&0&0&0&0&0&0&0&0&0&\dots\\
		\hline
		0&0&0&0&\textbf{1}&0&0&0&0&0&0&0&0&\dots\\
		0&0&\textbf{1}&0&0&0&0&0&0&0&0&0&0&\dots\\
		0&0&0&\textbf{1}&0&0&0&0&0&0&0&0&0&\dots\\
		0&\textbf{1}&0&0&0&0&0&0&0&0&0&0&0&\dots\\
		\hline
		0&0&0&0&0&0&0&0&0&0&0&0&\textbf{1}&\dots\\
		0&0&0&0&0&0&0&0&0&0&\textbf{1}&0&0&\dots\\
		0&0&0&0&0&0&0&0&0&0&0&\textbf{1}&0&\dots\\
		0&0&0&0&0&0&0&0&\textbf{1}&0&0&0&0&\dots\\
		0&0&0&0&0&0&0&0&0&\textbf{1}&0&0&0&\dots\\
		0&0&0&0&0&0&\textbf{1}&0&0&0&0&0&0&\dots\\
		0&0&0&0&0&0&0&\textbf{1}&0&0&0&0&0&\dots\\
		0&0&0&0&0&\textbf{1}&0&0&0&0&0&0&0&\dots\\
		\hline
		\vdots&\vdots &\vdots&\vdots& &\vdots&\vdots&\vdots&\vdots&\vdots&
		\vdots&\vdots &  \vdots&   \end{array}},\end{align*}
	\begin{align*}
	\eta_2&=\PARENS{
		\begin{array}{c|cccc|cccccccc|c}
		1&0&0&0&0&0&0&0&0&0&0&0&0&\dots\\
		\hline
		0&\textbf{1}&0&0&0&0&0&0&0&0&0&0&0&\dots\\
		0&0&0&\textbf{1}&0&0&0&0&0&0&0&0&0&\dots\\
		0&0&\textbf{1}&0&0&0&0&0&0&0&0&0&0&\dots\\
		0&0&0&0&\textbf{1}&0&0&0&0&0&0&0&0&\dots\\
		\hline
		0&0&0&0&0&\textbf{1}&0&0&0&0&0&0&0&\dots\\
		0&0&0&0&0&0&0&\textbf{1}&0&0&0&0&0&\dots\\
		0&0&0&0&0&0&\textbf{1}&0&0&0&0&0&0&\dots\\
		0&0&0&0&0&0&0&0&0&\textbf{1}&0&0&0&\dots\\
		0&0&0&0&0&0&0&0&\textbf{1}&0&0&0&0&\dots\\
		0&0&0&0&0&0&0&0&0&0&0&\textbf{1}&0&\dots\\
		0&0&0&0&0&0&0&0&0&0&\textbf{1}&0&0&\dots\\
		0&0&0&0&0&0&0&0&0&0&0&0&\textbf{1}&\dots\\
		\hline
		\vdots&\vdots &\vdots&\vdots& &\vdots&\vdots&\vdots&\vdots&\vdots&
		\vdots&\vdots &  \vdots&   \end{array}}.
	\end{align*}
}

For $D=1$ we have the spectral matrix
\begin{align*}
\Upsilon_1=\PARENS{
	\begin{array}{c|cc|cc|cc|cc|c}
	0&0&\textbf{1}&0&0&0&0&0&0&\dots\\
	\hline
	\textbf{1}&0&0&0&0&0&0&0&0&\dots\\
	0&0&0&0&\textbf{1}&0&0&0&0&\dots\\
	\hline
	0&\textbf{1}&0&0&0&0&0&0&0&\dots\\
	0&0&0&0&0&0&\textbf{1}&0&0&\dots\\
	\hline
	0&0&0&\textbf{1}&0&0&0&0&0&\dots\\
	0&0&0&0&0&0&0&0&\textbf{1}&\\
	\hline
	0&0&0&0&0&\textbf{1}&0&0&0&\dots\\
	0&0&0&0&0&0&0&0&0&\dots\\
	\hline
	\vdots&\vdots &\vdots&\vdots& \vdots& \vdots&\vdots&\vdots&\vdots&  \end{array}},
\end{align*}

As an example, we show the first blocks for the $D=2$ case
\begin{align*}
\Upsilon_1=\PARENS{
	\begin{array}{c|cccc|cccccccc|c}
	0&0&0&0&\textbf{1}&0&0&0&0&0&0&0&0&\dots\\
	\hline
	\textbf{1}&0&0&0&0&0&0&0&0&0&0&0&0&\dots\\
	0&0&0&0&0&0&0&0&0&0&\textbf{1}&0&0&\dots\\
	0&0&0&0&0&0&0&0&0&0&0&\textbf{1}&0&\dots\\
	0&0&0&0&0&0&0&0&0&0&0&0&\textbf{1}&\dots\\
	\hline
	0&\textbf{1}&0&0&0&0&0&0&0&0&0&0&0&\dots\\
	0&0&\textbf{1}&0&0&0&0&0&0&0&0&0&0&\dots\\
	0&0&0&\textbf{1}&0&0&0&0&0&0&0&0&0&\dots\\
	0&0&0&0&0&0&0&0&0&0&0&0&0&\dots\\
	0&0&0&0&0&0&0&0&0&0&0&0&0&\dots\\
	0&0&0&0&0&0&0&0&0&0&0&0&0&\dots\\
	0&0&0&0&0&0&0&0&0&0&0&0&0&\dots\\
	0&0&0&0&0&0&0&0&0&0&0&0&0&\dots\\
	\hline
	\vdots&\vdots &\vdots&\vdots& &\vdots&\vdots&\vdots&\vdots&\vdots&
	\vdots&\vdots &  \vdots&   \end{array}},
\end{align*}
and
\begin{align*}
\Upsilon_2=\PARENS{
	\begin{array}{c|cccc|cccccccc|c}
	0&0&0&\textbf{1}&0&0&0&0&0&0&0&0&0&\dots\\
	\hline
	0&0&0&0&0&0&0&\textbf{1}&0&0&0&0&0&\dots\\
	\textbf{1}&0&0&0&0&0&0&0&0&0&0&0&0&\dots\\
	0&0&0&0&0&0&0&0&0&\textbf{1}&0&0&0&\dots\\
	0&0&0&0&0&0&0&0&0&0&0&\textbf{1}&0&\dots\\
	\hline
	0&0&0&0&0&0&0&0&0&0&0&0&0&\dots\\
	0&\textbf{1}&0&0&0&0&0&0&0&0&0&0&0&\dots\\
	0&0&0&0&0&0&0&0&0&0&0&0&0&\dots\\
	0&0&\textbf{1}&0&0&0&0&0&0&0&0&0&0&\dots\\
	0&0&0&0&0&0&0&0&0&0&0&0&0&\dots\\
	0&0&0&0&\textbf{1}&0&0&0&0&0&0&0&0&\dots\\
	0&0&0&0&0&0&0&0&0&0&0&0&0&\dots\\
	0&0&0&0&0&0&0&0&0&0&0&0&0&\dots\\
	\hline
	\vdots&\vdots &\vdots&\vdots& &\vdots&\vdots&\vdots&\vdots&\vdots&
	\vdots&\vdots &  \vdots&  \end{array}
}.\\
\end{align*}

\section{Proofs}\label{proofss}
\settocdepth{section}
\subsection{Proof of Proposition \ref{pro:gauss}}\label{proof1}
\begin{proof}
\begin{enumerate}
\item Assuming  $\det A\neq 0$ for any block matrix $M=\left(\begin{smallmatrix}
    A & B\\
    C & D
   \end{smallmatrix}\right)$ we can write in terms of Schur complements
\begin{align*}
M&=\PARENS{\begin{matrix}
    \mathbb{I} & 0\\
    CA^{-1} & \mathbb{I}
   \end{matrix}}
\PARENS{\begin{matrix}
    A & 0\\
    0 & M/ A
   \end{matrix}}
\PARENS{\begin{matrix}
    \mathbb{I} & A^{-1}B\\
    0 & \mathbb{I}
   \end{matrix}}.
\end{align*}
Thus, as  $\det G^{[k]}\neq 0$ $\forall k=0,1,\dots$, we can write
\begin{align*}
 G^{[ l+1]}&=\PARENS{\begin{array}{c|c}
    \mathbb{I}^{q[ l]}              &    0    \\\hline\bigstrut[t]
     v^{[ l],[ l-1]}      &        \mathbb{I}_{|[ l]| }
    \end{array}}
\PARENS{\begin{array}{c|c}
  G^{[ l]}       & 0       \\\hline\bigstrut[t]
       0                 &              G^{[ l+1]}/ G^{[ l]}
      \end{array}}
\PARENS{\begin{array}{c|c}
\I^{[ l]}           &  (\hat v^{[ l],[ l-1]})^{\top}      \\\hline\bigstrut[t]
           0                   &  \I_{|[ l]| }
      \end{array}},
\end{align*}
 where
 \begin{align*}
v^{[ l],[ l-1]}\coloneq \PARENS{\begin{matrix}
               v_{[ l],[0]} & v_{[ l],[1]} & \dots & v_{[ l],[ l-1]}
              \end{matrix}}
 \end{align*}
Applying the  same factorization   to $G^{[ l]}$
we get
\begin{multline*}
G^{[ l+1]}=\PARENS{\begin{array}{c|cc}
     \mathbb{I}^{[ l-1]}              &     0       &  0  \\\hline\bigstrut[t]
      r^{[ l-1][ l-2]}                  & \mathbb{I}_{|[ l-1]| }   &  0    \\
      s^{[ l][ l-2]}                  &      t_{[ l][ l-1]}      &  \mathbb{I}_{|[ l]| }
      \end{array}}
\PARENS{\begin{array}{c|cc}
   G^{[ l-1]}                        &      0     &    0\\\hline\bigstrut[t]
      0              & G^{[ l]}/ G^{[ l-1]} &  0    \\
      0              &          0      &  G^{[ l+1]}/ G^{[ l]}
      \end{array}}\\\times
\PARENS{\begin{array}{c|cc}
   \mathbb{I}^{[ l-1]}                & (\hat r^{[ l-1][ l-2]})^\top  &  (\hat s^{[ l][ l-2]})^\top  \\\hline\bigstrut[t]
             0              & \mathbb{I}_{|[ l-1] |}   &    (\hat t_{[ l],[ l-1]})^\top \\
              0              &      0      &  \mathbb{I}_{|[ l]| }
      \end{array}}.
\end{multline*}
Here the zeros indicate zero rectangular matrices of different sizes. Finally, the iteration of these factorizations leads to
\begin{align*}
 G^{[ l+1]}=\PARENS{\begin{matrix}
  \mathbb{I}_{|[0] |}&     0     &\dots&0\\
         *   &\mathbb{I}_{|[1]| }&    \ddots  &\vdots\\
\vdots    &\ddots          &\ddots&0\\
*       &\dots   &*   &\mathbb{I}_{|[ l] |}\\
\end{matrix}}\diag(G^{[1]}/ G^{[0]},G^{[2]}/ G^{[1]},\dots, &G^{[ l+1]}/ G^{[ l]})
\PARENS{\begin{matrix}
  \mathbb{I}_{|[0]| }&     0     &\dots&0\\
         *   &\mathbb{I}_{|[1] |}&    \ddots  &\vdots\\
\vdots    &\ddots          &\ddots&0\\
*       &\dots   &*   &\mathbb{I}_{|[ l]| }\\
\end{matrix}}^\top
\end{align*}
Since this would have been valid for any $ l$ it would also hold for the direct
limit $\lim\limits_{\longrightarrow}G^{[ l]}$.

Consequently, $H_{[ l]}=G^{[ l+1]}/ G^{[ l]}$ is a Schur complement  and $\det G^{[ l+1]}=\prod_{k=0}^{ l}\det H_{[k]}\neq 0$.
\item  We have
\begin{align*}
  G^\dagger&=\Big(\oint_{\T^D}\chi(\z(\boldsymbol\theta))\d\mu(\boldsymbol\theta)\big(\chi(\z(\boldsymbol\theta))\big)^\dagger\Big)^\dagger\\
  &=\oint_{\T^D}\chi(\z(\boldsymbol\theta)){\d\bar\mu(-\boldsymbol\theta)}\big(\chi(\z(\boldsymbol\theta))\big)^\dagger.
\end{align*}
Thus,
\begin{align*}
G-G^\dagger=
&=\oint_{\T^D}\chi(\z(\boldsymbol\theta))\big[\d\mu(\boldsymbol\theta)-{\d\bar\mu(-\boldsymbol\theta)}\big]\big(\chi(\z(\boldsymbol\theta))\big)^\dagger.
\end{align*}
Consequently, for real measures the moment matrix is Hermitian. Conversely, if $G$ is Hermitian then  $\d\mu(\boldsymbol\theta)-{\d\bar\mu(-\boldsymbol\theta)}$ has all its moments equal to zero, and therefore can be identified with zero.

Now, as the factorization is unique from $\hat S^{-1} H^\dagger \big(S^{-1}\big)^\dagger =G^\dagger=G=S^{-1} H \big(\hat S^{-1}\big)^\dagger$ we deduce  that $\hat S=S$ and $H^\dagger=H$.
\item The truncated moment matrix $G^{[k]}$ satisfies for any truncated vector $v^{[k]}=(v_{[0]},v_{[1]},\dots,v_{[k-1]})^\top\in\C^{\sum_{j=0}^{k-1}|[j]|}$
\begin{align*}
  \big(v^{[k]}\big)^\dagger G^{[k]}v^{[k]}=&\sum_{i,j=0}^{k-1}\oint_{\T^D}\big(v_{[i]}\big)^\dagger\chi_{[i]}(\z(\boldsymbol\theta))\d\mu(\boldsymbol\theta\big(\chi_{[j]}(\z(\boldsymbol\theta))\big)^\dagger v_{[j]}\\
=&\oint_{\T^D}\Big|\sum_{i=0}^{k-1}\big(v_{[i]}\big)^\dagger\chi_{[i]}(\z(\boldsymbol\theta))\Big|^2\d\mu(\boldsymbol\theta),
\end{align*}
and as the measure is definite positive we conclude that the truncation is a definite positive matrix. Conversely, if every truncation is definite positive then $\oint |f_k(\boldsymbol{\theta})|^2\d\mu(\boldsymbol{\theta})>0$, for truncated Fourier series $f_k$. Consequently, the measure is definite positive.
\end{enumerate}
\end{proof}

\subsection{Proof of Proposition \ref{pro:secondMVOLPUT}}\label{proof2}
\begin{proof}
From
	\begin{align*}
	\mathcal C=&(G\hat S^\dagger H^{-1})^\dagger \chi(\z)\\
	=& (H^{-1})^\dagger  \hat S G^\dagger \chi(\z)\\
	\end{align*}
we conclude
	\begin{align*}
	(H^\dagger\mathcal C)_\q=& \sum_{0\leqslant |\q_1|\leqslant|\q|}\hat S_{\q,\q_1}\sum_{\q_2\in\Z^D}\bar G_{\q_2,\q_1} \z^{\q_2}\\
	=&(2\pi)^D \sum_{0\leqslant |\q_1|\leqslant|\q|}\hat S_{\q,\q_1}\sum_{\q_2\in\Z^D}\bar c_{\q_1-\q_2} \z^{\q_2}\\
	=&(2\uppi)^D\sum_{0\leqslant |\q_1|\leqslant|\q|}\hat S_{\q,\q_1}\z^{\q_1}\sum_{\q_2'\in\Z^D}\bar c_{\q_2'} \z^{-\q_2'}\\
	=&(2\uppi)^D\hat \phi_{\q}(\z)\bar{\hat\mu}(\z^{-1}).
	\end{align*}
	For the second relation we deduce
	\begin{align*}
	\hat {\mathcal  C}=&H^{-1}S G \chi(\z),
	\end{align*}
	and therefore conclude
	\begin{align*}
	(H\hat{\mathcal C})_\q=& \sum_{0\leqslant |\q_1|\leqslant|\q|} S_{\q,\q_1}\sum_{\q_2\in\Z^D} G_{\q_1,\q_2} \z^{\q_2}\\
	=&(2\uppi)^D \sum_{0\leqslant |\q_1|\leqslant|\q|} S_{\q,\q_1}\sum_{\q_2\in\Z^D} c_{-\q_1+\q_2} \z^{\q_2}\\
	=&(2\uppi)^D\sum_{0\leqslant |\q_1|\leqslant|\q|} S_{\q,\q_1}\z^{\q_1}\sum_{\q_2'\in\Z^D} c_{\q_2'} \z^{\q_2'}\\
	=&(2\uppi)^D\phi_{\q}(\z)\hat\mu(\z).
	\end{align*}
\end{proof}

\subsection{Proof of Proposition \ref{pro:hyperoctants}}\label{proof3}
\begin{proof}
	We need to show that for each  multi-index $\q\in\Z^D$ there exists a unique  $\sigma\subseteq\Z_D$ such that  $\q\in(\Z^D)_\sigma$.
	To prove it we look at the components of $\q$ and construct three associated subsets, $\sigma^\gtrless\subseteq \Z_D$ for the strictly  positive/negative entries and $\sigma^0\subseteq\Z_D$ that collects the components with zero entries. Consequently, we have that $\sigma^<\subseteq\sigma$ and $\sigma^>\subseteq\complement\sigma$, and $\Z_D=\sigma^<\cup\sigma^0\cup\sigma^>$ with
	$\sigma^<\cap\sigma^0= \sigma^<\cap\sigma^>=\sigma^0\cap\sigma^>=\emptyset$.  To construct our set $\sigma$ we need to complete $\sigma^<$ with elements in $\sigma^0$, but which? We split $\sigma^0=\sigma^{0,-}\cup\sigma^{0,+}$, with $\sigma^{0,\mp}\subset \sigma^0$ and $\sigma^{0,-}\cap\sigma^{0,+}=\emptyset$ in the following manner. Take an element in $i\in\sigma^<$
	and check whether or not $i-1,i-2,\dots,i-p\in\sigma^0$ and $i-p-1\not\in\sigma^0$, if  it happens add the elements $i-1,\dots,i-p$ to $\sigma^{0,-}$, after completing this procedure  $\forall i\in\sigma^<$ we have constructed $\sigma^{0,-}$ and we are ready to define $\sigma^{0,+}\coloneq\sigma^0\setminus\sigma^{0,-}$. Then, we declare $\sigma\coloneq\sigma^<\cup\sigma^{0,-}$ so that $\complement\sigma=\sigma^>\cup\sigma^{0,+}$. By construction $(\Z^D)_\sigma=\bigtimes\limits_{i=1}^D Z_i$ where each $Z_i$ is as follows, if $i\in\sigma^{0,-}$ then $i\in (\sigma\setminus\partial\sigma)$ so that $Z_i=\Z_-=\Z_>\cup\{0\}$. The set $\sigma^{0,+}$ is built up of  strings of zeroes lying at the left of an element in $\sigma^>$, hence $\sigma^{0,+}\subseteq (\complement\sigma\setminus\partial\complement\sigma)$ and for $i\in\sigma^{0,+}$ we have $Z_i=\Z_+=\Z_>\cup\{0\}$. For the elements $i\in\sigma^<$ we could have
	$Z_i=\Z_-$ for $i\in\sigma\setminus \partial\sigma$ or $Z_i=\Z_<$, otherwise; a similar discussion holds when  $i\in\sigma^>$.
\end{proof}

\subsection{Proof of Proposition \ref{importante}}\label{proof4}
\begin{proof}
	First, we prove that the set of spectral matrices  $\{\Upsilon_a\}_{a=1}^D$ is an Abelian set. In the one hand, a straightforward calculation leads to the relations
	\begin{align*}
	\big((\Upsilon_a\Upsilon_b)_{[k+1],[k-1]}\big)_{i,j}&=\sum_{r=1}^{|[k]|}(\Upsilon_a )_{\boldsymbol{\alpha}^{(k+1)}_i,\boldsymbol{\alpha}^{(k)}_{r}}
	(\Upsilon_b )_{\boldsymbol{\alpha}^{(k)}_{r},\boldsymbol{\alpha}^{(k-1)}_j}\\&=
	\delta_{\boldsymbol{\alpha}^{(k+1)}_i+\boldsymbol{e}_a+\boldsymbol{e}_b,\boldsymbol{\alpha}^{(k-1)}_j},\\
	\big((\Upsilon_a\Upsilon_b)_{[k+1],[k+3]}\big)_{i,j}&= \sum_{r=1}^{|[k+2]|}(\Upsilon_a )_{\boldsymbol{\alpha}^{(k+1)}_i,\boldsymbol{\alpha}^{(k+2)}_r}
	(\Upsilon_b )_{\boldsymbol{\alpha}^{(k+2)}_r,\boldsymbol{\alpha}^{(k+3)}_j}\\&=
	\delta_{\boldsymbol{\alpha}^{(k+1)}_i+\boldsymbol{e}_a+\boldsymbol{e}_b,\boldsymbol{\alpha}^{(k+3)}_j}.
	\end{align*}
	On the other hand
	\begin{multline*}
	\sum_{r=1}^{|[k]|}(\Upsilon_a )_{\boldsymbol{\alpha}^{(k+1)}_{j},\boldsymbol{\alpha}^{(k)}_{r}}
	(\Upsilon_b )_{\boldsymbol{\alpha}^{(k)}_{r},\boldsymbol{\alpha}^{(k+1)}_i}+
	\sum_{r=1}^{|[k+2]|}(\Upsilon_a )_{\boldsymbol{\alpha}^{(k)}_{j},\boldsymbol{\alpha}^{(k+2)}_r}
	(\Upsilon_b )_{\boldsymbol{\alpha}^{(k+2)}_r,\boldsymbol{\alpha}^{(k+1)}_i}\\=\sum_{r=1}^{|[k]|} \delta_{\boldsymbol{\alpha}^{(k+1)}_{j}+\boldsymbol{e}_a,\boldsymbol{\alpha}^{(k)}_{r}}
	\delta_{\boldsymbol{\alpha}^{(k)}_{j}+\boldsymbol{e}_b,\boldsymbol{\alpha}^{(k+1)}_i}+
	\sum_{r=1}^{|[k+2]|} \delta_{\boldsymbol{\alpha}^{(k+1)}_{j}+\boldsymbol{e}_a,\boldsymbol{\alpha}^{(k+2)}_r}
	\delta_{\boldsymbol{\alpha}^{(k+2)}_r+\boldsymbol{e}_b,\boldsymbol{\alpha}^{(k+1)}_i}.
	\end{multline*}
	In the RHS  the first term involves $\delta_{\boldsymbol{\alpha}^{(k+1)}_{j}+\boldsymbol{e}_a+\boldsymbol{e}_b,\boldsymbol{\alpha}^{(k+1)}_{r}}$ which is not zero  only when the   $a$-th component of the multi-index $\boldsymbol{\alpha}^{(k+1)}_i$ is negative and the $b$-th component of the multi-index $\boldsymbol{\alpha}^{(k+1)}_j$ is positive; not only but in the second term of the RHS for non zero terms we need the $a$-th of $\boldsymbol{\alpha}^{(k+1)}_i$ to be positive and the $b$-th component of
	$\boldsymbol{\alpha}^{(k+1)}_j$ to be negative. Hence, we are left with
	\begin{align*}
	\big( (\Upsilon_a\Upsilon_b)_{[k+1],[k+1]}\big)_{i,j}&=\sum_{r=1}^{|[k]|}(\Upsilon_a )_{\boldsymbol{\alpha}^{(k+1)}_i,\boldsymbol{\alpha}^{(k)}_{r}}
	(\Upsilon_b )_{\boldsymbol{\alpha}^{(k)}_{r},\boldsymbol{\alpha}^{(k+1)}_j}+
	\sum_{r=1}^{|[k+2]|}(\Upsilon_a )_{\boldsymbol{\alpha}^{(k)}_i,\boldsymbol{\alpha}^{(k+2)}_r}
	(\Upsilon_b )_{\boldsymbol{\alpha}^{(k+2)}_r,\boldsymbol{\alpha}^{(k+1)}_j}\\&=
	\delta_{\boldsymbol{\alpha}^{(k+1)}_i+\boldsymbol{e}_a+\boldsymbol{e}_b,\boldsymbol{\alpha}^{(k+1)}_j}.
	\end{align*}
	Second, we show that $\Upsilon_a$ is an orthogonal matrix; i.e., $\Upsilon_a\Upsilon_a^\top=\Upsilon_a^\top\Upsilon_a=\I$.
	From the very Definition \ref{def:upsilon} of the $\Upsilon$'s we deduce that
	\begin{align*}
	(\Upsilon_a^{\top} )_{\boldsymbol{\alpha}^{(k)}_i,\boldsymbol{\alpha}^{(k+1)}_j}&=
	\delta_{\boldsymbol{\alpha}^{(k)}_i-\boldsymbol{e}_a,\boldsymbol{\alpha}^{(k+1)}_j}, \\
	(\Upsilon_a^{\top} )_{\boldsymbol{\alpha}^{(k+1)}_i,\boldsymbol{\alpha}^{(k)}_j}&=
	\delta_{\boldsymbol{\alpha}^{(k+1)}_i-\boldsymbol{e}_a,\boldsymbol{\alpha}^{(k)}_j}.
	\end{align*}
	We compute $\Upsilon_a \Upsilon_a^{\top}$ as we did with $\Upsilon_a\Upsilon_b$  and we get the identity.
	Finally,  we observe that
	\begin{align*}
	(\eta_a \Upsilon_b \eta_a )_{\boldsymbol{\alpha}^{(k)}_i,\boldsymbol{\alpha}^{(k+1)}_j}&=
	\delta_{I_a\boldsymbol{\alpha}^{(k)}_i+\boldsymbol{e}_b,I_a\boldsymbol{\alpha}^{(k+1)}_j}=
	\delta_{\boldsymbol{\alpha}^{(k)}_i+I_a\boldsymbol{e}_b,\boldsymbol{\alpha}^{(k+1)}_j}.
	\end{align*}
	But as $I_a \boldsymbol{e}_b=\boldsymbol{e}_b$ whenever $a \neq b$ and $I_a \boldsymbol{e}_a=-\boldsymbol{e}_a$  we get the desired result.
\end{proof}

\subsection{Proof of Proposition \ref{JC}}\label{proof5}
\begin{proof}
	\begin{enumerate}
		\item From the string equation \eqref{string} and the Gauss--Borel factorization \eqref{cholesky} we get
		\begin{align*}
		\Upsilon_\q   S^{-1} H \big(\hat S^{-1}\big)^{\dagger}=  S^{-1} H \big(\hat S^{-1}\big)^{\dagger}\Upsilon_\q
		\end{align*}
		and hence
		\begin{align*}
		S\Upsilon_\q S^{-1} H=  H \big(\hat S^{-1}\big)^{\dagger}\Upsilon_\q\big(\hat S\big)^{\dagger},
		\end{align*}
		and the result follows.
		\item From the Definition \ref{def:jacobi} we deduce that $J_\q$ and $\hat J_\q$  have only the first $|\q|$ superdiagonals different from zero, therefore (1) of this Proposition leads to the stated result.
		\item It follows from $\Upsilon_0=\I$.
		\item It is a direct consequence of Proposition \ref{pro:upsilon} (3)
		\item From the string equation \eqref{string} and the persymmetry property \eqref{eq:persymmetry_moment} we get
		\begin{align*}
		\eta\Upsilon_\q G=&\eta G\Upsilon_\q\\
		=& G^\top \eta \Upsilon_\q,
		\end{align*}
		and using  the  Gauss--Borel factorization \eqref{cholesky} we obtain
		\begin{align*}
		\eta\Upsilon_\q S^{-1} H \big(\hat S^{-1}\big)^{\dagger}=& \big(S^{-1} H \big(\hat S^{-1}\big)^{\dagger}\big)^\top\eta \Upsilon_\q\\
		=& \big(\bar{\hat S}\big)^{-1} H^\top  (S^\top)^{-1}\eta \Upsilon_\q,
		\end{align*}
		so that
		\begin{align*}
		\bar{\hat S}  \eta\Upsilon_\q S^{-1} H=& \big(S^{-1} H \big(\hat S^{-1}\big)^{\dagger}\big)^\top\eta \Upsilon_\q\\
		=& H^\top  (S^\top)^{-1}\eta \Upsilon_\q\hat S^{\dagger}\\
		=& H^\top  (S^\top)^{-1}\Upsilon_\q^\top\eta\hat S^{\dagger},
		\end{align*}
		consequently $C_\q H=H^\top  C_\q^\top$ and $ C_{\q}=H^{\top} C_{\q}^{\top} H^{-1}$ follows. In particular,  $ C_0 H=H^{\top} C_0^{\top}$
		\begin{align*}
		C_0=\bar{\hat S}  \eta S^{-1} = H^\top  (S^\top)^{-1}\eta\hat S^{\dagger}H^{-1},
		\end{align*}
		and, as we have a lower triangular matrix on the LHS and a upper triangular matrix on the RHS, the only option for $C_0$ is to be a diagonal matrix  and the second relation $C_0=\eta$ is proven. Finally, notice that
		\begin{align*}
		C_\q^{-1}&=\Big(\bar{\hat S}  \eta \Upsilon_\q S^{-1} \Big)^{-1}\\
		&= S \Upsilon_q^{-1}\eta^{-1} \bar{\hat S}^{-1}
		\end{align*}
		and using Propositions \ref{eta} and  \ref{pro:upsilon} we conclude the third relation.
		\item Observe  that
		\begin{align*}
		C_\q=&\big(H^\top)^{-1}\bar{\hat S}\eta S^{-1} S \Upsilon_\q S^{-1}\\
		=&\eta J_\q.
		\end{align*}
	\end{enumerate}
\end{proof}

\subsection{Proof of Proposition \ref{pro:nice-longitude}}\label{proof6}
\begin{proof}
	$\Leftarrow$	 We need to prove that $\ell({z^{\q'}L})=|\q'|+m$ for all $\q'\in\Z^D$;  i.e., that the transformed Newton polytope $\{\q'\}+\operatorname{NP}(L)$  intersects with the
	faces of $\operatorname{Conv}([|\q'|+m])$ for each $\q'\in\Z^D$.
	From Proposition \ref{pro:hyperoctants} we know that there exists an orthant labelled by $\sigma\in2^{\Z_D}$ such that $\q'\in(\Z^D)_\sigma$. By hypothesis, for each $\sigma\in 2^{\Z_D}$ we can find  $\q_\sigma\in\overline{(\R^D)_\sigma}\cap  \operatorname{NP}(L)$ so that $|\q_\sigma|=m$ and we can ensure that
	$|\q'+\q_\sigma|=|\q'|+|\q_\sigma|$,  $\forall\q'\in(Z^D)_\sigma$. Finally,  for each orthant $\overline{(\R^D)_\sigma}$ we have such multi-index, consequently in the transformed Newton polytope $\{\q'\}+\operatorname{NP}(L)$ we have a non empty intersection with the
	faces of $\operatorname{Conv}([m+|\q'|])$ for all $\q'\in\Z$.
	
	$\Rightarrow$ 	Suppose that there is an orthant $\overline{(\R^D)_\sigma}$ in where we have no points in the Newton polynomial of longitude $m$ then for any $\q'\in (\Z^D)_\sigma$ in that orthant we have that $\ell(\z^{\q'}L)<|\q'|+\ell(L)$.
\end{proof}

\subsection{Proof of Proposition \ref{pro:prod-nice}}\label{proof7}
	\begin{proof}
		$\Rightarrow$ 	Let $L_1$ and $L_2$ be two nice Laurent polynomials.	Then,  for each subset $\sigma\in 2^{\Z_D}$ there are multi-indices $(\q_i)_{\sigma}\in\operatorname{NP}(L_i)\cap \overline{(\R^D)_\sigma}$ of longitude $|(\q_i)_{\sigma}|=\ell(L_i)$, $i\in\{1,2\}$. Hence, recalling \eqref{eq:sumNP} we see that $(\q_1)_{\sigma}+(\q_1)_{\sigma}\in\operatorname{NP}(L_1L_2)$ and moreover, as both belong to the same orthant, $|(\q_1)_{\sigma}+(\q_2)_{\sigma}|=|(\q_1)_{\sigma}|+|(\q_2)_{\sigma}|=\ell(L_1)+\ell(L_2)$. Therefore,
		$(\q_1)_{\sigma}+(\q_1)_{\sigma}\in\operatorname{NP}(L_1L_2)\cap\operatorname{Conv}([\ell(L_1)+\ell(L_2)])\cap\overline{(\R^D)_\sigma}\neq \emptyset $, and recalling Proposition \ref{pro:whoisnice} the result is proven.
		
		$\Leftarrow$ Let $L_1$ and $L_2$ be two  Laurent polynomials being at least one of them not nice, assume that $L_1$ is not nice. This means that there exists at least an orthant where all  the multi-indices in $\operatorname{NP}(L_1)\cap \overline{(\R^D)_\sigma}$ fulfill  $|(\q_1)_{\sigma}|< \ell(L_1)$. Thus , when we consider the product $L_1L_2$ and we look in this orthant  $\overline{(\R^D)_\sigma}$ we will have multi-indexes of the form $(\q_1)_{\sigma}+(\q_2)_{\sigma}$ with $|(\q_1)_{\sigma}+(\q_2)_{\sigma}|=|(\q_1)_{\sigma}|+|(\q_2)_{\sigma}|< \ell(L_1)+\ell(L_2)$, and the equality is never saturated. The sum of multi-indexes of different orthants will not help in achieving the longitude $\ell(L_1)+\ell(L_2)$.
	\end{proof}
	
	\subsection{Proof of Lemma \ref{lemma:fullrank}}\label{proof8}
	\begin{proof}
		We will show that the $|[k+1]|$ columns in this matrix are actually linearly independent. For that to be true it is enough
		to show that  the only solutions for the equation
		$\left(\chi_{[1]}(\boldsymbol{\Upsilon})\right)_{[k],[k+1]}\boldsymbol{X}=0$, with
		\begin{align*}
		\boldsymbol{X}=\PARENS{\begin{matrix}
			X_{\boldsymbol{\alpha}_1^{(k+1)}}\\\vdots\\X_{\boldsymbol{\alpha}_{|[k+1]|}^{(k+1)}}
			\end{matrix}}\in \C^{|[k+1]|},
		\end{align*}
		is $\boldsymbol{X}=0$.
		First, let us consider the equation $(\Upsilon_a)_{[k][k+1]}\boldsymbol{X}=0$ from the Definition \ref{def:upsilon}
		we can deduce that the previous equation will hold as long as every $X_{\boldsymbol{\alpha}_j^{(k+1)}}$
		such that ${\alpha}_{j,a}^{(k+1)}>0$ is equal to zero. The same holds for
		$(\Upsilon_a^{-1})_{[k][k+1]}\boldsymbol{X}=0$, it implies that every $X_{\boldsymbol{\alpha}_j^{(k+1)}}$
		with $\alpha_{j,a}^{(k+1)}<0$ must equal zero. Therefore, for both equations to hold, we need that every
		$X_{\boldsymbol{\alpha}_j^{(k+1)}}$ having ${\alpha}_{j,a}^{(k+1)}\neq 0$ must vanish. In other words,
		only those $X_{\boldsymbol{\alpha}_j^{(k+1)}}$ having ${\alpha}_{j,a}^{(k+1)}=0$ can be different from zero.
		But our request is for all $a\in\{1,2,\dots,D\}$ which implies that the only nonzero component of $\boldsymbol X$ is the one for $\q=(0,\dots,0)^\top\not\in[k+1]$ and therefore $\boldsymbol X=0$.
	\end{proof}
	
	\subsection{Proof of Lemma \ref{lemma:GD}}\label{proof9}
	\begin{proof}
		For the first line just consider the differentiation $\frac{\partial }{\partial t_{\q}}$ of the deformed $LU$
		factorization \eqref{eq:borel.evol} in the form
		\begin{align}
		S(t)W_0(t)G&= H(t)\big((\hat S(t))^{-1}\big)^\dagger
		\end{align}
		which gives
		\begin{align*}
		\frac{\partial S}{\partial t_{\q}}S^{-1}+S\bUpsilon^\q S^{-1}&=
		\Big(\frac{\partial H}{\partial t_{\q}}H^{-1}-H \left(\hat{S}^{-1}\right)^{\dagger}\frac{\partial \hat{S}^{\dagger}}{\partial t_{\q}}H^{-1}\Big).
		\end{align*}
		Split  it into upper and lower triangular to get
		\begin{align*}
		\frac{\partial S}{\partial t_{\q}}S^{-1}+S\bUpsilon^\q S^{-1}&=\Big(S\bUpsilon^\q S^{-1}\Big)_\geq
		=\Big(\frac{\partial H}{\partial t_{\q}}H^{-1}-H \left(\hat{S}^{-1}\right)^{\dagger}\frac{\partial \hat{S}^{\dagger}}{\partial t_{\q}}H^{-1}\Big),&
		\frac{\partial S}{\partial t_{\q}}S^{-1}&=-\Big(S\bUpsilon^\q S^{-1}\Big)_<
		\end{align*}
		While for the second one consider the differentiation $\frac{\partial }{\partial \bar{t}_{-\q}}$ of the deformed $LU$
		factorization \eqref{eq:borel.evol} in the form
		\begin{align}
		\hat{S}(t)\left(W_0(t)\right)^{\dagger}G^{\dagger}&= \left(H(t)\right)^{\dagger}\big(( S(t))^{-1}\big)^\dagger.
		\end{align}
		which gives
		\begin{align*}
		\frac{\partial \hat S}{\partial \bar{t}_{-\q}}\hat{S}^{-1}+\hat S\bUpsilon^\q \hat S^{-1}&=
		\Big(\frac{\partial H^{\dagger}}{\partial \bar{t}_{-\q}}(H^{-1})^{\dagger}-H^{\dagger}\left(S^{-1}\right)^{\dagger}\frac{\partial S^{\dagger}}{\partial \bar{t}_{-\q}}\left(H^{-1}\right)^{\dagger}\Big).
		\end{align*}
		Now,  using the same splitting technique as before
		\begin{align*}
		\frac{\partial \hat S}{\partial \bar{t}_{-\q}}\hat{S}^{-1}+\hat S\bUpsilon^\q \hat S^{-1}&=
		\Big(\hat S\bUpsilon^\q \hat{S}^{-1}\Big)_\geq=
		\Big(\frac{\partial H^{\dagger}}{\partial \bar{t}_{-\q}}(H^{-1})^{\dagger}-H^{\dagger}\left(S^{-1}\right)^{\dagger}\frac{\partial S^{\dagger}}{\partial \bar{t}_{-\q}}\left(H^{-1}\right)^{\dagger}\Big),\\
		\frac{\partial \hat S}{\partial \bar{t}_{-\q}}\hat S^{-1}&=-\Big(\hat S\bUpsilon^\q \hat{S}^{-1}\Big)_<.
		\end{align*}
	\end{proof}
	
\subsection{Proof of Proposition \ref{pro:integrable elements}}\label{proof10}
	\begin{proof}
		\begin{enumerate}
			\item By differentiation in Definition \ref{def:wave} we get
			\begin{align*}
			\frac{\partial W_1}{\partial t_{\q}}=&\Big(\frac{\partial S}{\partial t_{\q}}S^{-1}+S\bUpsilon^\q S^{-1}\Big)W_1,&
			\frac{\partial \hat W_1}{\partial \bar{t}_{-\q}}=&\Big(\frac{\partial \hat S}{\partial \bar{t}_{-\q}}\hat S^{-1}+\hat S\bUpsilon^{\q} \hat S^{-1}\Big)\hat W_1, \\
			\frac{\partial \hat W_2}{\partial t_{\q}}=&\Big(\frac{\partial H}{\partial t_{\q}}H^{-1}-H \left(\hat{S}^{-1}\right)^{\dagger}\frac{\partial \hat{S}^{\dagger}}{\partial t_{\q}}H^{-1}\Big)\hat{W}_2,&
			\frac{\partial  W_2}{\partial \bar{t}_{-\q}}=&\Big(\frac{\partial H^{\dagger}}{\partial \bar{t}_{-\q}}(H^{-1})^{\dagger}-H^{\dagger}\left(S^{-1}\right)^{\dagger}\frac{\partial S^{\dagger}}{\partial \bar{t}_{-\q}}\left(H^{-1}\right)^{\dagger}\Big)W_2.
			\end{align*}		
			Consequently,  from Lemma \ref{lemma:GD} and its proof we get the result.
			\item From   Definition \ref{def:jacobi} we know that the perturbed Jacobi matrices reads:
			$J_\q(t):= S(t)\bUpsilon_\q \big(S(t)\big)^{-1}$ and $\hat J_\q(t):= \hat S(t)\bUpsilon_\q \big(\hat S(t)\big)^{-1}$,
			therefore by differentiation we get
			\begin{align*}
			\frac{\partial J_\q}{\partial t_{\q'}}=&\Big[	\frac{\partial S}{\partial t_{\q'}}S^{-1},J_\q\Big],&
			\frac{\partial \hat J_\q}{\partial \bar{t}_{-\q'}}=&\Big[	\frac{\partial \hat S}{\partial \bar{t}_{-\q'}}\hat{S}^{-1},\hat J_\q\Big]
			\end{align*}
			and with Lemma \ref{lemma:GD} we obtain the stated result.
			\item As a compatibilitiy condition we can write
			\begin{align*}
			\frac{\partial^2 W_1}{\partial t_\q\partial t_{\q'}}=&\frac{\partial B_\q}{\partial t_{\q'}}W_1+B_\q\frac{\partial W_1}{\partial t_{\q'}}\\
			=&\Big(\frac{\partial B_\q}{\partial t_{\q'}}+B_\q B_{\q'}\Big)W_1\\
			=&\Big(\frac{\partial B_{\q'}}{\partial t_{\q}}+B_{\q'} B_{\q}\Big)W_1	\frac{\partial^2 W_1}{\partial t_{\q'}\partial t_{\q}}
			\end{align*}
			and therefore we deduce the zero-curvature equation. For the rest of the expressions we proceed similarly.
		\end{enumerate}
	\end{proof}
	
	\subsection{Proof of Proposition \ref{pro:kp-schrodinger}}\label{proof11}
	\begin{proof}
		In order to prove the first equation we notice that  
		\begin{align*}
		\partial_{(a,b)}W_1&=(\partial_{a,b}S+S\Upsilon_a\Upsilon_b)W_0,\\
		\partial_a\partial_bW_1&=(\partial_a\partial_bS+\partial_a S\Upsilon_b+\partial_bS\Upsilon_a
		+S\Upsilon_a\Upsilon_b)W_0
		\end{align*}
		and therefore
		$  (\partial_{(a,b)}-\partial_a\partial_b)W_1=-\big(\partial_a S(\Upsilon_b)_{>}+\partial_bS(\Upsilon_a)_{>}\big)W_0+\mathfrak{l}W_0$
		so that
		\begin{align*}
		\Big( \partial_{(a,b)}-\partial_a\partial_b+V_{a,b}+V_{b,a}\Big)W_1&\in\mathfrak{l}W_0.
		\end{align*}
		On the other hand, it is obvious that
		\begin{align*}
		\Big( \partial_{(a,b)}-\partial_a\partial_b+V_{a,b}+V_{b,a}\Big)W_2&\in\mathfrak{u}.
		\end{align*}
		Now, we apply Proposition \ref{pro:asymptotic-module} with
		\begin{align*}
		R_i&=  \Big( \partial_{(a,b)}-\partial_a\partial_b +V_{a,b}+V_{b,a}\Big)W_i, &i &=1,2,
		\end{align*}
		to get the first result. In order to prove the second result proceed in the exact same way starting with $\hat{W}_1$, deriving with 
		respect to $\bar{\partial}_{(-a,-b)}$ and $\bar{\partial}_{-a}$ and finally use the results in 
		Proposition \ref{pro:asymptotic-module} that involve $(\mathfrak{l}W_0)^{\dagger}$.
	\end{proof}

		\subsection{Proof of Proposition \ref{latriple}}\label{proof12}
		 \begin{proof}
	 	We begin with
	 	\begin{align*}
	 	( \partial_{(a,b,c)}-    \partial_a\partial_b\partial_c)(W_1)=&
	 	\big(\partial_{a,b,c}S-\partial_a\partial_b\partial_c S-\partial_a\partial_bS\Upsilon_c-
	 	\partial_b\partial_cS\Upsilon_a-\partial_c\partial_aS\Upsilon_b\\&-
	 	\partial_aS\Upsilon_b\Upsilon_c-\partial_bS\Upsilon_c\Upsilon_a-\partial_cS\Upsilon_a\Upsilon_b\big)W_0,
	 	\end{align*}
	 	and take into account 
	 	$S=\I+\beta^{(1)}+\beta^{(2)}+\cdots$, being $\beta^{(k)}$ the $k$-th subdiagonal of $S$, we can write
	 	%
	 	\begin{align*}
	 	&( \partial_{(a,b,c)}-\partial_a\partial_b\partial_c    )(W_1)=
	 	\big(-\partial_a\partial_b\beta^{(1)}(\Upsilon_c)_>-\partial_b\partial_c\beta^{(1)}(\Upsilon_a)_>-\partial_c\partial_a\beta^{(1)}(\Upsilon_b)_>\\
	 	&-(1-\delta_{0,b+c})\partial_a\beta^{(2)}(\Upsilon_b)_>(\Upsilon_c)_>
	 	-(1-\delta_{0,a+c})\partial_b\beta^{(2)}(\Upsilon_c)_>(\Upsilon_a)_>-
	 	(1-\delta_{0,a+b})\partial_c\beta^{(2)}(\Upsilon_a)_>(\Upsilon_b)_>\\
	 	&-(1-\delta_{0,b+c})\partial_a\beta^{(1)}(\Upsilon_b)_>(\Upsilon_c)_>
	 	-(1-\delta_{0,a+c})\partial_b\beta^{(1)}(\Upsilon_c)_>(\Upsilon_a)_>-
	 	(1-\delta_{0,a+b})\partial_c\beta^{(1)}(\Upsilon_a)_>(\Upsilon_b)_>
	 	\big)W_0+\mathfrak{l}W_0
	 	\end{align*}
	 	Now recall 
	 	\begin{align*}
	 	\partial_a W_1=\big((\Upsilon_a)_> +\beta (\Upsilon_a)_> \big) W_0+\mathfrak{l}W_0
	 	\end{align*}
	 	and use it in every term of the third line of the previous expression to obtain
	 	{\small\begin{gather*}
	 		( \partial_{(a,b,c)}-\partial_a\partial_b\partial_c    
	 		+(1-\delta_{0,b+c})\partial_a\beta^{(1)}(\Upsilon_b)_>\partial_c
	 		+(1-\delta_{0,a+c})\partial_b\beta^{(1)}(\Upsilon_c)_>\partial_a
	 		+(1-\delta_{0,a+b})\partial_c\beta^{(1)}(\Upsilon_a)_>\partial_b
	 		)(W_1)=\\
	 		\big(-\partial_a\partial_b\beta^{(1)}(\Upsilon_c)_>-\partial_b\partial_c\beta^{(1)}(\Upsilon_a)_>-\partial_c\partial_a\beta^{(1)}(\Upsilon_b)_>
	 		-(1-\delta_{0,b+c})\partial_a\beta^{(2)}(\Upsilon_b)_>(\Upsilon_c)_>
	 		\\   -(1-\delta_{0,a+c})\partial_b\beta^{(2)}(\Upsilon_c)_>(\Upsilon_a)_>-
	 		(1-\delta_{0,a+b})\partial_c\beta^{(2)}(\Upsilon_a)_>(\Upsilon_b)_>
	 		+(1-\delta_{0,b+c})\partial_a\beta^{(1)}(\Upsilon_b)_>\beta^{(1)}(\Upsilon_c)_>
	 		\\+(1-\delta_{0,a+c})\partial_b\beta^{(1)}(\Upsilon_c)_>\beta^{(1)}(\Upsilon_a)_>
	 		+(1-\delta_{0,a+b})\partial_c\beta^{(1)}(\Upsilon_a)_>\beta^{(1)}(\Upsilon_b)_>
	 		\big)W_0+\mathfrak{l}W_1.
	 		\end{gather*}  }
	 	From where the result follows.
	 \end{proof}
\end{appendices}

\end{document}